\documentclass[a4paper,11pt,twoside,reqno]{article}

\usepackage{amsmath, amsfonts, amssymb, amsthm, mathtools}
\usepackage{stmaryrd}
\usepackage{enumerate}
\usepackage{bbm}
\usepackage{cases}
\usepackage{hyperref}
\usepackage[T1]{fontenc}
\usepackage[latin1]{inputenc}
\usepackage{upgreek}
\usepackage[english]{babel}
\usepackage[cyr]{aeguill}
\usepackage{xspace}
\usepackage{graphicx}
\usepackage{caption}
\usepackage{subcaption}
\usepackage{color}
\usepackage{epigraph}
\usepackage{etoolbox}
\usepackage{enumitem}
\usepackage{stmaryrd}
\usepackage[margin=1in]{geometry}
\usepackage{cite}
\usepackage{dsfont}
\usepackage{ushort}
\usepackage{xcolor}
\usepackage{tikz}
\usetikzlibrary{arrows,backgrounds}
\usepgflibrary{shapes.multipart}
\usetikzlibrary{arrows.meta}
\usepackage[normalem]{ulem}

\numberwithin{equation}{section}
\newtheorem{thm}{Theorem}[section]
\newtheorem{lem}[thm]{Lemma}
\newtheorem{cor}[thm]{Corollary}
\newtheorem{prop}[thm]{Proposition}
\newtheorem{rem}[thm]{Remark}
\newtheorem{eg}[thm]{Example}
\newcommand{\x}{\times}
\newcommand{\pa}{\partial}

\newcommand{\E}{\mathbb{E}}

\newcommand{\N}{\mathbb{N}}
\newcommand{\R}{\mathbb{R}}

\newcommand{\g}{\gamma}
\newcommand{\si}{\sigma}

\newcommand{\Fc}{\mathcal{F}}
\newcommand{\Lc}{\mathcal{L}}
\newcommand{\Pc}{\mathcal{P}}
\newtheorem{assum}[thm]{Assumption}

\theoremstyle{definition}

\newcommand{\dbE}{\mathbb{E}}

\newcommand{\dbN}{\mathbb{N}}
\newcommand{\dbP}{\mathbb{P}}
\newcommand{\dbQ}{\mathbb{Q}}
\newcommand{\dbR}{\mathbb{R}}

\newcommand{\cC}{\mathcal{C}}

\newcommand{\cL}{\mathcal{L}}

\newcommand{\cN}{\mathcal{N}}

\newcommand{\cV}{\mathcal{V}}
\newcommand{\cW}{\mathcal{W}}

\def\a{\alpha}
\def\b{\beta}
\def\g{\gamma}
\def\d{\delta}
\def\e{\varepsilon}

\def\l{\lambda}

\def\si{\sigma}

\def\f{\varphi}
\def\th{\theta}
\def\o{\omega}

\def\D{\Delta}
\def\Si{\Sigma}

\newcommand{\fF}{\mathfrak{F}}

\def\eps{\varepsilon}

\def\cd{\cdot}
\def\cds{\cdots}

\def\rc{{\rm rc}}
\def\syc{{\rm sc}}

\def\ms{\vspace{3mm}}

\newcommand{\omb}{\bar \omega}
\newcommand{\Vt}{\widetilde V}
\newcommand{\Wb}{\overline W}
\newcommand{\Pb}{\overline \dbP}
\newcommand{\Qb}{\overline \dbQ}
\newcommand{\Pt}{\widetilde \dbP}
\newcommand{\Qt}{\widetilde \dbQ}
\newcommand{\Wt}{\widetilde W}
\newcommand{\bt}{\tilde b}
\newcommand{\Fb}{\overline {\mathbb{F}}}
\newcommand{\Fcb}{\overline {\mathcal{F}}}
\newcommand{\Om}{\Omega} 
\newcommand{\Omb}{\overline \Omega} 
\newcommand{\Xb}{\overline X}
 
\newcommand{\Vb}{\overline V}
\newcommand{\Xt}{\widetilde X}

\DeclareMathOperator*{\argmin}{\arg\!\min}
\DeclareMathOperator*{\argmax}{\arg\!\max}

\begin{document}

\title{Ergodicity of the underdamped mean-field Langevin dynamics}

\author{Anna Kazeykina
	   \footnote{D\'epartement de Math\'ematiques, Facult\'e des Sciences d'Orsay, Universit\'e Paris-Saclay, Orsay, France. anna.kazeykina@math.u-psud.fr.}
	   \and 
	   \footnote{CEREMADE, Universit\'e Paris-Dauphine, PSL, Paris, France. ren@ceremade.dauphine.fr.}
	   Zhenjie Ren
	   \and 
	   Xiaolu Tan
	   \footnote{Department of Mathematics, The Chinese University of Hong Kong, China. xiaolu.tan@cuhk.edu.hk.}
	   \and
	   Junjian YANG
	   \footnote{FAM, Fakult\"at f\"ur Mathematik und Geoinformation, Vienna University of Technology, A-1040 Vienna, Austria. junjian.yang@tuwien.ac.at}
       }

\date{\today}

\maketitle

\begin{abstract} 
 We study the long time behavior of an underdamped mean-field Langevin (MFL) equation,
 and provide a general convergence as well as an exponential convergence rate result under different conditions.
 The results on the MFL equation can be applied to study the convergence of the Hamiltonian gradient descent algorithm for the overparametrized optimization.
 We then provide some numerical examples of the algorithm to train a generative adversarial network (GAN).
\end{abstract}

\noindent
\textbf{MSC 2020 Subject Classification:} 37M25, 60H10, 60H30, 60J60. \newline
\vspace{-0.2cm}\newline
\noindent
\noindent
\textbf{Key words:} Underdamped mean-field Langevin dynamics, ergodicity, coupling, GAN.  


\section{Introduction}
	
	Let $\Pc_2(\R^n)$ denote the space of all probability measures on $\R^n$ with finite second order moment,
	and let $F:\Pc_2(\dbR^n)\longrightarrow \dbR_+$ be a potential function with intrinsic derivative denoted by $D_m F: \Pc_2(\R^n) \x \dbR^n \longrightarrow \dbR^n$ (see Section \ref{sec:pre} below for its definition).
We study in this paper the ergodicity of the following underdamped mean-field Langevin (MFL) equation:
	\begin{equation} \label{eq:MFL_intro}
		d X_t = V_t dt, \quad 
		d V_t = -\big(D_m F(\Lc(X_t), X_t )+  \g V_t \big) dt + \si dW_t,
	\end{equation}
	where $\gamma>0$, $\sigma\neq 0$ are two real constants, $\Lc(X_t)$ denotes the law of $X_t$, and  $W$ is an $n$-dimensional standard Brownian motion.
	Let us denote the marginal distribution by $m_t := \Lc(X_t, V_t)$,
	then it is easy to verify by the It\^o's formula that $m = (m_t)_{t \ge 0}$ is a solution (in the sense of distribution) to the nonlinear kinetic Vlasov-Fokker-Planck equation:
	\begin{equation} \label{eq:FP}
		\pa_t m = - v \cd \nabla_x m  + \nabla_v\cd \big((D_m F(m^X, x) +\g v) m\big)  + \frac{1}{2}  \sigma^2 \Delta_v m,
	\end{equation}
	where $m^X_t \in \Pc_2(\R^n)$ denotes the pushforward measure of $m_t \in \Pc_2(\R^n \x \R^n)$ under the map $(x,v) \longmapsto x$.

	\vspace{0.5em}

	This consists in an extension of the classical underdamped Langevin equation (without the mean-field term $\Lc(X_t)$ in \eqref{eq:MFL_intro}).
	More precisely, let $\phi: \R^n \longrightarrow \R_+$ be a differentiable potential function, and
	\begin{equation} \label{eq:F_phi_intro}
		F(m) = \int_{\R^n} \phi(x) m(dx),
		~\mbox{so that}~
		D_m F(m, x) = \nabla \phi(x),~\mbox{for all}~(m,x) \in \Pc_2(\R^n) \x \R^n.
	\end{equation}
	Then \eqref{eq:MFL_intro} turns to be the classical underdamped Langevin equation,
	which has been introduced in statistical physics to describe the motion of a particle with position $X$ and velocity $V$ in a potential field $\nabla_x \phi$ subject to damping and random collisions, see e.g.~\cite{Eins05, Lang08, Nel67}, and has been largely investigated in computational statistical physics, see e.g.~\cite{SS78, BBK1984}.
	It is well known that (see e.g.~\cite[Proposition 6.1]{Pav14}), under mild conditions, the classical Langevin dynamics (i.e.~Equation \eqref{eq:MFL_intro} with $F$ being given by \eqref{eq:F_phi_intro}) 
	has a (unique) invariant measure on $\dbR^n\times\dbR^n$ with the density:
	\begin{equation} \label{eq:inv_classic}
		m_\infty(x,v) = C e^{-\frac{2\gamma}{\si^2}\big( \phi(x) + \frac12 |v|^2\big)},
	\end{equation}
	for some normalization constant $C > 0$. 
	Moreover, in this setting, an important subject is to study the ergodicity of Langevin dynamic, i.e. the convergence as well as the convergence rate of the marginal distributions $m_t$ to the invariant measure $m_{\infty}$.

	\vspace{0.5em}

	To prove the convergence result, a classical approach is the free energy method, see e.g. Bonilla, Carrillo and Soler \cite{BCS97}.
	Let us define the free energy function along the marginals $(m_t)_{t \ge 0}$ of the Langevin dynamic by
	$$
		\Phi(t) := \int_{\R^n\times\R^n} \Big( \phi(x) + \frac12 |v|^2 \Big) m_t(dx, dv) +  \frac{\si^2}{2\gamma} H(m_t), ~t \ge 0,
	$$
	where $H: \Pc_2(\R^n \x \R^n) \longrightarrow \R \cup \{\infty\}$ is the entropy function (see Section \ref{sec:pre} below for a precise definition).
	By computing the derivative $\Phi'(t)$, one then proves the decay of the free energy, i.e. that $\Phi$ is a decreasing function.
	As the function $\Phi$ is bounded from below, it follows formally that $\Phi'(t) \longrightarrow 0$.
	Under further technical conditions, this can then be used to deduce a general convergence result $m_t \longrightarrow m_{\infty}$.

	\vspace{0.5em}

	As for the convergence rate, many works have been devoted to the subject and different approaches have been introduced.
	In \cite{Vil07, Vil09}, Villani popularized the concept ``hypocoercivity'' and proved the exponential convergence of $m_t$ in $H^1_{m_\infty}$. He generalized the idea already presented in the computations performed in Talay \cite[Section 3]{Talay02}.
	 A more direct approach was later developed in H\'erau \cite{Herau2006} and Dolbeault, Mouhot and Schmeiser \cite{DMS09, DMS15}, and it led to various results on kinetic equations. 
	 Notice that both Villani's and Dolbeault, Mouhot and Schmeiser's results on the exponential convergence rate highly depend on the dimension, and (therefore) do not apply to the case with mean-field interaction.
	It is noteworthy that in the recent paper by Cao, Lu and Wang \cite{CLW19}, the authors developed a new estimate on the convergence rate based on the variational method proposed by Armstrong and Mourrat \cite{AM19}.
	Let us also mention some work studying the ergodicity of underdamped Langevin dynamics using more probabilistic arguments, see e.g.~Ben Arous, Cranston and Kendall \cite{BCK95},  Wu \cite{Wu01}, Mattingly, Stuart, and Higham \cite{MSH02}, Rey-Bellet and Thomas \cite{RT02}, Talay \cite{Talay02}, Bakry, Cattiaux and Guillin \cite{BCG08}, Duong and Tugaut \cite{DT18}. These works are mostly based on Lyapunov conditions and the rates they obtain also depend on the dimension. 
	In the recent work by Guillin, Liu, Wu and Zhang \cite{GLWZ2021},  it was shown for the first time that a mean-field underdamped Langevin equation with non-convex potential is exponentially ergodic in $H^1_{m_\infty}$. The authors' argument combines Villani's hypocoercivity with a certain functional inequality and Lyapunov conditions.
	In  Monmarch\'e \cite{Mon17}, Guillin and Monmarch\'e \cite{GM2021}, the authors obtain an exponential convergence rate for the kinetic mean-field Langevin dynamics using a uniform propagation in chaos argument.  
To complete the brief literature review, we would draw special attention to the coupling argument applied in Bolley, Guillin and Malrieu \cite{BGM10} and  Eberle, Guillin and Zimmer \cite{EGZ19}, which found transparent convergence rates  in sense of Wasserstein-type distance.
	
	\vspace{0.5em}

	The main objective of this paper is to study the ergodicity of the underdamped Langevin dynamic in the mean-field setting of \eqref{eq:MFL_intro}.
	We will first develop the free energy approach in our setting to obtain a general convergence result of the marginal distribution $\Lc(X_t, V_t)$ to the invariant measure.
	Such an approach has already been initiated in Duong and Tugaut \cite{DT18} in a similar setting, but in a much more formal way.
	In a second part, we will apply the reflection-synchronous coupling technique that was initiated in Eberle, Guillin and Zimmer \cite{EGZ19, eberle2019quantitative} to obtain an exponential contraction result in some particular cases.

	\vspace{0.5em}

	For the first approach, we follows the main procedures as in Mei, Montanari and Nguyen \cite{mei2018mean}, Hu, Ren,  \v{S}i\v{s}ka and Szpruch \cite{HRSS19},  in particular the latter,  for the overdamped mean-field Langevin equation.
	First, we consider the follow new free energy function $\fF: \Pc_2(\R^n \x \R^n) \longrightarrow \R \cup \{\infty\}$
	\begin{equation} \label{intro:min}
		 \fF(m) 
		 ~:=~
		 F(m^X) +  \int_{\R^n \x \R^n} \frac{1}{2}|v|^2 m(dx, dv) +  \frac{\si^2}{2\gamma}  H(m),
	\end{equation}
	where $m^X \in \Pc_2(\R^n)$ denotes the pushforward measure of $m \in \Pc_2(\R^n \x \R^n)$ under the map $(x,v) \mapsto x$.
	We next deduce that the free energy is decreasing along the dynamics of the MFL,
	together with an explicit expression of the derivative $\frac{d \fF(m_t)}{dt}$.
	This is enough to show that, if $m^*$ is an accumulation points of $(m_t)_{t \ge 0}$ and it has a density function, then it satisfies
	\begin{align*}
		v + \frac{\si^2}{2}\nabla_v \log \big( m^*(x, v) \big) = 0.	
	\end{align*}
	Finally, by applying LaSalle's invariance principle for the dynamic system, we show that $m^*$ must satisfy the first order condition
	\begin{equation} \label{intro:FOC}
		D_m F(m^{*,X},x) + \frac{\si^2}{2 \gamma} \nabla_x \log \big( m^*(x, v) \big) = 0
		\quad \mbox{and}\quad
		v +  \frac{\si^2}{2 \gamma} \nabla_v \log \big( m^*(x, v) \big) = 0.
	\end{equation}
	When the potential function $F: \Pc_2(\R^n) \longrightarrow \R_+$ is convex so that $\fF: \Pc_2(\R^n \x \R^n) \longrightarrow \R \cup \{\infty\}$ is strictly convex, 
	the first order condition \eqref{intro:FOC} is sufficient to identify $m^*$ as the unique minimizer of $\fF$,
	which is also the unique invariant measure of \eqref{eq:MFL_intro}. 
	In this way, we are able to prove the uniqueness of the accumulation points of $(m_t)_{t \ge 0}$, which implies $m_t\longrightarrow m^*$.
	Due to the degeneracy and the mean-field interaction of the underdamped MFL process, the proof for the claim is non-trivial.
	In particular, we apply the time reversal technique for the SDE in F\"ollmer \cite{Follmer} to obtain some integrability properties of the marginal densities.
	
	\vspace{0.5em}

	For the second approach, we consider a potential functional $F: \Pc_2(\R^n) \longrightarrow \R$ which could be non-convex but is essentially with small (nonlinear) dependence on $m$, and aims at obtaining an exponential contraction result. 
	We mainly borrow the tools developed in  Eberle, Guillin and Zimmer \cite{EGZ19, eberle2019quantitative}, where   they initiate the reflection-synchronous coupling technique,  further validate it in the study of the  Langevin dynamics with a general non-convex potential, and make the point that the technique offers significant flexibility for additional development.
	But note that \cite{EGZ19} is not concerned with mean-field interaction and the rate found there is dimension dependent.
	In our context, we design a new metric involving a quadratic form (see Section \ref{sec:G}) to obtain the contraction when the coupled particles are far away, and as a result obtain a dimension-free convergence rate.
	The construction of the quadratic form shares some flavor with the argument in Bolley, Guillin and Malrieu \cite{BGM10}. 
	Notably, our construction helps to capture the optimal rate in the area of interest (see Remark \ref{rem:opt_rate}), so may be more intrinsic. 
	Notice that most of the articles concerning the ergodicity of underdamped Langevin dynamics obtain the convergence rates depending on the dimension, and in particular very few  allow both non-convex potential and the mean-field interaction. Some exceptions would be Guillin, Liu, Wu and Zhang \cite{GLWZ2021},  Monmarch\'e \cite{Mon17} and Guillin and Monmarch\'e \cite{GM2021}, but they focus on a particular convolution-type interaction. 
	When finishing our paper, we learned that independently Bolley, Guillin, Le Bris and Monmarch\'e \cite{BGBM20} are working out an exponential convergence result for the mean-field kinetic system through a similar approach.

\paragraph{More related works}

	Langevin dynamics have two specific limiting regimes: the Hamiltonian limit as $\gamma \to 0$, and the overdamped limit as $\gamma \to +\infty$ (in conjunction with a rescaling of time). There are various works in the literature which carefully study the scaling of the convergence rate in terms of the friction parameter $\gamma$, and obtain lower bounds $c \min(\gamma,\gamma^{-1})$, see for instance, Dolbeault, Klar, Mouhot, Schmeiser \cite{DKMS13}, Grothaus, Stilgenbauer \cite{GS16} and Iacobucci, Olla, Stoltz \cite{IOS19} for rather general potentials in addition to results for specific systems, see e.g.~Metafune, Pallara, Priola \cite{MPP2002} and Kozlov \cite{Kozlov89}. In our paper, as a corollary of the main convergence result, we shall prove that  the underdamped MFL dynamics also converges to its overdamped limit as $\g\rightarrow\infty$.

	Based on the ergodicity of the underdamped Langevin dynamic, and by considering  various discrete time versions, it has been developed the Hamiltonian Monte Carlo methods,
	where the objective is to sample according to the distributions in form of \eqref{eq:inv_classic}, see e.g.~Leli\`evre, Rousset and Stoltz \cite{LRS10}, Neal \cite{Neal11}, Bou-Rabee, Eberle and Zimmer \cite{BEZ20}, Bou-Rabee and Schuh \cite{BS2020}. 
	Nowadays this interest resurges in the community of machine learning. Notably, the underdamped Langevin dynamics has been empirically observed to converge more quickly to the invariant measure compared to the overdamped Langevin dynamics (of which the related MCMC was studied in e.g.~Dalalyan \cite{Dala17}, Durmus and Moulines \cite{DM16}), and it was theoretically justified by Cheng, Chatterji, Bartlett and Jordan in \cite{CCBJ18} for some particular choice of coefficients.

	\vspace{0.5em}
	
	Ergodicity of the underdamped MFL dynamics can also be used to solve optimization problems.
	As we shall see in Remark \ref{rem:entropyrewrite}, the optimization problem 
	$$
		\inf_{m\in \Pc_2(\R^n \x \R^n)} \fF(m) 
		~\mbox{is a regularized version of problem}~
		\inf_{m^X \in \Pc_2(\R^n)} F(m^X),
	$$
	while the latter appears naturally in the context of neural network with plenty of neurons (see e.g.~\cite{HRSS19} as well as Section \ref{sec:GAN}). 
	By identifying the minimizer $m^*= \argmin_{m\in \Pc_2(\R^n \x \R^n)} \fF(m)$ as  the limit of the marginal distributions of   the underdamped MFL dynamics, we justify the underdamped MFL dynamics as an efficient numerical algorithm to solve the mean-field optimization problem.

	\vspace{0.5em}

	The rest of the paper is organized as follows. In Section \ref{sec:mainresult} we announce the main results. Before entering the detailed proofs, we study  a numerical example concerning the so-called generative adversarial networks (GAN). The main theorems in Section \ref{sec:mainresult} guide us to propose a theoretical convergent algorithm for the GAN, and the numerical test in Section \ref{sec:GAN} shows a satisfactory result. Finally, we report the proofs in Section \ref{sec:proof}.

\section{Ergodicity of the mean-field Langevin dynamics} \label{sec:mainresult}

\subsection{Preliminaries} \label{sec:pre}

	Let us denote by $\Pc(\R^n)$ the space of all probability measures on $\dbR^n$, 
	and by $\Pc_p(\R^n)$ the space of all $m \in \Pc(\R^n)$ with finite $p$-th moment, for all $p \ge 1$. 
	Without further specification, in this paper the continuity on $\Pc_p(\R^n)$ is in the sense of $\cW_p$ ($p$-Wasserstein) distance, 
	i.e.
	$$
		\cW_p (\mu, \nu) ~:=~ \bigg( \inf_{\pi \in \Pi(\mu, \nu)} \int_{\R^n \x \R^n} |x-y|^p \pi (dx, dy) \bigg)^{1/p},
	$$
	where $\Pi(\mu, \nu)$ denotes the collection of all joint distribution on $\R^n \x \R^n$ with marginal distribution $\mu$ and $\nu$ on the first and second marginal space $\R^n$.
	The spaces $\Pc(\R^n \x \R^n)$, and $\Pc_p(\R^n \x \R^n)$ as well as the corresponding $p$-Wasserstein distance $\cW_p$ are defined similarly.

	\vspace{0.5em}

	A function $F:\Pc_2 (\R^n) \longrightarrow\dbR$ is said to be convex if
	\begin{equation} \label{eq:def_convex}
		F \big(\lambda m + (1-\lambda) m' \big) \le \lambda F(m) + (1-\lambda) F(m'),
		~~\mbox{for all}~ \lambda \in [0,1], ~m, m' \in \Pc_2(\R^n).
	\end{equation}
	We say $F$ is strictly convex if the above inequality is strict whenever $m \neq m'$ and $\lambda \in (0,1)$.
	Next, we say $F  \in {\cC^1}$,
	if there exists a continuous function $\frac{\d F}{\d m}:\Pc_2(\R^n) \times \dbR^n \longrightarrow \dbR$ satisfying $|\frac{\d F}{\d m}(m,x) | \le C(1 + |x|^2)$ for some constant $C> 0$, and such that, for all $m,m'\in\mathcal{P}(\dbR^n)$, one has
	\begin{align*}
		F(m') - F(m) = \int_0^1 \int_{\dbR^n} \frac{\d F}{\d m}\big( (1-u)m+um', x\big)~ (m'-m)(dx) du.
	\end{align*}
	When $\frac{\d F}{\d m}(m,x)$ is continuously differentiable in $x$, we define $D_m F: \Pc_2(\R^n) \x \R^n \longrightarrow \R^n$ by
	$$
		D_m F(m,x) ~:=~ \nabla_x \frac{\d F}{\d m}(m,x),
	$$
	which is the so-called intrinsic derivative of $F$ introduced by Lions \cite{LionsCours}.
	We also refer to \cite[Proposition 5.1.4, Proposition 5.1.5]{Card2018} and \cite[Section 2.2]{CDLLMaster2019} for more properties and interpretations of this notion of derivative.
	We say function $F\in \cC^\infty$ if, for all $k \in \N$, $i_1, \cdots, i_k \in \N$, the derivatives
	\begin{align*}
		\partial_{x_1}^{i_1}\cdots\partial_{x_k}^{i_k} D^k_m F(m,x_1, \cds, x_k)
		~\mbox{exist and are continuous.}
	\end{align*}

		\begin{eg}\label{eg:deriveF}
		Let us provide some simple examples of $F:\Pc_2(\R^n) \longrightarrow\dbR$ with its derivatives.
		\begin{enumerate}
		\item In case that $F$ is linear, namely, $F(m) = \int \phi(x) m(dx)$, where
		$\phi: \R^n \longrightarrow \R$ is continuously differentiable and has quadratic growth.
		Then 
		$$
			\frac{\d F}{\d m}(m,x) = \phi(x),
			~\mbox{and hence}~ 
			D_m F(m,x) = \nabla_x \phi(x).
		$$
		
		\item   In case that $F(m) = \int\int \phi(x, y) m(dx)m(dy)$. It is easy to check that 
		    $$ \frac{\d F}{\d m}(m,x) = \int \big(\phi(x,y)+\phi(y,x)\big) m(dy). $$
		
		\item In case that $F(m) = g\big(\int \phi(x) m(dx)\big)$ with with $g, \phi$ continuously differentiable, we may apply the chain rule to obtain that 
		$$
			\frac{\d F}{\d m}(m,x) =  g' \Big(\int \phi(y) m(dy)\Big) \phi(x)
			~\mbox{and thus}~
			D_m F(m,x) =  g' \Big(\int \phi(y) m(dy)\Big) \nabla_x\phi(x).
		$$
		\end{enumerate}
	\end{eg}

	Recall that, for $m\in \Pc_2(\R^n \x \R^n)$, we denote by $m^X \in \Pc_2(\R^n)$ the pushforward measure of $m$ under the map $(x,v) \mapsto x$.
	Denote by $H(m)$ the relative entropy of the measure $m\in \Pc_2(\R^n \x \R^n)$ with respect to the Lebesgue measure,
	that is,
	\begin{align*}
		H(m) 
		~:=~
		\E^m \big[  \log \big( m(X, V) \big) \big] 
		~=~  \int_{\R^n \x \R^n} \log \big( m(x, v) \big) m(x,v) dx dv.
	\end{align*}
	Notice that $H(m)$ is well defined for all $m \in \Pc_2(\R^n \x \R^n)$.
	Indeed, let $C > 0$ be the renormalization constant such that $C e^{-|x|^2- |v|^2}$ is a density function for a probability measure $\nu$,
	then
	\begin{align*}
		&\int_{\R^{2n}} \log \big( m(x, v) \big) m(x,v) dx dv \\
		=&
		\int_{\R^{2n}} \log \Big( \frac{m(x,v)}{C e^{-|x|^2 - |v|^2}}\Big) m(x,v) dx dv - \int_{\R^{2n}} \big( |x|^2 + |v|^2 \big) m(x,v) dx dv + \log(C)
	\end{align*}
	is well defined as the first term at the r.h.s. is the relative entropy between $m$ and $\nu$.
	In above, we use $m(x,v)$ to denote the density function of the probability measure $m \in \Pc_2(\R^n \x \R^n)$ by abuse of notation, and 
	let  $H(m) := \infty$ if $m$ does not has a density.
	In particular we recall that $H: \Pc_2(\R^n \x \R^n) \longrightarrow \R\cup\{ \infty\}$ is strictly convex, see e.g.~\cite[Lemma 1.4.3]{DE97}.

\subsection{The mean-field Langevin dynamics and free energy function}

	Throughout the paper we consider a potential function $F: \Pc_2(\R^n) \longrightarrow \R_+ \cup \{\infty\}$ in the form of
	\begin{align}\label{eq:potential_decomp}
		F(m^X)  ~=~ F_{\circ} (m^X) +   \int_{\R^n} f(x) m^X(dx),
	\end{align}
	where $F_\circ:\Pc_2(\R^n)\longrightarrow\R_+$ and $f:\dbR^n \longrightarrow\dbR_+$ satisfy $F_\circ \in \cC^1$ and $f \in C^1$.
	We slightly abuse the notation and denote
	\[ 
		\frac{\d F}{\d m }(m^X, x) := \frac{\d F_\circ}{\d m }(m^X, x) +f(x),
		\quad\mbox{and}\quad
		D_m F(m^X, x) := D_m F_\circ(m^X, x) + \nabla f(x).
	\] 
	Fixing a positive constant $\gamma > 0$, a constant $\sigma \neq 0$, we introduce the underdamped mean-field Langevin (MFL) dynamics:
	\begin{equation} \label{eq:MFL}
		d X_t = V_t dt, \quad 
		d V_t = -\big(D_m F(\Lc(X_t), X_t )+  \g V_t \big) dt + \si dW_t.
	\end{equation}

	We next consider the following free energy function minimization problem:
	\begin{align} \label{eq:minization}
		\inf_{m\in \Pc_2(\R^n \x \R^n)} \fF(m),
		   \quad\mbox{with}\quad		
		 \fF (m) := F(m^X) +  \int_{\R^n \x \R^n} \frac{|v|^2}{2} m(dx, dv) + \frac{\si^2}{2\g} H(m),
	\end{align}
	where $\fF: \Pc_2(\R^n \x \R^n) \longrightarrow \R \cup\{\infty\}$ is the so-called free energy function. Throughout the paper we also assume that $\fF$ is nontrivial, that is, there exists $m_\circ \in \Pc_2(\R^n \x \R^n)$ such that $\fF(m_\circ)$ is finite.
	The above free energy function is closely related to the Langevin dynamics \eqref{eq:MFL}.
	To see that, let us recall the first order condition for such an optimization problem from Hu, Ren, \v{S}i\v{s}ka and Szpruch \cite[Proposition 2.5]{HRSS19}. 

	\begin{lem}\label{lem:FOC}
		Let $F: \Pc_2(\R^n) \longrightarrow \R_+$ be a potential function in the form of \eqref{eq:potential_decomp}.
		If $m \in\argmin_{\mu \in \Pc_2(\R^n \x \R^n)} \fF(\mu)$,
		then $m$ has a density and there exists a constant $C \in \R$ such that
		\begin{align} \label{eq:FOCeq}
			\frac{\d F}{\d m }(m^X, x) + \frac{|v|^2}{2} + \frac{\si^2}{2\g}  \log \big( m(x, v) \big) = C,
			~~\mbox{for all $(x, v)\in \dbR^{n} \x \R^n$,}
		\end{align}
		or equivalently, for some $C' > 0$,
		\begin{align} \label{eq:FOCdensity}
			m(x, v) = C' \exp\bigg( - \frac{2\g}{\si^2} \Big(\frac{\d F}{\d m }(m^X, x) + \frac{|v|^2}{2} \Big)\bigg),
			~~\mbox{for all}~(x,v) \in \R^n \x \R^n.
		\end{align}
		Conversely, if $m \in \Pc_2(\R^n \x \R^n)$ has the density function given by \eqref{eq:FOCdensity}, and in addition $F$ is convex, then $m$ is a solution to \eqref{eq:minization}.
	\end{lem}

	\begin{rem}
		By direct computation, when $m$ in \eqref{eq:FOCdensity} is smooth enough, one can check directly that $m$ is a stationary solution to the Fokker-Planck equation \eqref{eq:FP} in the sense that
		$$
			 - v \cd \nabla_x m  + \nabla_v\cd \big((D_m F(m^X, x) +\g v) m\big)  + \frac{1}{2}  \sigma^2 \Delta_v m = 0.
		$$
		Besides, the marginal distribution of the underdamped MFL dynamics \eqref{eq:MFL} is a weak solution to the Fokker-Planck equation \eqref{eq:FP}.
		If \eqref{eq:FP} has a unique weak solution (with a given initial condition), then $m$ is an invariant measure to the dynamics \eqref{eq:MFL}.
		In fact, we will prove in our context that $m$ is indeed the unique invariant (probability) measure of the underdamped MFL dynamics \eqref{eq:MFL}.
	\end{rem}

	\begin{rem}\label{rem:entropyrewrite}
	$\mathrm{(i)}$ With the decomposition $F(m^X)  ~=~ F_{\circ} (m^X) +  \int_{\R^n} f(x) m^X(dx)$, 
	it is equivalent to define the free energy functional by
	\begin{align*}
		\fF(m) ~=~ F_\circ(m^X ) + \frac{\si^2}{2\g} H(m|\mu),
	\end{align*}
	where $\mu$ is the probability measure with density $C e^{-\frac{2\g}{\si^2}(f(x)+\frac12|v|^2)}$for some normalization constant $C> 0$, and $H(m|\mu)$ is the relative entropy of $m$ with respect to $\mu$.
	This is in fact the formulation of the free energy function in Hu, Ren, \v{S}i\v{s}ka and Szpruch \cite{HRSS19}. 
	Besides, it is assumed in \cite[Proposition 2.5]{HRSS19} that $F_\circ$ is convex and that $f(x) \ge \lambda |x|^2$ for some $\lambda > 0$.
	However,  the convexity of $F_{\circ}$ is only used to prove that \eqref{eq:FOCeq} or \eqref{eq:FOCdensity} is the sufficient condition for the optimality of $m$ in their proof.
	The growth condition on $f$ is only used to ensure the existence of a minimizer of $\fF(\mu)$, which is not stated in the above lemma.
	
	\vspace{0.5em}
	
	\noindent $\mathrm{(ii)}$ Notice that the entropy $H$ is strictly convex, 
	so that the free energy function $\fF$ is also strictly convex whenever $F$ is convex.
	Consequently, the optimization problem \eqref{eq:minization} has at most one minimizer, which must be the unique solution to  the first order equation \eqref{eq:FOCeq}. 

	\vspace{0.5em}
	
	\noindent 	
	$\mathrm{(iii)}$ In particular, given a probability measure $m$ satisfying \eqref{eq:FOCdensity}, one has
		$$
			\frac{\d F}{\d m }( m^X, x)   + \frac{\si^2}{2\g}  \log \big( m^X(x) \big) = C.
		$$
		For convex $F$, this implies (again thanks to  \cite[Proposition 2.5]{HRSS19}) that the marginal distribution $m^X$ is the minimizer of $m^X \longmapsto F(m^X) + \frac{\si^2}{2\g} H(m^X)$. Moreover, it follows from  \cite[Theorem 2.11]{HRSS19} that such $m^X$ is also the unique invariant measure of the overdamped MFL dynamics
		$$
			dX_t = -D_m F(\cL(X_t), X_t) dt + \frac{\si}{\sqrt \g} dW_t.
		$$
	\end{rem}

\begin{rem}
	To intuitively understand how the linear derivative $\frac{\d F}{\d m}$ characterizes the minimizer as in the first order condition \eqref{eq:FOCeq}, we may first ignore the terms $\frac12 \dbE^m\big[|V|^2\big]$ and $\frac{\si^2}{2\g} H(m)$ in the free energy $\fF (m)$ and consider a convex potential function $F:\Pc_2(\dbR^n)\longrightarrow\dbR$.
	Let $m \in \Pc_2(\R^n)$ be such that $\frac{\d F}{\d m}\big( m,x\big) = C$ for all $x \in \R^n$, 
	$m' \in \Pc_2(\R^n)$ be arbitrary and $m^\e := (1-\e) m +\e m'$ for $\e \in [0,1]$.
	Then, by the convexity of $F$, one has
	\begin{align*}
		F(m') - F(m) 
		&\ge  \frac{1}{\e} \Big( F(m^\e)-F(m) \Big) \\
		&=  \frac{1}{\e}  \int_0^\e \int_{\dbR^n} \frac{\d F}{\d m}\big( (1-u)m+um',x\big)~ (m'-m)(dx) du\\
		&\longrightarrow \int_{\dbR^n}\frac{\d F}{\d m}\big( m,x\big)(m'-m)(dx) ~=~ 0, \quad \mbox{as $\e\to 0$}.
	\end{align*}
	Therefore, for arbitrary $m' \in \Pc_2(\R^n)$, one has
	$$
		F(m') - F(m) \ge 0.
	$$
	In other words, when $F$ is convex and smooth enough, the condition $\frac{\d F}{\d m}\big( m,x\big) = C$ for all $x \in \R^n$ is sufficient  for $m$ being a minimizer of $F$.
\end{rem}

\subsection{Decay of the free energy and ergodicity of the Langevin dynamics}

	We will provide a first ergodicity result of the MFL dynamic \eqref{eq:MFL} based on the free energy approach.
	Let us first formulate some technical conditions.

	\begin{assum} \hspace{1mm} \label{assum:conv_initial}
	$\mathrm{(i)}$ The potential function $F: \Pc_2(\R^n) \longrightarrow \R_+$ is  in the form of \eqref{eq:potential_decomp},
		where $ F_\circ:\Pc_2(\R^n)\longrightarrow\R_+$ and $  f:\dbR^n \longrightarrow\dbR_+$ satisfy $F_\circ \in \cC^\infty$, $f \in C^\infty$ and
		\begin{align} \label{eq:flowerbound}
		f(x) \ge \l |x|^2, 
		~~\mbox{for all}~ x \in \R^n.
		\end{align}
		Moreover, for each $k \ge 2$, the derivatives $D^k_m F$  is bounded,
		and $D_m F(m^X, x)$ is Lipschitz, i.e.
		\begin{align*}
			\big| D_m F\big(m^X_1, x_1\big) -D_m F\big(m^X_2, x_2\big) \big| 
			~\le~
			C\Big(\cW_1\big(m^X_1, m^X_2\big) + |x_1 - x_2| \Big).
		\end{align*}
	\noindent $\mathrm{(ii)}$ For all $p \ge 1$, one has $\dbE[|X_0|^p+|V_0|^p]<\infty$, as well as $H(m_0) < \infty $.
\end{assum}

\begin{eg}
	Take the third example in Example \ref{eg:deriveF}, that is,  $F_\circ(m^X) := g\big(\int \phi(x) m^X(dx)\big)$
	with $g: \R \longrightarrow \R$ and $\phi: \R^n \longrightarrow \R$. 
	The intrinsic derivative reads 
	$$D_m F_\circ(m^X,x) ~=~  g' \bigg( \int \phi(y) m^X(dy) \bigg) \nabla_x\phi(x). $$ 	
	Then Assumption \ref{assum:conv_initial} holds true, provided that $f(x):= \l |x|^2$ and $g' \in C^{\infty}_b(\R)$, $\phi\in C^\infty_b(\R^n)$. 
	In particular, it covers the case where $g$ is the identify function, so that
	$$
		F_\circ(m^X) := \int \phi(x) m^X(dx)
		~~\mbox{and}~~
		D_m F_\circ(m^X, x) = \nabla_x \phi(x).
	$$
	as discussed in Example \ref{eg:deriveF}.
\end{eg}

\begin{rem}[Lower bound of the free energy function]\label{rem:lowerbound}
	Let $\mu \in \Pc_2(\R^n \x \R^n)$ denote the probability measure with the density $Ce^{-\frac{2\g}{\si^2}(|x|+|v|)}$,
	where $C > 0$ is the normalization constant. Then the free energy function can be rewritten as
	\begin{align*}
		\fF(m) = F_\circ(m^X) + \int \left(f(x) - |x| + \frac12 |v|^2 -|v|\right) m(dx,dv) + \frac{\si^2}{2\g} H(m|\mu) +  \frac{\si^2}{2\g} \log C.
	\end{align*}
	Note that  
	\begin{align*}
		f(x) - |x| + \frac12 |v|^2 -|v|+  \frac{\si^2}{2\g} \log C ~\ge~ \frac{\l}{2}|x|^2 +\frac14 |v|^2 -\widetilde C,
	\end{align*}
	where 
	$$
		\widetilde C ~:=~ \frac{1}{2 \lambda} + 1 +  \frac{\si^2}{2\g} \big| \log C \big|.
	$$
	Further, since $F_\circ(m^X) \ge 0$ and $H(m|\mu)\ge 0$,
	we obtain a lower bound for the energy function:
	\begin{align}\label{eq:lowerbound}
		\fF(m)  \ge  \int \left( \frac{\l}{2}|x|^2 +\frac14 |v|^2 \right) m(dx, dv)-\widetilde C.
	\end{align}
\end{rem}

	Under Assumption \ref{assum:conv_initial}, it is well known that the MFL equation \eqref{eq:MFL} admits a unique strong solution $(X_t, V_t)_{t \ge 0}$, see e.g.~the proof of Sznitman \cite[Theorem 1.1]{Sznit89} or Carmona \cite[Theorem 1.7]{Carmona2016Book}. 
	We first prove that the function $\fF$ defined in \eqref{eq:minization} decays along the marginal $m_t := \Lc(X_t, V_t)$ of the MFL dynamics \eqref{eq:MFL} $(X_t, V_t)_{t \ge 0}$.

\begin{thm}\label{thm:Lyapunov}
	Let Assumption \ref{assum:conv_initial} hold true. 
	Then, for all $t>0$, $m_t$ has a smooth and strictly positive density function, which is again denoted by $m_t(\cdot)$ by abuse of notation.
	Moreover, for all $t > s>0$, one has
	\begin{align*}
		\fF(m_t) - \fF(m_s) = - \int_s^t \g \dbE\bigg[\Big| V_r + \frac{\si^2}{2\g} \nabla_v \log \big( m_r (X_r, V_r) \big) \Big|^2 \bigg] dr.
	\end{align*}
\end{thm}

\begin{rem}
The time derivative of the free energy function has been calculated in the context of the classical underdamped Langevin dynamics in Bonilla et al.~\cite{BCS97}. Also, an informal analytical computation has been developed for some mean-field potentials in the paper of Dong and Tugaut \cite{DT18}.
\end{rem}

With the help of the free energy function $\fF$, we may prove the convergence of the marginal laws of \eqref{eq:MFL} towards the minimizer $\underline m:=\argmin_{m\in \Pc_2(\R^n \x \R^n)} \fF(m)$, provided that the function $F$ is convex.

\begin{thm}\label{thm:ergodicity}
	Let Assumption \ref{assum:conv_initial} hold true.
	Suppose in addition that the function $F$ (or equivalently $F_{\circ}$) is convex 
	(see \eqref{eq:def_convex}).
	Then the underdamped MFL dynamics \eqref{eq:MFL} has a unique invariant measure, which is also the unique minimizer $\underline m$ of \eqref{eq:minization}.
	Moreover, one has
	\begin{align*}
		\lim_{t\to\infty} \cW_1(m_t, \underline m)= 0.
	\end{align*}
\end{thm}

\begin{rem}
The ergodicity of diffusions with mean-field interaction is a long-standing problem. 
Theorem \ref{thm:ergodicity} shows that,
	for non-degenerate confinement potentials in the sense of \eqref{eq:flowerbound}, $F$ being convex on the space of probability measures (with 2nd order moment) is sufficient for the underdamped MFL dynamics to be ergodic. 
It is an analogue of Theorem 2.11 in Hu, Ren, \v Si\v ska and Szpruch \cite{HRSS19}, where it has been proved that the convexity of the potential function ensures the ergodicity of the overdamped MFL dynamics. 
\end{rem}

	We finally provide an analogue of the classical convergence result of the underdamped Langevin dynamic to the overdamped Langevin dynamic when $\gamma \longrightarrow \infty$.
	To this end, we set the scaling $\sigma = \sigma_0 \sqrt{\gamma}$ for some fixed constant $\sigma_0 > 0$,
	and denote by $(X^{\gamma}, V^{\gamma})$ the solution of the underdamped MFL dynamics \eqref{eq:MFL},
	with the same initial distribution, i.e. $\Lc\big(X^{\gamma}_0, V^{\gamma}_0\big) = m_0$.

\begin{cor}\label{cor:underover}
	Let Assumption \ref{assum:conv_initial} hold true, and suppose that the common initial distribution $m_0 = \Lc(X^{\gamma}_0, V^{\gamma}_0)$ (for all $\gamma > 0$) satisfies $\fF(m_0) < \infty$.
	Then, for all $t \ge 0$, one has $X^{\gamma}_{\gamma t} \longrightarrow Y_t$ in distribution as $\gamma \longrightarrow \infty$,
	where $Y$ is the overdamped MFL dynamic defined by
	\begin{equation} \label{eq:SDE_Y}
		dY_t = -D_m F\big(\Lc(Y_t),  Y_t\big) dt + \sigma_0 dW_t,
		~~\mbox{for all}~t > 0,
		~~\mbox{and}~~\Lc(Y_0) =  m_0.
	\end{equation}
\end{cor}

\subsection{Exponential ergodicity given small mean-field dependence}

	We now study the MFL dynamics \eqref{eq:MFL} under another set of technical conditions, where in particular $F$ is possibly non-convex but with small mean-field dependence. We shall obtain an exponential convergence rate if the invariant measure exists.

\begin{assum} \label{assum:exp}
	The potential function $F: \Pc_2(\R^n) \longrightarrow \R$ is given by
	\begin{align*}
		F(m^X)  ~=~ F_{\circ} (m^X) +   \frac{\lambda}{2} \int_{\R^n}  |x|^2 m^X(dx),
	\end{align*}
	where $F_{\circ}: \Pc_2(\R^n) \longrightarrow \R$ belongs to ${\cC^1}$ and $D_m F_{\circ}$ exists and is Lipschitz continuous.
	Moreover, for any $\e>0$, there exists $K>0$ such that for all $(m^X_1,x_1), (m^X_2, x_2)  \in \Pc_2(\R^n) \x \R^n$, one has
	\begin{equation} \label{eq:Cond_DmF}
		\big|D_m F_{\circ}(m^X_1,x_1) -D_m F_{\circ}(m^X_2,x_2) \big| 
		 \le
		\e |x_1-x_2|, \quad \mbox{whenever ~$|x_1-x_2|\ge K$}.
	\end{equation}
\end{assum}

\begin{rem}
	 Under Assumption \ref{assum:exp}, one has
		$$
			D_m F(m^X, x) = D_m F_{\circ}(m^X,x) ~+~ \lambda x.
		$$
		We highlight that the condition \eqref{eq:Cond_DmF} involves  two (possibly different) arguments $m^X_1, m^X_2 \in \Pc_2(\R^n)$.
		On the other hand,  \eqref{eq:Cond_DmF} only need to hold true for $|x_1-x_2|$ big enough, in particular, $D_m F_{\circ}$ is not necessarily a small perturbation of the linear term $\lambda x$, see Example \ref{eg:smallmf} below. 
\end{rem}

	Define the function $\psi: \R^n \x \R^n \longrightarrow \R_+$ by
	\begin{align} \label{eq:funpsi}
		\psi(z,w):= \big(1+\beta G(z, w+\gamma z)\big)h\big(\eta|w+\gamma z|+|z|\big),
	\end{align}
	where the positive constants $\b, \eta$, the non-negative quadratic form $G$ and the  
	non-decreasing concave function $h:\dbR_+\longrightarrow\dbR_+$ will be determined later, see \eqref{eq:defG} and \eqref{eq:def-h}. 
	Given $(x, v), (x',v')\in \dbR^{n} \x \R^n$, we denote
	\begin{align*}
		p:= v-v' + \g (x -x'),\quad r := |x-x'|, \quad u:= |p|, 
	\end{align*}
	and therefore
	\begin{align*}
		\psi(x-x', v-v') = \big(1+\b G( x-x' , p) \big) h(\eta u +r).
	\end{align*}
	Notice that $\psi(x-x', v-v')$ is a semi-metric between $(x,v)$ and $(x',v')$, we then define the semi-metric (see also discussions in Remark \ref{rem:W_psi} below):
	\begin{align*}
		\cW_\psi(m,m') = \inf\left\{\int \psi(x-x', v-v') d\pi (x,v,x',v'):  \pi~\mbox{coupling of $m,m'\in \Pc_2(\R^n \x \R^n)$}\right\}.
	\end{align*}
	Notice that there exists a constant $C>0$ such that $\psi(x-x', v-v') \ge C|(x-x',v-v')|$. Therefore, the semi-metric dominates the Wasserstein-$1$ distance in the sense that
	\begin{align}\label{WpsiW1}
	\cW_\psi(m,m') \ge C\cW_1(m,m'),\quad \mbox{for $m,m'\in \Pc_2(\R^n \x \R^n)$}.
	\end{align}

	\begin{thm}[Exponential convergence under small mean-field condition] \label{thm:Contraction}
	Let Assumption \ref{assum:exp} hold true. 
	Assume in addition that, for all $m^X_1 ,m^X_2 \in\Pc_2(\R^n)$ and $x \in \R^n$, one has
	\begin{align*}
	\big|D_m F_{\circ}\big(m^X_1, x\big) - D_m F_{\circ}\big(m^X_2,x\big)\big| \le \iota \cW_1\big(m^X_1, m^X_2\big).
	\end{align*}
	Then, for $\iota>0$ small enough (satisfying the quantitative condition \eqref{eq:smalliota} below), we have
	\begin{align*}
          \cW_\psi(m_t, m'_t) \leq e^{-ct} \cW_\psi(m_0, m'_0),
	\end{align*}
	where $c > 0$ is  a constant defined below in \eqref{eq:ciota} (see also Remark \ref{rem:convergence rate} for a lower bound of the constant $c> 0$ in a specific example). 
	In particular, the rate $c > 0$ does not depend on the dimension $n$, and for some $C > 0$, it satisfies that
	\[c \le C\overline\g, \quad \mbox{where} \quad\overline \g : = 
	\begin{cases}
	\gamma - \sqrt{\gamma^2-4\lambda},   & \mbox{if }\gamma^2>4\lambda, \\
\gamma,  & \mbox{if }\gamma^2<4\lambda.
	\end{cases} \]
	\end{thm}

	\begin{eg} \hspace{1mm} \label{eg:smallmf}
		$\mathrm{(i)}$ Let $F_{\circ}(m^X) := \int g(x) m^X(dx)$ for some smooth function $g: \R^n \longrightarrow \R$ such that 
			$$ \lim_{|x| \to \infty} \big| \nabla^2 g(x) \big| = 0. $$
			Then, $D_m F_{\circ}(m^X,x) = \nabla g(x)$ satisfies  \eqref{eq:Cond_DmF} and it has no mean-field dependence. 
		
		\vspace{0.5em}
		
		\noindent $\mathrm{(ii)}$ Let 
		     $$ F_\circ(m^X) := g\bigg(\int \phi(x) m^X(dx)\bigg) $$ 
		     with $g, \phi$ continuously differentiable, so that 
			 $$ D_m F_\circ(m^X,x) =  g' \left(\int \phi(y) m^X(dy)\right) \nabla_x\phi(x). $$ 
			Assume that $g \ge 0$, and $g', \nabla_x \phi$ are bounded, then $D_m F_\circ$ is bounded and hence it satisfies \eqref{eq:Cond_DmF}.
			In addition, if $g'$ is $L$-Lipschitz, then
			\[|D_m F_{\circ}(m^X_1, x) - D_m F_{\circ}(m^X_2,x)| \le L  \|\nabla_x\phi\|_\infty \cW_1(m^X_1, m^X_2). \]
			It follows that the conditions of Theorem \ref{thm:Contraction} are satisfied when $L$ is small enough.
	\end{eg}

\begin{rem}[Small mean-field condition]
As we shall see in the proof, it is important to first understand the exponential convergence for the Markov diffusion (without mean-field dependence). Then, the small mean-field dependence, viewed as a small perturbation from the Markov diffusion, is a ``sufficient'' condition for inheriting the desired exponential convergence. On the other hand, from the following example, we shall see that sometimes    the small mean-field dependence is also necessary for the convergence result. 
Consider the potential function 
	$$
		F(m^X) 
		 :=
		\frac12 \left( \int_{\R^n} |x|^2 m^X(dx) - \a \Big( \int_{\R^n} |x| m^X(dx) \Big)^2 \right).
	$$
Notice that $F$ is non-convex (indeed concave). 
The corresponding underdamped MFL dynamics reads:
	\begin{align*}
	\begin{cases}
		dX_t =V_t dt \\
		dV_t = - (X_t - \a \dbE[X_t] +\g V_t) dt + \si dW_t.
	\end{cases}
	\end{align*}
It is not hard to show that if $\a >1$, then $\dbE[X_t]$ diverges. In other words, in this case one cannot expect the convergence of marginal distribution when the mean-field dependence is large. 
\end{rem}

\begin{rem}[Semi-metric $\cW_\psi$] \label{rem:W_psi}
	We observe that $\psi$ is not a function of the norm $|(x-x',v-v')|$. Furthermore, the quadratic form $G(x-x',p)$ is convex in $(x-x',v-v')$, while $h$ is concave. Thus, $\psi(x-x',v-v')$ defines a semi-metric (rather than a metric) on $\R^n \times \R^n$, and $\cW_\psi$ is a semi-metric on the space of measures. Consequently, the contraction result established above does not guarantee the existence of the invariant measure, but only characterizes the convergence rate under $\cW_\psi$, assuming that the invariant measure exists (e.g. when $F$ is convex). Finally, considering the domination of $\cW_1$ by $\cW_\psi$, the contraction result also implies exponential convergence to the invariant measure in $\cW_1$.
\end{rem}

\begin{rem}[Comparison with recent results on coupling of kinetic Langevin dynamics]
	The proof of Theorem \ref{thm:Contraction} relies on the reflection-synchronous coupling technique, which was developed by Eberle, Guillin, and Zimmer in \cite{EGZ19}. 
	 In their work, they establish a contraction result under a semi-metric $\cW_{\widehat\psi}$ with   
	\begin{align}\label{eq:psihat}
	 \widehat \psi\big((x,v), (x', v')\big):= \big(1+ \eps\cV(x,v)+ \eps\cV(x',v')\big)h(\eta |v-v' + \g (x -x')| +|x-x'|),
	\end{align}
	where $h$ is a constructed non-negative increasing concave function and $\cV$ is a constructed Lyapunov function. 
	Notably, their contraction result holds under more general conditions than our Assumption \ref{assum:exp}. Specifically, they allow for a more general confinement function than $\frac{\l}{2}|x|^2$, but their contraction rate is dimension-dependent, and the kinetic dynamics do not permit mean-field dependence. To overcome these limitations, we introduce a new semi-metric that enables us to obtain exponential ergodicity in scenarios with small mean-field dependence. 
	Note that our contraction result has an advantage over the one presented in \cite{EGZ19}, even in cases without mean-field dependence.  For example, consider  the MFL dynamics \eqref{eq:MFL}  $(X_t, V_t)$  (resp. $(X',V')$) starting from $(x,v)$ (resp. $(x',v')$).  Suppose we have a Lipschitz continuous function $w:\dbR^n\x\dbR^n\rightarrow \dbR$, and define
	\begin{align*}
	\eta(x,v): = \dbE\left[ \int_0^\infty w(X_t, V_t)dt\right].
	\end{align*} 
	Our contraction result guarantees that
\begin{align*}
|\eta(x,v) - \eta(x',v')| &\le C \int_0^\infty \cW_1(m_t, m'_t)dt \le C \int_0^\infty \cW_\psi (m_t, m'_t)dt\\
 &\le C \big(1+\b G( x-x' ,v-v' + \g (x -x') ) \big) |(x,v)-(x',v')|.
\end{align*}
In particular, $\limsup_{(x',v')\rightarrow(x,v)} \frac{|\eta(x,v) - \eta(x',v')|}{ |(x,v)-(x',v')|} \le C$. Hence, $\eta$ is globally Lipschitz continuous. On the other hand, using the contraction result from \cite{EGZ19} and a similar estimate, we obtain
\[|\eta(x,v) - \eta(x',v')| \le C \big(1+ \eps\cV(x,v)+ \eps\cV(x',v')\big) |(x,v)-(x',v')|,\]
which implies that $\eta$ is only  locally Lipschitz.

Since the initial version of this paper, there have been recent advancements in the coupling method for the kinetic Langevin dynamics. Notably,  Guillin, Le Bris, and Monmarch\'e \cite{GLM22} and Schuh \cite{Schuh22} have both made important contributions using the reflection-synchronous coupling technique introduced in \cite{EGZ19}. In \cite{GLM22}, the authors establish a contraction result under a semi-metric $\cW_{\widehat\psi}$ with a function $\widehat\psi$ of the same form as in \cite{EGZ19} (see \eqref{eq:psihat}), but for more general cases where the kinetic process has small mean-field dependence and the confinement function can exhibit growth larger than quadratic. In \cite{Schuh22}, the author assumes a confinement function with at most quadratic growth and additionally requires  the function $x\mapsto D_m F^\circ (m,x)$ to be $L_g$-Lipschitz continuous with a small enough constant $L_g$. The author successfully proves a contraction result under a metric instead of a semi-metric. 
	
	 It is also worth mentioning that Guillin, Liu, Wu, and Zhang provided a proof of exponential ergodicity in \cite{GLWZ2021} for underdamped Langevin dynamics with convolution-type interactions, using an entirely different approach based on Villani's hypocoercivity and functional inequality.
\end{rem}

\begin{rem}[Mean-field Hamiltonian Monte Carlo method in \cite{BS2020}]
In their very recent work, Rou-Rabee and Schuh \cite{BS2020} provide a quantitative rate for the convergence  of unadjusted Hamiltonian Monte Carlo method for some mean-field models. More precisely, at each round $k$ of their Monte Carlo algorithm, they simulate the McKean-Vlasov process on $[0,T]$:
\begin{align*}
\begin{cases}
 dX^k_t = V^k_t dt, & \quad X^k_0 := X^{k-1}_T\\
 dV^k_t = \left(-\nabla_x f(X^k_t) - \eps \widetilde\dbE\left[\nabla_x w\big(X^k_t- \widetilde X^k_t\big) -\nabla_x w\big(\widetilde X^k_t - X^k_t\big) \right]\right)dt, & \quad V^k_0\sim \cN\big(0, \frac{\si^2}{2\g}\big),
\end{cases}
\end{align*}
where $\widetilde X^k$ is an independent copy of $X^k$, and the expectation $\widetilde{\mathbb{E}}$ is taken with respect to $\widetilde X^k$. Their main result, Theorem 3, shows that for small enough $\eps$, the distribution $\mathcal{L}(X^k_0)$ converges exponentially as $k\rightarrow \infty$.
Recall the second example in Example \ref{eg:deriveF} and note that 
\[
\nabla_x f(X^k_t) +\eps \widetilde\dbE\left[\nabla_x w\big(X^k_t- \widetilde X^k_t\big) -\nabla_x w\big(\widetilde X^k_t - X^k_t\big) \right] = D_m F\big(\Lc(X_t), X_t\big),
 \]
 where 
   $$ F(m):= \int f(x) m(dx) + \eps \int w(x - y) m(dx) m (dy), $$ 
  so the potential function in their study is a specific example of the one in our paper, although their argument may be adaptable to more general cases. The simulated McKean-Vlasov process above closely resembles the underdamped MFL dynamics \eqref{eq:MFL}, except that it replaces the damping term $-\gamma V_tdt$ and the Brownian noise $\sigma dW_t$ by setting $V^k_0\sim \mathcal{N}(0, \frac{\sigma^2}{2\gamma})$ at each iteration.

Based on the results, both papers require small mean-field dependence to ensure exponential convergence, even though they differ in the assumptions on the potential function $F$. Both papers use the component-wise coupling technique. In our paper, we apply the coupling to the Brownian noise $dW_t$, while in \cite{BS2020}, the authors apply a similar coupling to the Gaussian variable $V^k_0$. A significant difference is that we need to control both the couplings of $X$ and $V$ to prove the contraction for the underdamped MFL dynamics \eqref{eq:MFL}, while in \cite{BS2020}, the authors only need to consider the coupling of $X$ (as the law of $V^k_0$ is fixed). Consequently, our function to measure the coupling distance $\psi$ in \eqref{eq:funpsi} is more complicated than the counterpart in \cite{BS2020}, and the calculus in Section \ref{subsec:proofcontraction} is more elaborate.
\end{rem}

\section{Application to GAN}\label{sec:GAN}

\subsection{A mathematical model of GAN}
The mean-field Langevin dynamics draws increasing attention among the attempts to rigorously prove the trainability of neural networks, in particular the two-layer networks (with one hidden layer). It becomes popular (see e.g.~Chizat and Bach \cite{chizat2018global}, Mei, Montanari and Nguyen \cite{mei2018mean}, Rotskoff and Vanden-Eijnden \cite{rotskoff2018neural}, Hu, Ren, \v Si\v ska and Szpruch \cite{HRSS19}) to rewrite the two-layer network training problem as an optimization problem over the space of probability measures.
Namely, a two-layer network training problem can be formulated as 
	\begin{equation*}
		\inf_{c, a, b} \int \Big|y - \sum_i c_i \f(a_i z +b_i) \Big|^2 \mu(dy, dz), 
	\end{equation*}
	with the distribution $\mu$ of the data $z$ and the label $y$.
	By considering the law $m$ of the random variable $X:=(C, A, B)$ in $\R^n$,
	one can reformulate it as 
	\begin{align}\label{eq:mf_nn}
		&\inf_{m\in \Pc_2(\R^n)} F(m),\notag\\
		&\mbox{where}~~
		F(m):=   \! \int \! \big|y - \dbE^{m} [\Phi(X,z)] \big|^2 \mu(dy, dz) ~~\mbox{and}~~\Phi(X,z ):=C\f(Az+B).
	\end{align}
	In Mei, Montanari and Nguyen \cite{mei2018mean} and Hu, Ren, \v Si\v ska and Szpruch \cite{HRSS19} the authors further add an entropic regularization to the minimization:
	\begin{equation} \label{eq:overdamp_opt}
		\inf_{m\in \Pc_2(\R^n)}\left\{ F(m) + \frac{\si^2}{2} H(m)\right\}.
	\end{equation}
 	It  is due to \cite[Proposition 2.5]{HRSS19}  that
	an optimal solution $m^*$ to \eqref{eq:overdamp_opt} admits a density and  satisfies the first order necessary condition
	\begin{equation*}
		D_m F(m^*,x) + \frac{\si^2}{2}\nabla_x \log \big( m^*(x) \big) =0,
	\end{equation*}
	where the intrinsic derivative $D_m F(m,x) $ reads, with $x = (c,a,b)$,
	\begin{equation*}
	D_m F(m, x) = \! \int \! 2 \Big( \dbE^{m} [C\f(Az+B)] -y\Big)\begin{pmatrix}
		\f(az+b)\\
		c\dot\f(az+b)z\\
		c\dot\f(az+b)
		\end{pmatrix} \mu(dy, dz).
	\end{equation*}
Moreover, since  $F$ defined above is convex and $H$ is strictly convex, this is also a sufficient condition for $m^*$ being the unique minimizer. 
	It has been proved in \cite[Theorem 2.11]{HRSS19} that such $m^*$ can be characterized as the invariant measure of the overdamped mean-field Langevin dynamics:
    \begin{equation*}
     dX_t = - D_m F(\Lc(X_t), X_t) dt +\si dW_t.
    \end{equation*}
	Also it  has been  shown that the marginal laws $m_t$ converge towards $m^*$ in Wasserstein metric. Notably, the (stochastic) gradient descent algorithm used in training the neural networks can be viewed as a numerical discretization scheme for the overdamped MFL dynamics, see \cite[Section 3.2]{HRSS19} for more details.
	 It is  noteworthy that the optimization of the weights of the deep neural network (containing more than one hidden layer) can also be formulated as a mean-field optimization problem, which however is not convex. As a result, a general theory for the deep learning via mean-field optimization is still absent, though some partial results and analogs have been developed, see e.g.~Hu, Kazeykina and Ren \cite{HKR19}, Jabir, \v{S}i\v{s}ka and Szpruch \cite{JSS19}, Conforti, Kazeykina and Ren \cite{CKR20}, Domingo-Enrich, Jelassi and Mensch \cite{DJMRB2020}, \v{S}i\v{s}ka and Szpruch \cite{SS20}, Lu, Ma, Lu, Lu and Ying \cite{LMLLY20}.

Recently, there is a strong interest in generating samplings according to a distribution only empirically known using the so-called generative adversarial networks (GAN), see e.g.~the pioneering work  \cite{GAN14}.   
From a mathematical perspective, the GAN can be viewed as a (zero-sum) game between two players: the generator and the discriminator, and can be trained through an overdamped Langevin process, see e.g.~Conforti, Kazeykina and Ren \cite{CKR20}, Domingo-Enrich, Jelassi, Mensch, Rotskoff and Bruna \cite{DJMRB2020}. On the other hand, it has been empirically observed and theoretically  proved in some cases in Cheng, Chatterji, Bartlett and Jordan in \cite{CCBJ18} as well as in Cheng, Chatterji, Abbasi-Yadkori, Bartlett and Jordan \cite{CCAYBJ2020} that the simulation of the underdamped Langevin process converges more quickly than that of the overdamped Langevin dynamics.
Therefore, in this section we shall implement an algorithm to train the GAN through the underdamped mean-field Langevin dynamics.

\vspace{0.3cm} 

We first recall the mathematical model of GAN in \cite{CKR20}. The task of the discriminator is to measure the difference between the target measure $\hat\mu$ and another given measure $\mu\in \Pc_2(\dbR^{n_1})$. In our toy model we ask the discriminator to solve (approximately) the following maximization problem:
\begin{align*}
D(\mu, \hat\mu) : = \sup_{m^X\in \Pc_2(\dbR^{n_2})} \left\{ -F_\circ(m^X, \mu) \right\}:=\sup_{m^X\in \Pc_2(\dbR^{n_2})} \int_{\dbR^{n_1}}\dbE^{m^X}[\Phi(X, z) ](\mu -\hat\mu) (dz),
\end{align*}
where $z\mapsto \dbE^m[\Phi(X, z) ]$ is the output of the two-layer network with an activation function $\f$ as in \eqref{eq:mf_nn}. Indeed, the functional $D(\cd, \cd)$, resembling the dual form of the Wasserstein-$1$ metric, can be viewed as a distance between two probability measures,  and hence our model is similar to the mathematical model of the once popular Wasserstein-GAN \cite{ACB2017}. On the other hand, the generator aims at sampling a probability measure  $\mu$ so as to  minimize the distance $D(\mu, \hat\mu)$. 
The solution to this minimization problem is the required distribution $ \hat\mu$.


Let us add the velocity variable $V$ and the regularizers to the potential $F_\circ$:  
	\begin{align*}
		\fF(m, \mu) := -F_\circ(m^X, \mu)- \frac{\eta}{2}\dbE^m[|V|^2] +\frac{\l_0}{2} \int |z|^2 \mu(dz) - \frac{\l_1}{2} \dbE^m[|X|^2] + \frac{\si^2_0}{2} H(\mu) - \frac{\eta\si^2_1}{2\g}H(m).
	\end{align*}
The regularized GAN aims at computing the Nash equilibrium of the zero-sum game:
	\begin{eqnarray}\label{eq:game}
		\begin{cases}
			{\rm generator:}      &    \inf_{\mu\in \Pc_2(\dbR^{n_1})} \fF(m, \mu)\\
			{\rm discriminator:} &    \sup_{m\in \Pc_2(\dbR^{n_2}\times\dbR^{n_2})} \fF(m, \mu)
		\end{cases}.
	\end{eqnarray} 
\begin{rem}
A zero-sum game is a game in which the two players aim at minimizing and maximizing the same objective function. In particular, let $(m^*, \mu^*)$ be a Nash equilibrium of the zero game \eqref{eq:game}, i.e.
\[\mu^* = \argmin_{\mu\in \Pc_2(\dbR^{n_1})} \fF(m^*, \mu)\quad\mbox{and} \quad m^* =\argmax_{m\in \Pc_2(\dbR^{n_2}\times\dbR^{n_2})} \fF(m, \mu^*)\]
(since $\mu \mapsto \fF(m^*, \mu)$ and $m \mapsto -\fF(m, \mu^*)$ are both strictly convex, the minimizer and the maximizer above are both unique). It is well-known that the couple $(m^*, \mu^*)$ solves the min-max problem:
\begin{align*}
	\fF(m^*, \mu^*) =  \inf_{\mu\in \Pc_2(\dbR^{n_1})} \sup_{m\in \Pc_2(\dbR^{n_2}\times\dbR^{n_2})} \fF(m, \mu) =  \sup_{m\in \Pc_2(\dbR^{n_2}\times\dbR^{n_2})} \inf_{\mu\in \Pc_2(\dbR^{n_1})} \fF(m, \mu).
\end{align*}
\end{rem}
\begin{rem}
Here we briefly report the relation between the regularized game \eqref{eq:game} and the non-regularized one. Take a sequence $(\l^n_0, \l^n_1, \si^n_0, \si^n_1) \rightarrow 0$ as $n\rightarrow\infty$, denote the corresponding regularized function by $\fF^n$, and denote by $(m^n, \mu^n)$ the corresponding Nash equilibria. Let $(m^*, \mu^*)$ be one of the $\cW_2$-accumulation point of $(m^n, \mu^n)$, namely it is a Wasserstein-2 limit of a subsequence, still denoted by  $(m^n, \mu^n)$.   

Note that the regularizers for the generator and the discriminator,  namely $\frac{\l^n_0}{2} \int |z|^2 \mu(dz)+ \frac{(\si^n_0)^2}{2} H(\mu) $ and $\frac{\l^n_1}{2} \dbE^m[|X|^2] + \frac{\eta(\si^n_1)^n}{2\g}H(m)$, both $\Gamma$-converge (with respect to the $\cW_2$-distance) to $0$ as $n\rightarrow \infty$, see e.g.~the proof of \cite[Proposition 2.3]{HRSS19}. 
As a result, we have
  \begin{equation} \label{eq:limsupliminf}
  	\begin{aligned}
  	   \limsup_{n\rightarrow\infty} \fF^n(m^n, \mu^n) &= \inf_{\mu\in \Pc_2(\dbR^{n_1})}  \fF^\infty(m^*, \mu) \quad \mbox{and} \\
  	   \liminf_{n\rightarrow\infty} \fF^n(m^n, \mu^n) &= \sup_{m\in \Pc_2(\dbR^{n_2}\times\dbR^{n_2})} \fF^\infty(m, \mu^*),
  	\end{aligned}
  \end{equation}
where 
  $$ \fF^\infty(m,\mu):= -F_\circ(m^X, \mu)- \frac{\eta}{2}\dbE^m[|V|^2]. $$
It follows from \eqref{eq:limsupliminf} that
\begin{align*}
\fF^\infty(m^*, \mu^*) \ge \inf_{\mu\in \Pc_2(\dbR^{n_1})} \fF^\infty(m^*, \mu) \ge  \sup_{m\in \Pc_2(\dbR^{n_2}\times\dbR^{n_2})} \fF^\infty(m, \mu^*) \ge \fF^\infty(m^*, \mu^*),
\end{align*}
in other word, $(m^*, \mu^*)$ is a Nash equilibrium of the limit game. In particular, we have 
\begin{align*}
D(\mu^* ,\hat\mu) &=\sup_{m\in \Pc_2(\dbR^{n_2}\times\dbR^{n_2})} \left\{ -F_\circ(m^X, \mu^*) \right\} \\
&= \sup_{m\in \Pc_2(\dbR^{n_2}\times\dbR^{n_2})} \left\{ -F_\circ(m^X, \mu^*)- \frac{\eta}{2}\dbE^m[|V|^2]\right\} \\
&= \fF^\infty(m^*, \mu^*)\\
 &= \inf_{\mu\in \Pc_2(\dbR^{n_1})} \sup_{m\in \Pc_2(\dbR^{n_2}\times\dbR^{n_2})} \fF^\infty( m,\mu)\\
&=  \inf_{\mu\in \Pc_2(\dbR^{n_1})} \sup_{m\in \Pc_2(\dbR^{n_2}\times\dbR^{n_2})} \left\{ -F_\circ(m^X, \mu)- \frac{\eta}{2}\dbE^m[|V|^2]\right\} \\
&=\inf_{\mu\in \Pc_2(\dbR^{n_1})} \sup_{m\in \Pc_2(\dbR^{n_2}\times\dbR^{n_2})} \left\{ -F_\circ(m^X, \mu) \right\}\\
&= \inf_{\mu\in \Pc_2(\dbR^{n_1})} D(\mu ,\hat\mu)=0,
\end{align*}
 where the third and the forth equalities are due to the fact that $(m^*, \mu^*)$ is a Nash equilibrium.  Finally we conclude that $\mu^* = \hat\mu$, and that the initial sequence (not only the subsequence) $\mu^n\rightarrow\hat\mu$ in $\cW_2$. Therefore, when $\l^n_0, \l^n_1, \si^n_0, \si^n_1$ are small, the generator should eventually  generate the distribution $\mu^n$ close to the target distribution $\hat \mu$.
\end{rem}

In order to compute the equilibrium of the game \eqref{eq:game}, we observe as in Conforti, Kazeykina and Ren \cite{CKR20} that given the choice $m$ of the discriminator, the optimal response $\mu^*[m]$ of the generative must satisfy the first order condition
\[\frac{\d F}{\d \mu}(m, \mu^*[m], z) + \frac{\si^2_0}{2} \log\mu^*[m](z) = \mathrm{Constant}. \]
Therefore it has the density
	\begin{align} \label{eq:opt_gen}
		\mu^*[m](z) = C(m) e^{-\frac{2}{\si_0^2}\left(\dbE^{m^X}\left[\Phi(X, z)\right] + \frac{\l_0}{2}|z|^2\right)},
	\end{align}
    where $C(m)$ is the normalization constant depending on $m$. 
Then computing the value of the zero-sum game becomes an optimization over $m$:
	\begin{align*}
		\sup_{m\in \Pc_2(\dbR^{n_2}\times\dbR^{n_2})} \inf_{\mu\in \Pc_2(\dbR^{n_1})} \fF(m,\mu) = \sup_{m\in \Pc_2(\dbR^{n_2}\times\dbR^{n_2})} \fF(m, \mu^*[m]).
	\end{align*}	
Thanks to Theorem \ref{thm:ergodicity}, the optimizer of the problem above can be characterized by the invariant measure of the underdamped MFL dynamics
	\begin{align} \label{eq:MFLapp}
		d X_t = \eta V_t dt,
		~~~~
		d V_t = -\big( D_m F(\Lc(X_t), X_t )+  \g V_t \big) dt + \si_1 dW_t,
	\end{align}
	with the potential function:
	\begin{align*}
		F(m):= F_\circ\big(m^X, \mu^*[m]\big) + \frac{\l_1}{2} \dbE^m\big[|X|^2\big]- \frac{\si^2_0}{2} H(\mu^*[m]).
	\end{align*}
Together with \eqref{eq:opt_gen}, we may calculate and obtain
	\begin{align}\label{eq:DmFapp}
		D_m F(m, x) = \int \nabla_x \Phi(x, z)(\hat\mu - \mu^*[m] )(dz) + \l_1 x.
	\end{align}

\begin{rem}
Instead of reporting the rigorous but tedious computation to obtain \eqref{eq:DmFapp}, here is a quick way to intuitively realize the form of $D_m F$. Observe the following \textnormal{formal} calculus:
\begin{align*}
	D_m F(m,x) &= - D_m \left(\int_{\dbR^{n_1}}\dbE^{m^X}[\Phi(X, z) ](\mu^*[m] -\hat\mu) (dz) + \frac{\si^2_0}{2} H(\mu^*[m])\right) + \l_1 x\\
	&= \int \nabla_x \Phi(x, z)(\hat\mu - \mu^*[m] )(dz) + \l_1 x \\
	& \quad + \frac{\pa}{\pa \mu^*[m]} \left(\int_{\dbR^{n_1}}\dbE^{m^X}[\Phi(X, z) ](\mu^*[m] -\hat\mu) (dz) + \frac{\si^2_0}{2} H(\mu^*[m])\right) D_m \mu^*[m].
\end{align*}
Since by its definition $\mu^*[m]\in\argmin_{\mu\in \Pc_2(\dbR^{n_1})}  \left(\int_{\dbR^{n_1}}\dbE^{m^X}[\Phi(X, z) ](\mu -\hat\mu) (dz) + \frac{\si^2_0}{2} H(\mu)\right)$, the formal ``partial derivative'' term above should be equal to $0$, and this leads to the result \eqref{eq:DmFapp}. 
\end{rem}

\subsection{Numerical test}

We shall illustrate the theoretical result with a simple numerical example. 
Set $\hat \mu$ as the empirical law of $\widehat M=2000$ samples, denoted by $(\widehat z_k)_{k\le \widehat M}$, of the distribution $\frac12\cN(-1,1) + \frac12 \cN(4,1)$. We choose the activation function of our two-layer neural network as 
	\begin{align*}
		\f (\o) = \max\{-10, \min\{10,\o\}\}, 
	\end{align*}
 and the coefficients of the model as:
	\begin{align}\label{parametersapp}
		\eta =0.3, ~~\si_0 = 0.5,  ~~ \g =3, ~~ \l_0=\l_1 =0.02.
	\end{align}
We shall approximate the MFL dynamics \eqref{eq:MFLapp} using the following Euler scheme of the $N$-particle system: for all $1\le i \le N$, $t\in \dbN,$
\begin{align} \label{eq:Euler}
\begin{cases}
	X^i_{t+1} - X^i_t  = \eta V^i_t \D t  \\
	V^i_{t+1} - V^i_t  = -\left( D_m F\Big(\frac1N\sum_{j=1}^N \d_{X^j_t}, X^i_t \Big)+  \g V^i_t \right) \D t + \si_1 \sqrt{\D t} \cN^i_t,
\end{cases}
\end{align}
where $(\cN^i_t)_{i, t}$ are independent normal Gaussian random variables. In our numerical test we set
\begin{align*}
N = 3000, ~~ \D t = 0.005,
\end{align*}
and set the initial values of $(X^i, V^i)_i$ according to the following Gaussian distributions:
\begin{align*}
X\sim \cN(0, 1)*12, ~~ V\sim\cN(0,1)*0.2.
\end{align*}
Recall that in order to evaluate $D_m F\big(\frac1N\sum_{j=1}^N \d_{X^j_t}, X^i_t \big)$ in the MFL dynamics, defined in \eqref{eq:DmFapp}, one need to sample the optimal response of the generator $\mu^*[\frac1N\sum_{j=1}^N \d_{X^j_t}]$, defined in \eqref{eq:opt_gen}. Given $M$ such samples, denoted by $(z^k_t)_{k\le M}$, we may evaluate
$$
	D_m F\big(\frac1N\sum_{j=1}^N \d_{X^j_t}, X^i_t \big)\approx
	 \frac{1}{\widehat M}\sum_{k=1}^{\widehat M}\nabla_x \Phi(X^i_t, \hat z_k) - \frac1M\sum_{k=1}^M \nabla_x \Phi(X^i_t,  z^k_t)  + \l_1 X^i_t.
$$
In the numerical test, we use the Gaussian random walk Metropolis Hasting algorithm, with the optimal scaling proposed in Gelman, Roberts and Gilks \cite{metropolis_efficient}, to generate $M=2000$ samples of the distribution $\mu^*[\frac1N\sum_{j=1}^N \d_{X^j_t}]$. 
Further, on the basis of the Euler scheme \eqref{eq:Euler}, we shall apply the well-known OBABO splitting procedure for the underdamped Langevin process, see \cite[Chapter 7.3.1]{Molecular15}. 
Along the simulation, we record the potential energy
	\begin{align}\label{eq:numericalpotential}
		\frac1N\sum_{i=1}^N \left( \frac{1}{\widehat M}\sum_{k=1}^{\widehat M} \Phi(X^i_t, \hat z_k) - \frac1M\sum_{k=1}^M \Phi(X^i_t,  z^k_t)  \right),
	\end{align}
    as well as the kinetic one 
    \begin{align}\label{eq:numericalkinetic}
    \frac{\eta}{2}\frac1N\sum_{i=1}^N |V^i_t|^2,
    \end{align}
     and we run $1000$ iterations. 
    As a reference, we also compute the Wasserstein-1 distance between the target distribution and the result of the generator, namely, 
    \[ 
    \cW_1\left(\hat\mu , \mu^*\bigg[\frac1N\sum_{i=1}^N \d_{X^i_t}\bigg] \right)\approx
    	 \cW_1\left(\frac{1}{\widehat M} \sum_{k=1}^{\widehat M}\d_{\widehat z_k}, \frac1M\sum_{k=1}^M \d_{z^k_t} \right).
    \]
We have done the numerical tests for $\sigma_1$ equal to $0.5$ and $1.5$, respectively. 
In the following Figure \ref{fig:GAN}, where we  show the potential energy \eqref{eq:numericalpotential} as well as the sum of the potential energy  \eqref{eq:numericalpotential} and the kinetic one \eqref{eq:numericalkinetic} along the trainings. As a reference we also compute and show the $\cW_1$-distance to the target distribution. Note that throughout the training process, the potential exhibits oscillatory behavior while the sum of the potential and kinetic energy decreases almost monotonically.
\begin{figure}[h]
	\centering
	\includegraphics[width=0.45\textwidth]{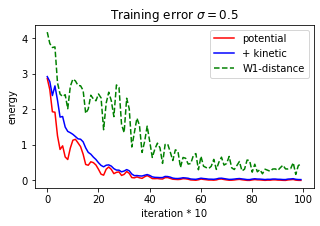}
	\includegraphics[width=0.45\textwidth]{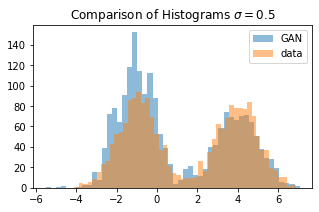}\\
	\includegraphics[width=0.45\textwidth]{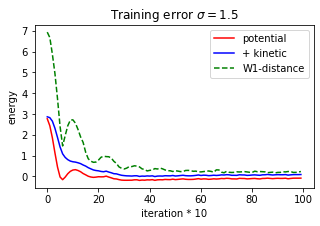}
	\includegraphics[width=0.45\textwidth]{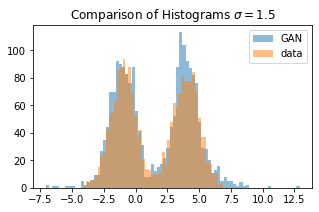}\\
	\includegraphics[width=0.45\textwidth]{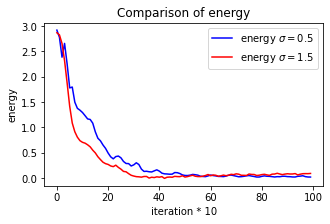}
	\includegraphics[width=0.45\textwidth]{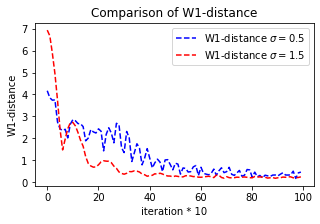}
	\caption{Training errors and GAN samplings} \label{fig:GAN}
\end{figure}
Both the energy and the Wasserstein-1 distance between the target distribution and the generated one become eventually small and their histograms match satisfactorily.
Note that the sum of the potential and the kinetic energy   is not yet monotonously decreasing, because the entropy has not been counted in the total energy.
When comparing the results for different $\sigma_1$, we observe that the energy drops more quickly with bigger $\sigma_1$ but eventually bears a bigger bias. This observation aligns with the contraction bound established in Theorem \ref{thm:Contraction}, which is highlighted in Remark \ref{rem:convergence rate}. 

As a comparison, we also train the GAN using the overdamped MFL dynamics as in \cite{CKR20}. More precisely, we simulate the diffusion following the Euler scheme:
 \begin{align*}
X^i_{t+1} - X^i_t  = -  D_m F\bigg(\frac1N\sum_{j=1}^N \d_{X^j_t}, X^i_t \bigg) \D t + \si_1 \sqrt{\D t} \cN^i_t,\quad  \mbox{for all}~~1\le i \le N, ~~t\in \dbN,
\end{align*}
where again $(\cN^i_t)_{i, t}$ are independent normal Gaussian random variables. We set the same parameters as in \eqref{parametersapp}, and choose 
\begin{align*}
N = 3000, ~~ \D t = 0.005, ~~\sigma_1=0.1.
\end{align*}
In order to  evaluate $D_m F\big(\frac1N\sum_{j=1}^N \d_{X^j_t}, X^i_t \big)$, the samples of $\mu^*[\frac1N\sum_{j=1}^N \d_{X^j_t}]$ are still generated by the same Gaussian random walk Metropolis Hasting algorithm as above. 
We run the simulation of the overdamped MFL dynamics, and compare the result to that of the underdamped MFL dynamics with the volatility $\sigma_1=0.1$  in Figure \ref{fig:GAN_over}. As we can see, in both cases the Wasserstein-1 distances between the target distribution and the generated one evolve in a similar pattern as the potential energies. In the underdamped case they decrease with oscillation, whereas in the overdamped case they exhibit a predominantly monotonically decreasing trend.
At the end of the training, the Wasserstein-1 distances between the target distribution and the generated one become comparably small.
\begin{figure}[h]
	\centering
	\includegraphics[width=0.45\textwidth]{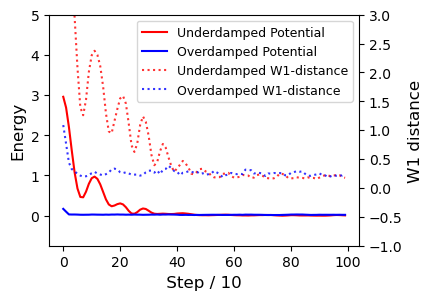}\\
	\includegraphics[width=0.45\textwidth]{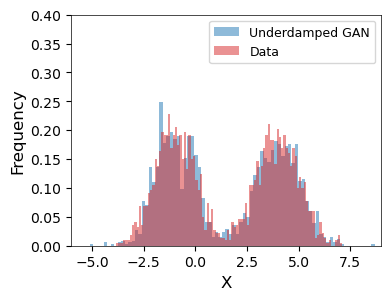}
	\includegraphics[width=0.45\textwidth]{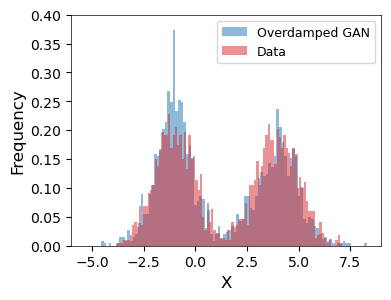}
	\caption{Training errors and GAN samplings - Overdamped MFL} \label{fig:GAN_over}
\end{figure}

\section{Proofs} \label{sec:proof}

	In this section we prove the main results stated in Section \ref{sec:mainresult}. In Section \ref{secproof:prepare} we prepare some preliminary results on the properties of the marginal distributions of the underdamped MFL dynamics, in particular, the integrability result in  Lemma \ref{lemm:Dvm_integ} and the regularity result in Proposition \ref{prop:regularity}. Using them we prove the main Theorem \ref{thm:Lyapunov} (decay of free energy function) in Section \ref{secproof:Lyapunov} and Theorem \ref{thm:ergodicity} (convergence of marginal distributions) in Section \ref{secproof:ergodicity}. Note that the proofs of Theorem \ref{thm:Lyapunov}  and \ref{thm:ergodicity} largely follow the strategy developed in \cite{HRSS19} which investigates similar problems for the overdamped MFL dynamics.   Further, we prove the exponential ergodicity result, Theorem \ref{thm:Contraction}, using the reflection coupling in Section \ref{secproof:exp}.

\subsection{Some fine properties of the marginal distributions of the SDE}\label{secproof:prepare}

	Let $(\Om, \Fc, \dbP)$ be an abstract probability space, equipped with a $n$-dimensional standard Brownian motion $W$.
	Let $T> 0$, $b: [0,T] \x \R^n \x \R^n \longrightarrow \R^n$ be a continuous function such that,  for some constant $C> 0$,
	$$
		|b(t,x,v) - b(t,x', v')| 
		~\le~
		C(|x-x'| + |v-v'|),  
	$$
	for all $(t,x,v,x',v') \in [0,T] \x \R^{2n} \x \R^{2n}$, and $\sigma > 0$ be a positive constant,
	we study the stochastic differential equation (SDE):
	\begin{equation} \label{eq:SDE}
		d X_t = V_t dt,
		~~~~
		d V_t = b(t, X_t, V_t)  dt + \sigma dW_t,
	\end{equation}
	where the initial condition $(X_0, V_0)$ satisfies 
	$$
		\E\big[ |X_0|^2 + |V_0|^2\big] < \infty.
	$$
	
	\begin{rem}
	   	$\mathrm{(i)}$ Under the Lipschtiz condition on the drift function $b$, it is well-known that SDE \eqref{eq:SDE} has a unique strong solution $(X,V)$,
	   	  and the marginal distribution $\rho_t := \Lc(X_t, V_t)$ satisfies (in the sense of distribution) the corresponding Fokker-Planck equation:
	   	  \begin{equation} \label{eq:FP_abstact}
	   	  \pa_t \rho  + v \cd \nabla_x \rho  + \nabla_v\cd \big(b \rho\big)  - \frac{1}{2}  \sigma^2 \Delta_v \rho = 0.
	   	  \end{equation}
	   	$\mathrm{(ii)}$ We will first consider the SDE with general drift function $b$ and deduce some fine properties of the density function $\rho_t(x,v)$ of $\rho_t = \Lc(X_t, V_t)$.
	   	  In a second step, we apply these results to the MFL dynamic \eqref{eq:MFL} whose marginal distribution is denoted by $m_t$.
	\end{rem}

\paragraph{Existence of strict positive and smooth density function}

Let us fix a time horizon $T> 0$.
Let $C([0,T], \R^n)$ be the space of all $\R^n$-valued continuous paths on $[0,T]$.
Denote by $\Omb := C([0,T], \R^n) \x C([0,T], \R^n)$ the canonical space, with canonical process $(\Xb, \Vb) = (\Xb_t, \Vb_t)_{0 \le t \le T}$ and canonical filtration $\Fb = (\Fcb_t)_{t \in [0,T]}$ defined by $\Fcb_t := \sigma(\Xb_s, \Vb_s ~: s \le t)$.
Let $\Pb$ be a (Borel) probability measure on $\Omb$, under which 
	\begin{equation} \label{eq:dXV_P}
		\Xb_t = \Xb_0 + \int_0^t \Vb_s ds, 
		~\mbox{a.s. and}~~
		\big( \sigma^{-1}(\Vb_t-\Vb_0)\big)_{0\leq t\leq T}
		~\mbox{is a Brownian motion},
	\end{equation}
and $\Pb \circ (\Xb_0, \Vb_0)^{-1}  = \dbP \circ (X_0, V_0)^{-1}$.

	\vspace{0.5em}

	Then under the  measure $\Pb$, $(\Xb_0, \Vb_0)$ is independent of $(\Xb_t - \Xb_0 - \Vb_0 t, \Vb_t - \Vb_0)$, and the latter follows a Gaussian distribution with mean value $0$ and $2n \x 2n$ variance matrix
	\begin{equation} \label{eq:VarMatrix}
		\Sigma
		:=
		\sigma^2
		\begin{pmatrix}
		{t^3} I_n /3 &  {t^2} I_n  / {2} \\
		{t^2} I_n  /{2} & t I_n 
		\end{pmatrix}.
	\end{equation}
	Let $\Qb:= \dbP\circ (X, V)^{-1}$ be the image measure of the solution $(X, V)$ to the SDE \eqref{eq:SDE}, so that
	\begin{equation} \label{eq:XVh_dyanmic}
		d \Xb_t = \Vb_t dt,
		~~~~
		d \Vb_t = b(t, \Xb_t, \Vb_t) dt + \sigma d \Wb_t,
		~~\Qb \mbox{--a.s.,}
	\end{equation}
	with a $\Qb$-Brownian motion $\Wb$.

	\vspace{0.5em}

	Similar to the time homogeneous context (i.e. $b(t,x,v)$ is independent of $t$) in Talay \cite{Talay02},
	we provide a result on the existence of  strictly positive (smooth) density function of the marginal distribution of solution $(X,V)$ to  \eqref{eq:SDE}.

\begin{lem} \hspace{0mm} \label{lem:exist_density}
		$\mathrm{(i)}$ The probability measure $\Qb$ is equivalent to $\Pb$ and, with ${\bar b}_s := b(s, \Xb_s, \Vb_s)$,
		\begin{align} \label{eq:DensityZ}
		\frac{d \Qb}{d \Pb} \Big|_{\Fcb_T} =  Z_T, 
		~~\mbox{with}~
		Z_t := \exp \bigg( \int_0^t \sigma^{-2} {\bar b}_s \cdot d \Vb_s - \frac12 \int_0^t  \big|  \sigma^{-1} {\bar b}_s  \big|^2 ds \bigg).
		\end{align}
		
		\noindent $\mathrm{(ii)}$ Consequently, for the solution $(X,V)$ of \eqref{eq:SDE}, it marginal distribution $\Lc(X_t, V_t)$ has strictly positive density function, denoted by $\rho_t(x,v)$ for $t > 0$.
		
		\vspace{2mm}
		
		\noindent $\mathrm{(iii)}$ Assume in addition that $b \in C^{\infty}((0,T) \x \R^{2n})$ with all derivatives of order $k$ bounded for all $k\ge 1$.
		Then the function $(t,x,v)\mapsto \rho_t(x,v)$ belongs to $C^{\infty}((0,T) \x \R^{2n})$.
\end{lem}		

\begin{proof}
	$\mathrm{(i)}.$ Notice that $\Xb_t = \int_0^t \Vb_s ds$, $\Qb$-a.s., we can then apply e.g.~\"Ust\"unel and Zaka\"i  \cite[Theorem 2.4.2]{ustunel2013transformation} to obtain that $\Qb$ and $\Pb$ are equivalent, and that $\frac{d \Qb}{d \Pb} \Big|_{\Fcb_T} = Z_T$.

	\vspace{0.5em}

	\noindent $\mathrm{(ii)}.$ We observe that under $\Pb$, $(\Xb, \Vb)$ can be written as the sum of a square integrable random variable and an independent Gaussian random variable with variance \eqref{eq:VarMatrix}, then $\Pb \circ (\Xb_t, \Vb_t)^{-1}$ has strictly positive and smooth density function.
	Besides, $\Qb$ is equivalent to $\Pb$, with strictly positive density ${d \Qb}/{d \Pb} = Z_T$, it follows that $\dbP \circ (X_t, V_t)^{-1} = \Qb \circ (\Xb_t, \Vb_t)^{-1}$ has also a strictly positive density function.

	\vspace{0.5em}

	\noindent  $\mathrm{(iii)}.$ 
	Under the additional regularity conditions on $b$, it is easy to check that the coefficients of SDE \eqref{eq:SDE} satisfies the H{\"o}rmander's conditions,
	and hence the density function $\rho \in C^{\infty}((0,T) \x \R^{2n})$, see e.g.~Cattiaux and Mesnager \cite[(1.5), P.~2]{CM02} or Bally \cite[Theorem 5.1, Remark 5.2]{bally1991connection}.  
\end{proof}

\paragraph{Estimates on the densities}

We next provide an estimate on $\nabla_v \log \big( \rho_t(x,v) \big)$,
which is crucial for proving Theorem \ref{thm:Lyapunov}. 

\begin{lem}[Moment estimate]\label{lem:moment}
	Suppose that $\E \big[ |X_0|^{2p} + |V_0|^{2p} \big] <\infty$ for some $p \ge 1$, then 
	\begin{align} \label{eq:2ndMoment}
		\E \bigg[ \sup_{0 \le t \le T} \big( |X_t|^{2p} + |V_t|^{2p} \big) \bigg]< \infty.
	\end{align}
	Consequently, the relative entropy between $\Qb$ and $\Pb$ is finite, i.e.
	\begin{align} \label{eq:EntroyFinite}
		H \big(\Qb \big| \Pb \big) 
		:= \E^{\Qb} \bigg[ \log \bigg( \frac{d \Qb}{d \Pb} \bigg) \bigg] 
		= \E \bigg[ \frac12 \int_0^T \big| \sigma^{-1} b(t, X_t, V_t) \big|^2 dt \bigg] < \infty.
	\end{align}
\end{lem}

\begin{proof}
	Let us first consider \eqref{eq:2ndMoment} under the condition $\E \big[ |X_0|^{2} + |V_0|^{2} \big] <\infty$.
	As $b$ is of linear growth in $(x,v)$,
	it is standard to apply It\^o formula on $|X_t|^{2p} + |V_t|^{2p}$, and use BDG inequality and then Gr\"ownwall lemma to obtain \eqref{eq:2ndMoment}.
	
	\vspace{0.5em}

	Next, by \eqref{eq:DensityZ}, one has
	$$
		H \big(\Qb \big| \Pb \big) 
		=
		\E^{\Qb} \bigg[ \frac12 \int_0^T \big| \sigma^{-1} b(t, \Xb_t, \Vb_t) \big|^2 dt \bigg].
	$$
	Since $\Qb = \mathbb{P} \circ (X, V)^{-1}$ (see \eqref{eq:SDE} and \eqref{eq:XVh_dyanmic}), then
	it follows by the linear growth of $b$ together with \eqref{eq:2ndMoment} that
	$$
		H \big(\Qb \big| \Pb \big) 
		=
		\E \bigg[ \frac12 \int_0^T \big| \sigma^{-1} b(t, X_t, V_t) \big|^2 dt \bigg] 
		<
		\infty.
	$$
\end{proof}

\vspace{0.5em}

Let us introduce the time reverse process $(\Xt, \Vt)$ and time reverse probability measures $\Pt$ and $\Qt$ on the canonical space $\Omb$ by
$$
\Xt_t := \Xb_{T-t}, ~ \Vt_t := \Vb_{T-t},~t \in [0,T],
~~\mbox{and}~~
\Pt := \Pb \circ (\Xt, \Vt)^{-1},
~~
\Qt := \Qb \circ (\Xt, \Vt)^{-1}.
$$

\begin{lem} \label{lemm:Dvm_integ}
	The density function $\rho_t(x,v)$ is absolutely continuous in $v$, and it holds that
	\begin{equation} \label{eq:DvRho_integ}
	\E \bigg[\int_{t}^T \big|\nabla_v \log \big( \rho_s (X_s, V_s) \big) \big|^2 ds \bigg] 
	<\infty,
	\quad \mbox{for all} \quad t > 0.
	\end{equation}
\end{lem}

\begin{proof}
This proof is largely based on the time-reversal argument in F\"ollmer \cite[Lemma 3.1 and Theorem 3.10]{Follmer}, where the author sought a similar estimate for a non-degenerate diffusion. 
For simplicity of notations, let us assume  $\sigma = 1$.
	
	\vspace{0.5em}
	
	\noindent {\it Step 1.}\quad We first prove that, $(\Xb, \Vb)$ is an It\^o process under $\Qt$, 
	and there exists a $\Fb$-predictable process $\tilde b = (\tilde b_s)_{s \in [0, T)}$ such that, with a $(\Fb, \Qt)$-Brownian motion $\Wt$,
	\begin{equation} \label{eq:bt2_integ}
		\E^{\Qt} \bigg[ \int_0^t \big| \tilde b_s \big|^2 ds \bigg] < \infty,
		~~\mbox{and}~~
		\Vb_t  = \Vb_0 + \int_0^t \tilde b_s ds + \Wt_t,
		~~\mbox{for all}~t \in [0, T).
	\end{equation}
	
	\vspace{0.5em}

	Let $(\Pb_{x_0,v_0})_{(x_0, v_0) \in \R^n \x \R^n}$ be a family of regular conditional probability distribution (r.c.p.d.) of $\Pb$ knowing $\sigma(X_0, V_0)$,
	such that $\Pb_{x_0, v_0} \big[ \Xb_0 = x_0, \Vb_0 = v_0 \big] = 1$ and Conditions \eqref{eq:dXV_P} holds still true under $\Pb_{x_0, v_0}$ for every $(x_0, v_0) \in \R^n \x \R^n$.
	Let 
	$$
		\Pt_{x_0,v_0} := \Pb_{x_0,v_0} \circ (\Xt, \Vt)^{-1}.
	$$
	Recall the dynamic of $(\Xb, \Vb)$ under $\Pb$ in \eqref{eq:dXV_P} and notice that the marginal distribution of $(\Xb_t, \Vb_t)$ under $\Pb_{x_0, v_0}$ is Gaussian with density function 
	$$
		\rho^{x_0, v_0}_t(x,v)
		=
		\frac{1}{\sqrt{(2\pi)^{2n} |\Sigma|}}
		\exp\bigg( 
			-\frac12  \begin{pmatrix} x - (x_0 + v_0 t) \\ v- v_0 \end{pmatrix}^{\top} \Sigma^{-1}  \begin{pmatrix} x - (x_0 + v_0 t) \\ v- v_0 \end{pmatrix}
		\bigg),
	$$
	where $\Sigma$ is defined by \eqref{eq:VarMatrix}, so that
	$$
		\Sigma^{-1} 
		= 
		2 \sigma^{-2}
		\begin{pmatrix}
		6 t^{-3} I_n   &  - 3 {t^{-2}} I_n   \\
		- 3 {t^{-2}} I_n  & 2 {t^{-1}} I_n
		\end{pmatrix}.
	$$
	By direct computation, one obtains
	$$
		\tilde c_s(x_0, v_0, x, v)
		:=
		\nabla_v \log \big( \rho^{x_0, v_0}_{T-s} (x,v) \big) 
		=
		\frac{6 ( x_0 + (T- s) v_0 - x)}{(T-s)^2} + \frac{ 4 (v_0 - v)}{T-s},
		~~s \in[0, T).
	$$
	It follows from Theorem 2.1 of Haussmann and Pardoux \cite{haussmann1986time} (or Theorem 2.3 of Millet, Nualart and Sanz \cite{millet1989integration}) that $\Vb$ is still a diffusion process w.r.t. $(\Fb, \Pt_{x_0, v_0})$, and
	$$
		\Vb_t  - \Vb_0 - \int_0^t \nabla_v \log \big( \rho^{x_0, v_0}_{T-s}  (\Xb_s, \Vb_s) \big) ds
		~~\mbox{is a}~~(\Fb, \Pt_{x_0, v_0})\mbox{-Brownian motion on}~[0,T).
	$$
	Notice that, by its definition, $(\Pt_{x_0, v_0})_{(x_0, v_0) \in \R^n \x \R^n}$ is a family of conditional probability of $\Pt$ knowing $(\Xb_T, \Vb_T)$,
	or equivalently $\Pt_{x_0, v_0}[\cdot] = \Pt \big[ \cdot \big| (\Xb_T, \Vb_T) = (x_0, v_0) \big]$,
	it follows that
	$$
	\Wt^1_t := \Vb_t - \Vb_0 - \int_0^t \tilde c_s(\Xb_T, \Vb_T, \Xb_s, \Vb_s) ds
	\quad \mbox{is a}~ 
	(\Fb^*, \Pt)
	\mbox{-Brownian motion on}~[0,T),
	$$
	where  the enlarged filtration $\Fb^* = (\Fcb^*_t)_{0 \le t \le T}$ is defined by
	$$
	\Fcb^*_t := \sigma \big(\Xb_T, \Vb_T, \Xb_s, \Vb_s ~: s \in [0,t] \big).
	$$
	By the moment estimate \eqref{eq:2ndMoment}, we have
	$$
	\E^{\Qt} \bigg[ \int_{0}^{t} |\tilde c_s(\Xb_T, \Vb_T, \Xb_s, \Vb_s)|^2 ds \bigg] 
	=
	\E^{\Qb} \bigg[ \int_{T-t}^{T} |\tilde c_s(\Xb_0, \Vb_0, \Xb_s, \Vb_s)|^2 ds \bigg]  
	< \infty, ~~\forall t \in [0,T).
	$$
	Next notice that the relative entropy satisfies
	$$
	H\big(\Qt | \Pt\big) = H(\Qb | \Pb) < \infty.
	$$
	Therefore, there exists a $\Fb^*$-predictable process $\tilde a$ such that 
	$$
		\E^{\Qt} \left[ \int_0^T |\tilde a_t|^2 dt \right] 
		=
		H(\Qt | \Pt)
		<
		\infty,
	$$
	and 
	\begin{align*}
		\Wt^2_t := \Wt^1_t - \int_0^t \tilde a_s ds 
		= \Vb_t -\Vb_0 - \int_0^t \ \big( \tilde a_s + \tilde c_s(\Xb_T, \Vb_T, \Xb_s, \Vb_s) \big)ds,
		~~t  \in [0,T),
	\end{align*}
	is a $(\Fb^*, \Qt)$-Brownian motion.
	Finally, by letting $(\tilde b_s)_{s \in [0,T)}$ be the predictable projection of the process $ \big( \tilde a_s + \tilde c_s (\Xb_T, \Vb_T, \Xb_s, \Vb_s) \big)_{s \in [0,T)}$ w.r.t. $(\Fb, \Qt)$,
	we concludes the proof of Claim \eqref{eq:bt2_integ}.
\vspace{4mm}
	
\noindent {\it Step 2.} \quad		
Let $R: \Omb \longrightarrow \Omb$ be the reverse operator defined by $R(\omb) = (\omb_{T-t})_{0 \le t \le T}$.
Then for every fixed $t < T$ and $\varphi \in C_c(\R^{2n})$, one has
	$$
		\E^{\Qb} \Big[ \big( \bt_{T-t} \circ R \big) \varphi( \Xb_t, \Vb_t) \Big]
		~=~
		- \lim_{h \searrow 0} \frac1h \E^{\Qb} \Big[ \big( \Vb_t - \Vb_{t-h} \big) \varphi \big( \Xb_t, \Vb_t \big)  \Big].
	$$
	Recall the dynamic of $(\Xb, \Vb)$ under $\Qb$ in \eqref{eq:XVh_dyanmic}, and thus
	\begin{align*}
		\varphi \big( \Xb_t, \Vb_t \big) 
		&= \varphi \big( \Xb_{t-h}, \Vb_{t-h} \big) + \int_{t-h}^t \nabla_x \varphi(\Xb_s, \Vb_s) \cdot \Vb_s ds \\
		&\quad + \int_{t-h}^t \nabla_v \varphi(\Xb_s, \Vb_s) \cdot d \Vb_s ~+~ \frac12 \int_{t-h}^t  \Delta_v \varphi(\Xb_s, \Vb_s) ds,
		\quad \Qb\mbox{-a.s.}
	\end{align*}
	Therefore,
	\begin{eqnarray*}
		&& \E^{\Qb} \Big[ \big( \Vb_t - \Vb_{t-h} \big) \varphi ( \Xb_t, \Vb_t )  \Big] \\
		&=& 
		\E^{\Qb} \Big[ \big( \Vb_t - \Vb_{t-h} \big) \varphi ( \Xb_{t-h}, \Vb_{t-h} )  \Big]
		~+~
		\E^{\Qb} \Big[ \big( \Vb_t - \Vb_{t-h} \big) \Big( \varphi ( \Xb_t, \Vb_t ) -  \varphi ( \Xb_{t-h}, \Vb_{t-h}) \Big)  \Big] \\
		&=&
		\E^{\Qb} \left[  \varphi ( \Xb_{t-h}, \Vb_{t-h} ) \int_{t-h}^t b(s, \Xb_s, \Vb_s) ds \right]
		~+~
		\E^{\Qb} \left[ \int_{t-h}^t \nabla_v \varphi( \Xb_s, \Vb_s) ds \right].
	\end{eqnarray*}
	Denoting
	\begin{equation} \label{eq:bh2_integ}
	\bar b_t := b(t, \Xb_t, \Vb_t),
	~\mbox{which clearly satisfies that }~ \E^{\Qb} \bigg[ \int_0^T \big|{\bar b}_t\big|^2 dt  \bigg] < \infty,
	\end{equation}
	it follows that 
	\begin{align*}
		\E^{\Qb} \Big[ \big( \bt_{T-t} \circ R \big) \varphi( \Xb_t, \Vb_t) \Big]
		=
		- \E^{\Qb} \Big[ {\bar b}_t \varphi( \Xb_t, \Vb_t) \Big] - \E^{\Qb} \Big[ \nabla_v \varphi( \Xb_t, \Vb_t) \Big].
	\end{align*}
Therefore, denoting by $\nabla_v \rho_t(x,v)$ the weak derivative of $\rho$ in the sense of distribution, one has
	\begin{align*}
		\int_{\R^{2n}} \nabla_v \rho_t(x,v) \varphi(x,v) dx dv
		 &= - \E^{\Qb} \Big[ \nabla_v \varphi( \Xb_s, \Vb_s) \Big] \\
		 &= \E^{\Qb} \Big[ \big( \bt_{T-t} \circ R + {\bar b}_t \big) \varphi( \Xb_t, \Vb_t) \Big].
	\end{align*}
As $\varphi \in C_c(\R^{2n})$ is arbitrary, this implies that, for a.e.~$(x,v)$,
	\begin{align*}
		\nabla_v \rho_t(x,v)
		&= \rho_t (x,v)\E^{\Qb} \Big[  \big( \bt_{T-t} \circ R + {\bar b}_t \big) \Big| \Xb_t = x, \Vb_t = v \Big].
	\end{align*}
Finally, it follows from  the moment estimates in \eqref{eq:bt2_integ} and \eqref{eq:bh2_integ} that
    \begin{align*}
    	\E^{\Qb} \bigg[ \int_{t_0}^{t_1} \big| \nabla_v  \log \big( \rho_t (\Xb_t, \Vb_t) \big) \big|^2 dt \bigg] 
    	= \E^{\Qb} \bigg[ \int_{t_0}^{t_1} \Big| \frac{\nabla_v \rho_t}{\rho_t}(\Xb_t, \Vb_t) \Big|^2 dt \bigg] 
    	< \infty.
    \end{align*}
We hence conclude the proof by the fact that $\dbP \circ (X,V)^{-1} = \Qb \circ (\Xb, \Vb)^{-1}$.
\end{proof}

	\vspace{0.5em}

From \eqref{eq:bt2_integ}, we already know that $\Vb$ is a diffusion process w.r.t. $(\Fb, \Qt)$.
With the integrability result \eqref{eq:DvRho_integ}, we can say more on its dynamics.

\begin{lem} \label{lemm:ReverseSDE}
	The reverse process $(\Xt, \Vt)$ is a diffusion process under $\Qb$,
	or equivalently, the canonical process $(\Xb, \Vb)$ is a diffusion process under the reverse probability $\Qt$.
	Moreover, $\Qt$ is a weak solution to the SDE:
	\begin{equation} \label{eq:ReverseSDE}
	  \begin{aligned}
	    d \Xb_t &= - \Vb_t dt, \\
	    d \Vb_t &= \big( -b(t,  \Xb_t, \Vb_t) +\sigma^2 \nabla_v \log \big( \rho_{T-t}( \Xb_t,  \Vb_t) \big) \big) dt + \sigma d \Wt_t, \quad \Qt \mbox{--a.s.},
	  \end{aligned}
	\end{equation}
	where $\Wt$ is a $(\Fb, \Qt)$--Brownian motion.	
\end{lem}
		
\begin{proof}
	It follows from the Cauchy-Schwarz inequality and \eqref{eq:DvRho_integ} that
	\begin{align*}
		\int_t^T \int_{\R^{2n}} |\nabla_v \rho_s(x,v)| dxdv 
		\le 
		\left(\int_t^T \int_{\R^{2n}} \frac{|\nabla_v \rho_s(x,v)|^2}{\rho_s(x,v)^2}\rho_s(x,v) dxdv \right)^\frac12<\infty,
	\end{align*}
	for all $T>t>0$.
	Together with the Lipschitz assumption on the coefficient $b(t,x,v)$,
	the desired result is a direct consequence of Haussmann and Pardoux \cite[Theorem 2.1]{haussmann1986time}, or Millet, Nualart and Sanz \cite[Theorem 2.3]{millet1989integration}.
\end{proof}

\paragraph{Application to the MFL equation \eqref{eq:MFL}}

	We will apply the above technical results to the MFL equation  \eqref{eq:MFL}.
	First, we know that the MFL SDE \eqref{eq:MFL} has a unique strong solution $(X, V)$.
	We take $m^X_t := \Lc(X_t)$ as an input and define 
	\begin{align} \label{eq:def_b_MFL}
		b(t,x,v) 
		~:=~ 
		D_m F(m^X_t, x) + \gamma v 
		= 
		D_m F_{\circ} (m^X_t, x) + \nabla_x f(x) +  \gamma v.
	\end{align}
	Then $(X, V)$ is also the unique solution of SDE \eqref{eq:SDE} with drift function $b$ defined above.

	\begin{prop}  \hspace{1mm} \label{prop:regularity}
		$\mathrm{(i)}$ Let Assumption \ref{assum:conv_initial}.$\mathrm{(i)}$ hold true.
		Then the function $b(t,x,v)$ defined by \eqref{eq:def_b_MFL} is a continuous function, uniformly Lipschitz in $(x,v)$.
		
		\vspace{0.5em}
		
		\noindent $\mathrm{(ii)}$ Suppose in addition that Assumption \ref{assum:conv_initial}.$\mathrm{(ii)}$ holds true.
		Then $b \in C^{\infty}((0,\infty) \x \R^n \x \R^n)$ and all derivatives of order $k$, for each $k \ge 1$,  are bounded on $(0,T] \x \R^n \x \R^n$ for any $T>0$.
	\end{prop}

\begin{proof}
	$\mathrm{(i)}.$ For a diffusion process $(X, V)$, it is clear that $t \mapsto m^X_t := \Lc(X_t)$ is continuous under the weak convergence topology,
	then $(t,x,v) \mapsto b(t,x,v) := D_mF(m^X_t, x) + \gamma v$ is continuous.
	Moreover, it is clear that $b$ is globally Lipschitz in $(x,v)$ under Assumption \ref{assum:conv_initial}.$\mathrm{(i)}$.
	
	\vspace{0.5em}
	
	\noindent $\mathrm{(ii)}.$ 
	Let us denote
	$$
		b_{\circ}(t_0, x_0) := D_m F (m^X_{t_0}, x_0),
		~~\mbox{so that}~~ 
		b(t_0,x_0,v_0) = b_{\circ}(t_0, x_0)  +  \gamma v_0.
	$$
	Then it is enough to check the differentiability of $b_{\circ}$.
	We claim that, for all $k \ge 1$, one has
	\begin{align} \label{eq:claim_Db}
		\partial^k_t b_{\circ} (t_0,x_0) 
		=
		\E \bigg[ \sum_{i=0}^k \sum_{j=0}^{k-i} \varphi^k_{i,j} \big(m^X_{t_0}, X_{t_0}, V_{t_0}, x_0 \big) X^i_{t_0} V^j_{t_0} \bigg],
	\end{align}
	where $\varphi^k_{i,j}(m^X, x, v, x_0)$ are bounded functions.

	Further, it follows by Lemma \ref{lem:moment} that, 
	under additional conditions in Assumption \ref{assum:conv_initial}.$\mathrm{(ii)}$,
	one has $\E\big[\sup_{0 \le t \le T} (|X_t|^p + |V_t|^p) \big] < \infty$ for all $T> 0$ and $p \ge 1$.
	By the dominated convergence theorem, one has $b_{\circ} \in C^{\infty}((0, \infty) \x \R^n)$ and hence $b \in C^{\infty}((0,\infty) \x \R^n \x \R^n)$, and in particular  all its derivatives of order $k$, for each $k \ge 1$,  are bounded on $(0,T] \x \R^n \x \R^n$ for any $T>0$.

	\vspace{0.5em}
	
	Then it is enough to prove \eqref{eq:claim_Db}.	
	Recall (see e.g.~Carmona and Delarue \cite[Proposition 5.102]{CarmonaDelarueMFGBook1}) that for a smooth function $\varphi: \Pc_2(\R^n) \x \R^n \x \R^n \longrightarrow \R$, 
	one has the It\^o's formula
	\begin{align} \label{eq:Ito_measure}
		d \varphi (m^X_t, X_t, V_t) 
		&= \int_{\R^{2n}} D_m \varphi ( m^X_t, x, X_t, V_t) \cdot v ~m_t(dx, dv) dt + \nabla_x \varphi(m^X_t, X_t, V_t) \cdot V_t dt  \nonumber \\
		&\quad - \nabla_v \varphi(m^X_t, X_t, V_t) \big( D_mF (m^X_t, X_t)  + \gamma V_t \big) dt \nonumber \\
		&\quad + \frac12 \sigma^2 \Delta_v \varphi(m^X_t, X_t, V_t) dt 
			+ \nabla_v \varphi(m^X_t, X_t, V_t) \cdot \sigma dW_t.
	\end{align}
	
	First, for fixed $(t_0, x_0) \in \R_+ \x \R^n$, we set $\varphi^0(m^X, x_0) := D_m F (m^X, x_0)$.
	Recall that the derivative $D^k_m F$ of any order $k \ge 2$ is bounded, 
	then the derivative $D^k_m \varphi^0$ of any order $k\ge 1$ is bounded.

	\vspace{0.5em}
	
	Applying It\^o formula \eqref{eq:Ito_measure} on $\varphi^0(m^X_t, x_0)$, we obtain that
	$$
		\partial_t b_{\circ}(t_0, x_0) 
		=
		\frac{d \varphi^0(m^X_t, x_0)}{d t} \Big|_{t = t_0} 
		= 
		\E \big[ D_m \varphi^0 (m^X_{t_0}, X_{t_0}, x_0) \cdot V_{t_0}  \big] 
		=
		\E \big[ \varphi^1(m^X_{t_0}, X_{t_0}, V_{t_0}, x_0) \big],
	$$
	with 
	$$
		\varphi^1(m^X, x, v, x_0) := D_m \varphi^0 (m^X, x, x_0) \cdot  v,
	$$
	so that Claim \eqref{eq:claim_Db} is true for $k=1$.
	
	\vspace{0.5em}	
	
	Next, applying It\^o formula on $\varphi^1(m^X_t, X_t, V_t, x_0)$,  we obtain that
	$$
		\partial^2_t b_{\circ}(t_0, x_0) 
		=
		\frac{d}{dt} \E \big[ \varphi^1(m^X_t, X_t, V_t, x_0) \big] \Big|_{t= t_0}
		=
		\E \big[ \varphi^2(m^X_{t_0}, X_{t_0}, V_{t_0}, x_0) \big],
	$$
	with
	\begin{align*}
		\varphi^2(m^X, x, v, x_0) 
		&:=
		D_m \varphi^1(m^X, x,v, x_0) \cdot v 
		+
		\nabla_x \varphi^1(m^X, x, v) \cdot v \\
		&\hspace{6mm} - 
		\nabla_v \varphi^1(m^X, x,v) \cdot \big( D_mF(m^X, x) + \gamma v \big)
		+
		\frac12 \sigma^2 \Delta_v \varphi^1(m^X, x, v),
	\end{align*}
	so that Claim \eqref{eq:claim_Db} is true for $k=2$.
	
	\vspace{0.5em}	
	
	By repeating the same arguments with induction, it is easy to deduce that Claim \eqref{eq:claim_Db} is true for all $k\ge 1$.
\end{proof}

\subsection{Proofs for Theorem \ref{thm:Lyapunov} and Corollary \ref{cor:underover}}\label{secproof:Lyapunov}

\begin{proof}[Proof of Theorem \ref{thm:Lyapunov}]

	In the context of \eqref{eq:MFL}, where the drift function $b$ is given by \eqref{eq:def_b_MFL},
	we use $m_t(x,v)$ (rather than $\rho_t(x,v)$) to denote the density function of the marginal distribution of $(X_t, V_t)$.

	\vspace{0.5em}	

	Let us fix $T> 0$, and consider the reverse probability $\Qt$ given before Lemma \ref{lemm:Dvm_integ} with coefficient function $b$ in \eqref{eq:def_b_MFL}.
	Recall also the dynamics of $(\Xb, \Vb)$ under $\Qt$ in \eqref{eq:ReverseSDE}. 
	Applying It\^o's formula on $\log \big(m_{T-t} ( \Xb_t, \Vb_t) \big)$, and then using the Fokker-Planck equation \eqref{eq:FP}, it follows that
	\begin{align}\label{eq:backpathwisecalculus}
		& d \log  \big(m_{T-t}(\Xb_t, \Vb_t) \big) \\ 
		&= \Big\{-\frac{\partial_t m_{T-t}}{m_{T-t}}(\Xb_t, \Vb_t) - \nabla_x \log \big(m_{T-t}(\Xb_t, \Vb_t) \big)  \cdot \Vb_t  + \frac12 \sigma^2  \Delta_v \log \big(m_{T-t}(\Xb_t,\Vb_t) \big) \notag\\
		&~~~~~~~~~~~~~~~~~~~~
		+ \nabla_v  \log \big( m_{T-t}(\Xb_t, \Vb_t) \big) \cdot \big( -b(t, \Xb_t, \Vb_t) + \si^2 \nabla_v \log \big( m_{T-t}(\Xb_t, \Vb_t) \big) \big)
		\Big\} dt\notag \\
		& \qquad + \nabla_v \log \big( m_{T-t}(\Xb_t, \Vb_t) \big) \cdot \sigma dW_t \notag\\
		&= \Big( \!\! -n \gamma +\frac12 \Big| \frac{ \sigma \nabla_v m_{T-t}}{m_{T-t}} (\Xb_t, \Vb_t) \Big|^2 \Big) dt 
		+\nabla_v \log \big( m_{T-t}(\Xb_t, \Vb_t) \big) \cdot \sigma d \Wt_t,
		\quad \Qt\mbox{--a.s.}\notag
	\end{align}
	Notice that $m_t = \Lc(X_t ,V_t) = \Lc^{\Qt}(X_{T-t}, V_{T-t})$, then it follows by \eqref{eq:DvRho_integ} that, for all $t \in (0, T)$,
	\begin{align} \label{eq:dH_t}
	   d H(m_t) 
	   &= d \E^{\Qt} \Big[\log \big(m_t(\Xb_{T-t}, \Vb_{T-t}) \big) \Big] \nonumber \\
	   &= \Big(  n \gamma - \frac12 \E \Big[ \big| \sigma \nabla_v  \log \big( m_t(X_t, X_t) \big) \big|^2 \Big] \Big) dt.
	\end{align}
On the other hand, recall that
  \begin{equation} \label{eq:D_mF}
   F(m) = F_\circ(m) + \E^m[f(X)],
    ~~\mbox{and}~~
    D_m F \big( \Lc(X_t) \big) 
    = D_m F_\circ \big( \Lc(X_t) \big) + \nabla f.
  \end{equation}
By a direct computation, one has
  \begin{equation} \label{eq:dF_0t}
    dF_\circ \big( \Lc(X_t) \big)
    = \E \big[ D_m F_\circ \big( \Lc(X_t), X_t \big) \cdot V_t \big] dt.
  \end{equation}
By It\^o formula and \eqref{eq:D_mF}, one has
  \begin{align} \label{eq:df}
  	& d \Big( f(X_t) + \frac12 |V_t|^2 \Big) \nonumber \\
  	&\quad =\Big( \nabla f(X_t) \cdot V_t - V_t \cdot \big( D_m F(\Lc(X_t), X_t) + \gamma V_t \big) + \frac12 \sigma^2 n \Big) dt + V_t \cdot \sigma dW_t \nonumber \\
  	&\quad =\Big( -  D_m F_\circ(\Lc(X_t), X_t) \cdot V_t  - \gamma |V_t|^2 +\frac12 \sigma^2 n \Big) dt+V_t \cdot \sigma dW_t.
  \end{align}

Combining \eqref{eq:dH_t}, \eqref{eq:dF_0t} and \eqref{eq:df},
we obtain
	\begin{align} \label{eq:dFfinal}
	   d \fF (m_t) 	
	   &= d \Big( F\big( \Lc(X_t) \big) + \frac12 \E\big[|V_t|^2\big] + \frac{\sigma^2}{2\g} H(m_t) \Big) \notag \\
	   &= \E \Big[ - \gamma |V_t|^2 +  \sigma^2 n  - \frac{ \sigma^4}{4 \gamma} \big| \nabla_v \log \big( m_t (X_t, V_t)  \big) \big|^2 \Big] dt.
	\end{align}
Further, by Lemmas \ref{lem:moment} and \ref{lemm:Dvm_integ}, it is clear that $\E \big[ \big| \nabla_v \log \big( m_t(X_t, V_t) \big) \cdot V_t \big| \big] < \infty$
and by integration by parts we have
	\begin{align*}
		\E \Big[ \nabla_v \log \big( m_t(X_t, V_t) \big) \cdot V_t \Big]
		&= \frac12 \int_{\R^n}\int_{\R^n} \Big( \nabla_v m_t(x,v) \cdot \nabla_v |v|^2 \Big) dx dv \nonumber \\
		&= - \frac12 \int_{\R^n}\int_{\R^n} \Big( m_t(x,v) \Delta_v |v|^2 \Big) dx dv
		 = -n.
	\end{align*}	
Together with \eqref{eq:dFfinal}, it follows
   \begin{align*}
   	  d \fF (m_t)  = - \g \dbE\left[ \bigg| V_t + \frac{\si^2}{2\g}  \nabla_v \log \big( m_t (X_t, V_t)  \big) \bigg|^2 \right] dt.
   \end{align*}
\end{proof}

\begin{rem}[Time reversal argument]
$\mathrm{(i)}$  In case that $m\mapsto F_\circ(m)$ is linear, Theorem  \ref{thm:Lyapunov} is  a specific case of Fontbona and Jourdain \cite[Corollary 1.6]{FJ16}. Indeed, when  $m\mapsto F_\circ(m)$ is linear, i.e.~$F_\circ(m) = \int f_\circ(x) m(dx)$, the free energy function can be viewed as a relative entropy $m\mapsto H(m|\mu)$ for some reference measure $\mu$, see Remark \ref{rem:entropyrewrite}. In particular, $\mu$ shall be the invariant measure of the (classical) Langevin dynamics. Define $\Qt^*$ to be the law of the time-reverse Langevin diffusion starting from the invariant distribution $\mu$. In \cite{FJ16}, the authors observed that the likelihood process
 \[ \ell_t(\overline X_t, \overline V_t) := \frac{m_{T-t}}{\mu}(\overline X_t, \overline V_t)\]
  is a $\Qt^*$-martingale. Together with the dynamics of $(\overline X_t, \overline V_t)$ under $\Qt$ and the It\^o's formula, one obtains that
  $d\ell_t (\overline X_t, \overline V_t) = \si \nabla_v \ell_t (\overline X_t, \overline V_t)  d\Wt^*_t$ with  a $\Qt^*$-Brownian motion $\Wt^*$ and further that
  \[d \ell_t\log \ell_t (\overline X_t, \overline V_t) = \frac{\si^2}{2} |\nabla_v \log \ell_t |^2  \ell_t(\overline X_t, \overline V_t)dt + \si (\log\ell_t+1) \nabla_v \ell_t(\overline X_t, \overline V_t) d\Wt^*_t.\]
Finally note that 
\begin{align*}
d H(m_t|\mu) &= -  d \dbE^{\Qt^*} \left[\ell_{T-t}\log \ell_{T-t} (\overline X_{T-t}, \overline V_{T-t}) \right]\\
&= -\int\frac{\si^2}{2} |\nabla_v \log m_t |^2  m_t(x,v) dxdvdt,
\end{align*}
which gives the result of Theorem \ref{thm:Lyapunov}.

	\vspace{0.5em}

	\noindent $\mathrm{(ii)}$ In the general case where $m\mapsto F_\circ(m)$ is nonlinear, we still share a similar backward pathwise calculus for the likelihood in \eqref{eq:backpathwisecalculus} as in the proof of \cite[Theorem 1.4]{FJ16}. More naturally, one might try to apply a forward pathwise calculus on $\log(m_t(X_t, V_t))$ instead of the backward one in \eqref{eq:backpathwisecalculus}.  However, by doing so, we shall face an extra term involving $\frac{\D_v  m_t(X_t, V_t)}{m_t(X_t, V_t)} $. Of course, one might further expect to cancel this term under the expectation by applying the integration by part 
\begin{align*}
\dbE\left[\frac{\D_v  m_t(X_t, V_t)}{m_t(X_t, V_t)}\right] = \int \D_v m_t(x, v) dxdv =0.
\end{align*}
Nevertheless, in order to make it rigorous, one needs some integrability property of the $ \D_v m_t$ term, which seems nontrivial to us. The backward pathwise calculus in \eqref{eq:backpathwisecalculus} avoids this technical difficulty. 
\end{rem}

We now provide the proof of the overdamping limit result in Corollary \ref{cor:underover}. 

\begin{proof}[Proof of Corollary \ref{cor:underover}]
	First, as $\sigma = \sigma_0 \sqrt{\gamma}$, one observes that the constant $\widetilde C$ in \eqref{eq:lowerbound} is independent of $\gamma > 0$.
	Notice further that $\fF(m_0) < \infty$.
	Then it follows by \eqref{eq:lowerbound}, together with  the decay of free energy in Theorem \ref{thm:Lyapunov}, that 
	$$
		\sup_{\gamma > 0} \sup_{t \ge 0} \E \Big[ \big| X^{\gamma}_t \big|^2 + \big| V^{\gamma}_t \big|^2  \Big] < \infty.
	$$	
	
	Next, with $\sigma = \sigma_0 \sqrt{\gamma}$, it follows by \eqref{eq:MFL} together with direct computation that
	$$
		X^{\gamma}_{\gamma t}
		-
		X^{\gamma}_0
		~=~
		- \frac{1}{\gamma} \big(V^{\gamma}_{\gamma t} - V_0 \big)
		- 
		\frac{1}{\gamma} \int_0^{\gamma t}  D_m F\big(\Lc(X^{\gamma}_s), X^{\gamma}_s\big) ds
		+
		\frac{\sigma_0}{\sqrt{\gamma}} W_{\gamma t}.
	$$
	Let us define two processes $W^{\gamma}$ and $Y^{\gamma}$ by
	$$
		W^{\gamma}_t := \frac{1}{\sqrt{\gamma}} W_{\gamma t},
		~~\mbox{and}~
		Y^{\gamma}_t := X^{\gamma}_{\gamma t} + V^{\gamma}_{\gamma t}/\gamma,
		~~\mbox{for all}~t \ge 0,
	$$
	one obtains that $Y^{\gamma}$ is the unique strong solution to the SDE
	\begin{eqnarray} \label{eq:SDE_gamma}
		Y^{\gamma}_t - Y^{\gamma}_0
		&=&
		\int_0^{ t}  D_m F(\Lc(X^{\gamma}_{\gamma s}), X^{\gamma}_{\gamma s}) ds
		+
		\sigma_0 W^{\gamma}_t \nonumber \\
		&=&
		\int_0^t K^{\gamma}_s ds
		+
		\int_0^{ t}  D_m F(\Lc(Y^{\gamma}_s), Y^{\gamma}_s) ds
		+
		\sigma_0 W^{\gamma}_t,
	\end{eqnarray}
	where, by Lipschitz property of $D_mF$ in Assumption \ref{assum:conv_initial}, $K^{\gamma}$ satisfies
	$$
		\E \big[ \big| K^{\gamma}_s \big| \big]
		~\le~
		C \E \big[ \big|Y^{\gamma}_s - X^{\gamma}_{\gamma s} \big| \big]
		~=~
		C \E \big[ |V^{\gamma}_{\gamma s}| \big] / \gamma.
	$$
	As $\sup_{\gamma > 0} \sup_{t \ge 0} \E \big[  |V^{\gamma}_{\gamma s}|^2 \big] < \infty$, one has for all $t \ge 0$,
	$$
		\E \Big[ \int_0^t \big| K^{\gamma}_s \big| ds \Big] \longrightarrow 0,
		~\mbox{as}~ \gamma \longrightarrow \infty.
	$$
	Then one can easily apply (with some trivial adaptation) the standard stability result of SDEs (see e.g.~Jacod and M\'emin \cite[Section 3]{jacod1981weak}) to obtain that
	$Y^{\gamma}$ converges weakly on $\Pc\big(C(\R_+, \dbR^n)\big)$ to the overdamped MFL dynamics  $Y$ defined in \eqref{eq:SDE_Y}, as $\gamma \longrightarrow \infty$.

	\vspace{0.5em}
	
	\noindent Finally, as $\sup_{\gamma > 0} \sup_{t \ge 0} \E \big[ \big| V^{\gamma}_t \big|^2  \big] < \infty$,
	it follows that, for all $t \ge 0$, $X^{\gamma}_{\gamma t} \longrightarrow Y_t$ weakly, as $\gamma \longrightarrow \infty$.
\end{proof}

\subsection{Proof of Theorem \ref{thm:ergodicity}}\label{secproof:ergodicity}

Let $(m_t)_{t\ge 0}$ be the flow of marginal laws of the solution to \eqref{eq:MFL}, given an initial law $m_0\in \Pc_2(\R^n \x \R^n)$. Define $S(t)[m_0]:=m_t$. We shall consider the so-called $w$-limit set:
	\begin{align*}
	   w(m_0):= \left\{\mu\in \Pc_2(\R^{n} \x \R^n):~~\mbox{there exist}~t_k\longrightarrow\infty~\mbox{such that}~~ \cW_1\Big(S(t_k)[m_0],\mu\Big)\longrightarrow 0\right\}.
	\end{align*}
	Let us recall LaSalle's invariance principle for dynamical system. 
\begin{lem}
	Let Assumption \ref{assum:conv_initial} hold true. Then $(S(t))_{t \ge 0}$ is a dynamical system on $\cW_1$ space, i.e.
\begin{enumerate}
\item[$\mathrm{(i)}.$] $S(0)$ is $\mathrm{identity}$;
\item[$\mathrm{(ii)}.$] $S(t)\big(S(t')[\mu] \big) = S(t+t')[\mu]$ for all $\mu$ and $t, t'\ge 0$;
\item[$\mathrm{(iii)}.$] for each $\mu$, $t\mapsto S(t)[\mu]$ is continuous (w.r.t. the weak convergence topology);
\item[$\mathrm{(iv)}.$] for each $t\ge 0$, $\mu\mapsto S(t)[\mu]$ is $\cW_1$-continuous.
\end{enumerate}
\end{lem}
\begin{proof}
The properties {\rm(i), (ii)} are trivial. The continuity in {\rm (iii)} follows from the standard stability result for the McKean-Vlasov SDE. Here we prove the property {\rm (iv)}.
Under the upholding assumptions, it follows from Lemma \ref{lem:moment} that 
		\begin{align*}
	\sup_{t\le T}\int \left( |x|^2 + |v|^2\right) m_t(dx, dv)<\infty.
	\end{align*}
	Together with the fact that $t\mapsto m_t$ is continuous w.r.t.~the weak convergence topology, we deduce that  $t\mapsto S(t)$ is continuous with respect to the $\cW_1$-topology.
\end{proof}

\begin{prop}\label{prop:invprin}[Invariance Principle for dynamical system]
	Let Assumption \ref{assum:conv_initial} hold true. Then the set $w(m_0)$ is nonempty, $\cW_1$-compact and  invariant, that is, 
	\begin{enumerate}
		\item for any $\mu\in w(m_0)$, we have $S(t)[\mu]\in w(m_0)$ for all $t\ge 0$;
		\item for any $\mu\in w(m_0)$ and all $t\ge 0$, there exists $\mu'\in w(m_0)$ such that $S(t)[\mu'] =\mu$.
	\end{enumerate}
\end{prop}

\begin{proof}
	Due to Theorem \ref{thm:Lyapunov},  we know that $\{\fF(m_t)\}_{t\ge 0}$ is  bounded. Together with the lower bound of the energy function  \eqref{eq:lowerbound}, we obtain 
	\begin{align*}
		C:=\sup_{t\ge 0} \dbE\big[|X_t|^2 + |V_t|^2 \big] <\infty.
	\end{align*}
	Therefore $\big(S(t)[m_0]\big)_{t\ge 0} = (m_t)_{t\ge 0}$ lies in the $\cW_1$-compact set 
	  $$ \left\{ m\in \Pc_2(\R^n \x \R^n) : \int \left( |x|^2 + |v|^2\right) m (dx, dv)\le C \right\}. $$
	Then the desired invariance result follows from \cite[Theorem 4.3.3]{Henry81}.
\end{proof}

\begin{lem}\label{lem:limit_necessary}
	Let Assumption \ref{assum:conv_initial} hold true. 
	Then, every $m^*\in w(m_0)$ has a density and we have
	\begin{align}\label{eq:limit_necessary}
		v + \frac{\si^2}{2\g} \nabla_v \log \big( m^*(x,v) \big) = 0,\quad {\rm Leb^{2n}-}a.e.
	\end{align}
\end{lem}

\begin{proof}
	Let $m^*\in w(m_0)$ and denote by $(m_{t_k})_{k\in \dbN}$ the subesquence converging to $m^*$ in $\cW_1$. 
	
	\ms 
	\noindent {\bf Step 1.}\quad We first prove that there exists a sequence $\d_i \longrightarrow 0$ such that 
	 \begin{align} \label{eq:deltailimit}
	 	\liminf_{k\longrightarrow\infty}  \dbE\bigg[ \Big| V_{t_k+\d_i} + \frac{\si^2}{2\g}  \nabla_v \log \big( m_{t_k+\d_i} (X_{t_k+\d_i}, V_{t_k+\d_i})  \big) \Big|^2 \bigg]  = 0, \quad \mbox{for all $i\in \dbN$}.
	 \end{align}
	Suppose the contrary. Then we would have for some $\d >0$
	\begin{align*}
		0 &< \int_0^\d \liminf_{k\longrightarrow\infty}  \dbE\bigg[ \Big| V_{t_k+s} + \frac{\si^2}{2\g}  \nabla_v \log \big( m_{t_k+s} (X_{t_k+s}, V_{t_k+s})  \big) \Big|^2 \bigg]  ds \\
		&\le \liminf_{k\longrightarrow\infty}  \int_0^\d \dbE\bigg[ \Big| V_{t_k+s} + \frac{\si^2}{2\g}  \nabla_v \log \big( m_{t_k+s} (X_{t_k+s}, V_{t_k+s})  \big) \Big|^2 \bigg] ds,
	\end{align*}
	where the last inequality is due to Fatou's lemma. This is a contradiction against  Theorem \ref{thm:Lyapunov} and  the fact that $\fF$ is bounded from below. 
	\ms
	
	\noindent {\bf Step 2.}\quad Denote by $t_k^i: = t_k+\d_i$ and $m^*_t: =S(t)[m^*]$. 
	Note that
	\begin{align*}
	 &\lim_{k\longrightarrow\infty}\cW_1\Big(m_{t_k}, m^* \Big) = 0  \\
	 &\quad \Longrightarrow \quad 
 	  \lim_{k\longrightarrow\infty}\cW_1\Big(m_{t_k^i}, m^*_{\d_i}\Big) = \lim_{k\longrightarrow\infty}\cW_1\Big(S(\d_i)[m_{t_k}], S(\d_i)[m^*]\Big) =0.
	\end{align*}

	Now fix $i\in \dbN$.  Due to Theorem \ref{thm:Lyapunov} and the fact that $\{F(m_t) + \frac12\dbE[|V_t|^2]\}_{t\ge 0}$ is bounded from below, the set $\{H(m_{t_k^i})\}_{k\in \dbN}$ is uniformly bounded. Therefore the densities $(m_{t_k^i})_{k\in \dbN}$ are uniformly integrable with respect to Lebesgue measure, and thus $m^*$ has a density.   Note that
	  \begin{align*}
	  	&\dbE\bigg[ \Big| V_{t_k^i} + \frac{\si^2}{2\g}  \nabla_v \log \Big( m_{t_k^i}\big(X_{t_k^i}, V_{t_k^i}\big)  \Big) \Big|^2 \bigg] \\
	  	&= \frac{\si^4}{4\g^2} \int_{\dbR^{2n}}\frac{ \left| \nabla_v \big(m_{t_k^i}(x,v) e^{\frac{\g}{\si^2} |v|^2}\big)\right|^2}{ m_{t_k^i}(x,v) e^{\frac{\g}{\si^2} |v|^2} }e^{-\frac{\g}{\si^2} |v|^2} dv dx.
	  \end{align*}
	 It is noteworthy that $ \int_{\dbR^{n}}\frac{ \left| \nabla_v \big(m_{t_k^i}(x,v) e^{\frac{\g}{\si^2} |v|^2}\big)\right|^2}{ m_{t_k^i}(x,v) e^{\frac{\g}{\si^2} |v|^2} }e^{-\frac{\g}{\si^2} |v|^2} dv$ is the relative Fisher information of the law $m_{t_k^i}(x, \cdot)$ with respect to the  Gaussian distribution $\mu^*_v := \cN(0,\frac{\si^2}{2\g} I_n)$.
	 Define the function $h_k^i(x,v):= m_{t_k^i}(x,v) e^{\frac{\g}{\si^2} |v|^2}$. By logarithmic Sobolev inequality for the Gaussian distribution we obtain
	  \begin{align*}
	     \int \left(\int h_k^i \log ( h_k^i ) d \mu^*_v  - \Big( \int h_k^i d\mu^*_v \Big) \log \Big( \int h_k^i d\mu^*_v \Big) \right)dx
	     \le C \int \frac{|\nabla_v h_k^i|^2}{h_k^i} d\mu^*_v dx.
	  \end{align*}
	Together with \eqref{eq:deltailimit} we obtain
	  \begin{align} \label{ineq:tk}
	  	0 &= \lim_{k\longrightarrow\infty}  \dbE\Big[ \big| V_{t_k^i} + \frac{\si^2}{2\g}  \nabla_v \log \big( m_{t_k^i} (X_{t_k^i}, V_{t_k^i})  \big) \big|^2 \Big] \notag \\
	  	  &\ge C \limsup_{k\longrightarrow\infty} \int \left(\int h_k^i \log (h_k^i)  d \mu^*_v  - \Big( \int h_k^i d\mu^*_v \Big) \log \Big( \int h_k^i d\mu^*_v \Big) \right)dx.
	  \end{align}
	Since $\int h_k^i d\mu^*_v = \int m_{t_k^i} dv = m^X_{t_k^i}$, we further have
	  \begin{align*}
	  	 0 &\ge C  \limsup_{k\longrightarrow\infty} \int \left( m_{t_k^i} \log ( h_k^i ) -  m_{t_k^i} \log \big( m^X_{t_k^i} \big) \right)dvdx \\
	  	  &= C \limsup_{k\longrightarrow\infty} \int \left( m_{t_k^i} \log \frac{m_{t_k^i}}{m^X_{t_k^i}e^{-\frac{\g}{\si^2}|v|^2}}\right)dvdx \\
	  	  & = C \limsup_{k\longrightarrow\infty} H \left(m_{t_k^i} \bigg| m^X_{t_k^i}\times \cN\Big(0, \frac{\si^2}{2\g} I_n\Big)\right) \\
	  	  &\ge C H \left(m^*_{\d_i} \bigg| m^{*,X}_{\d_i} \times \cN\Big(0, \frac{\si^2}{2\g} I_n\Big)\right).
	  \end{align*}
	The last inequality is due to the lower semi-continuity of the relative entropy in weak topology. Finally, since $\lim_{i\to\infty}\cW_1(m^*_{\d_i},m^*)=0$,  we get 
	  \begin{align*}
	  	H \left(m^* \Big| m^{*, X} \times \cN\Big(0, \frac{\si^2}{2\g} I_n\Big)\right)=0
	  \end{align*} 
	  and thus
	  \begin{align*}
	  	m^* = m^{*,X} \times \cN\Big(0, \frac{\si^2}{2\g} I_n\Big).
	  \end{align*}
	This immediately implies \eqref{eq:limit_necessary}.
\end{proof}

\begin{lem} \label{lem:mstar}
	Let Assumption \ref{assum:conv_initial} hold true. Then, each $m^*\in w(m_0)$ is equivalent to Lebesgue measure.
\end{lem}

\begin{proof}
	By the invariant principle we may find a probability measure $m^\circ\in w(m_0)$ such that $m^* = S(t)[m^\circ]$ for a fixed $t>0$. 
	Then the desired result follows from Lemma \ref{lem:exist_density}. 
\end{proof}

\vspace{0.5em}

	Note that the necessary condition \eqref{eq:limit_necessary} for $m^*\in w(m_0)$ is not enough to identify $m^*$ as the invariant measure as required in Theorem \ref{thm:ergodicity}.  
	We are going to trigger the invariance principle to complete the proof.

\vspace{0.5em}

\begin{proof}[Proof of Theorem \ref{thm:ergodicity}]

	Since $F$ is convex, so that $\fF$ is strictly convex, so that the optimization problem $\inf_{m\in \Pc_2(\R^{n} \x \R^n)} \fF(m)$ has a unique minimizer $\underline m$,
	which is given by \eqref{eq:FOCeq}.
	Therefore, to conclude the proof, it is enough to check that any $m^* \in w(m_0)$ satisfies \eqref{eq:FOCdensity}.

	\vspace{0.5em}

	Let $m^*\in w(m_0)$ and define $m^*_t:= S(t)[m^*]$ for all $t\ge 0$. Denote by $(X^*_t, V^*_t)_{t\ge 0}$ the solution to the MFL equation \eqref{eq:MFL} with initial distribution $m^*$. Take a test function $h\in C^1(\dbR^n)$ with compact support.  It follows from It\^o's formula that  
	\begin{align} \label{eq:dynamic_fourier}
		d V^*_t h(X^*_t)
		 &= \Big(- h(X^*_t) \big(D_m F(m^{X,*}_t, X^*_t)+\g V^*_t\big) + \big(\nabla_x  h(X^*_t) \cd V^*_t\big) V^*_t\Big) dt \nonumber \\
		 &\qquad + \si h(X^*_t) dW_t,
	\end{align}
	where $m^{X,*}_t$ denotes the pushforward measure of $m^*_t$ under the map $(x,v) \longmapsto x$.
	By the invariance principle, we have $m^*_t \in w(m_0)$ for all $t\ge 0$, and by Lemma \ref{lem:limit_necessary} we have
	\begin{align*}
	  v + \frac{\si^2}{2\g} \nabla_v \log \big( m^*_t(x,v) \big) = 0, \quad {\rm Leb^{2n}}-a.e.	
	\end{align*}
So there exists a measurable function $(t,x)\mapsto \widehat m_t(x)$ such that $m^*_t(x,v) = e^{-\frac{\g}{\si^2}|v|^2}\widehat m_t(x)$. In particular, we observe that for each $t\ge 0$,  the random variables $X^*_t,V^*_t$ are independent and $V^*_t$ follows the Gaussian distribution $\cN(0, \frac{\si^2}{2\g}I_n)$. Taking expectation on both sides of \eqref{eq:dynamic_fourier}, we obtain
   \begin{align} \label{eq:exp_fourier}
   	 0 &= \dbE \Big[ - h(X^*_t) \big(D_m F(m^{X,*}_t, X^*_t)+\g V^*_t\big) + \big(\nabla_x  h(X^*_t) \cd V^*_t\big)  V^*_t \Big] \notag \\
   	   &= \dbE \Big[ - h(X^*_t)  D_m F(m^{X,*}_t, X^*_t)+  \frac{\si^2}{2\g} \nabla_x  h(X^*_t)   \Big],
	   ~~ \mbox{for a.e.~$t > 0$}.
   \end{align}
	Observe that
   \begin{align*}
   	  \dbE \left[\nabla_x  h(X^*_t) \right] = C_t\int_{\dbR^n}  \nabla_x h(x) \widehat  m_t(x)dx
   	  = -C_t \int_{\dbR^n} h(x) \nabla_x \widehat m_t(x) dx,
   \end{align*}
   where $C_t$ is the normalization constant such that $C_t \widehat m_t$ is a density function, and $\nabla_x \hat m_t$ is the weak derivative in sense of distribution. 
Together with \eqref{eq:exp_fourier} we have
   \begin{align*}
   	 \int_{\dbR^n} h(x) \Big(-  D_m F(m^{X,*}_t, x) \widehat m_t(x) - \frac{\si^2}{2\g}  \nabla_x \widehat m_t(x) \Big) dx =0,
	 ~~\mbox{for a.e.}~t > 0.
   \end{align*}
		Notice that for all $N > 0$, the Hilbert space $L^2([-N, N]^n)$ has a countable (smooth functional) basis, this allows us to consider
		arbitrary $h$ in a countable space to obtain that
   \begin{align*}
   	 D_m F(m^{X,*}_t, x) + \frac{\si^2}{2\g} \nabla_x \log \big( \widehat m_t(x) \big) =0, \quad m^*_t\mbox{-a.s. for a.e. $t > 0$}.
   \end{align*}
By Lemma \ref{lem:mstar}, $m^*_t$ is equivalent to Lebesgue measure, and thus we have
   \begin{align*}
   	 \begin{cases}
   	 D_m F(m^{X,*}_t, x) + \frac{\si^2}{2\g} \nabla_x \log \big( m^*_t(x,v) \big) =0, \\
   	 v + \frac{\si^2}{2\g} \nabla_v \log \big( m^*_t(x,v) \big) =0,
   	 \end{cases}
   	 \mbox{for all $(x,v)\in \dbR_+ \times \dbR^{2n}$}, \,\, \mbox{for a.e.}~ t > 0.
   \end{align*}
	Therefore,  by Lemma \ref{lem:FOC}, one has $m^*_t = \underline m$.
	Taking into account that $\lim_{t\longrightarrow 0}\cW_1(m^*_t, m^*) =0$, we obtain $m^* = \underline m$. 
	This is enough conclude that $w(m_0) = \{ \underline m\}$, and thus $\lim_{t\longrightarrow\infty} \cW_1(m_t,\underline m ) =0$.
\end{proof}

\subsection{Exponential ergodicity given small mean-field dependence}\label{secproof:exp}

Under Assumption \ref{assum:conv_initial} and Assumption \ref{assum:exp}, we consider the following equation:
\begin{align}
\begin{cases} \label{eqn:MFLcoupling}
d X_t = V_t dt, \\
d V_t = -\big(b(m^X_t, X_t) + \lambda X_t + \g V_t \big) dt + \si dW_t,
\end{cases}
\end{align}
where $b: \Pc_2(\dbR^n)\times \R^n\longrightarrow \dbR$ is Lipschitz in the variable $x$ 
\begin{align*}
	\big|b(m,x)-b(m,x')\big| \leq L^x |x-x'|,
\end{align*}
and for any $\varepsilon>0$ there exists $M>0$ such that for any $m,m'\in\Pc_2(\dbR^n)$ 
\begin{equation*}
\big|b(m,x)-b(m',x')\big| \leq \varepsilon|x-x'|, \quad\mbox{whenever } |x-x'|\geq M, 
\end{equation*}
and for each $x\in\dbR^n$
\begin{equation}  \label{eq:AssumpLipMeasure}
\big|b(m, x) - b(m',x)\big|
\le   \iota \cW_1(m,m'),
\end{equation}
where the constant $\iota>0$ is small enough, satisfying the quantitative condition \eqref{eq:smalliota} below. In the sequel, for the simplicity of notation, we write $b$ instead $D_mF_o$. Besides, we notice that for this part the special form $D_mF_o$, the intrinsic derivative of a function, is not necessary.

\subsubsection{Reflection-Synchronous Coupling}
We are going to show the contraction result in Theorem \ref{thm:Contraction} via the coupling technique.   Let $(X,V)$ and $(X',V')$ be the two solutions of \eqref{eqn:MFLcoupling} driven by the Brownian motions $W$ and $W'$, respectively.  Define $\delta X = X-X'$ and $\delta V = V-V'$.
We introduce the change of variable 
$$ P_t:= \delta V_t + \gamma \delta X_t. $$
Then, the processes $\delta X$ and $P$ satisfy the following stochastic differential equations
\begin{align*}
d\delta X_t &= (P_t - \gamma\delta X_t)dt, \\
dP_t &= -\big(\delta b_t + \lambda\delta X_t\big)dt + \sigma d{\delta W_t},
\end{align*}
where $\delta W = W-W'$ and $\delta b_t:= b\big(m_t^X,X_t\big)-b\big(m_t^{X'},X'_t\big)$.

\begin{rem}\label{rem:EGZ}
	We shall apply the reflection-synchronous coupling following the blueprint in Eberle, Guillin and Zimmer \cite{EGZ19}, of which the main idea is to separate the space $\dbR^n\times \dbR^n$ into two parts: 
	  \begin{enumerate}
	     \item[$\mathrm{(i)}.$] $(\d X_t, P_t)$ locates in a compact set; 
	     \item[$\mathrm{(ii)}.$] $|\delta X_t|+\eta |P_t| $ is big enough, where the constant $\eta$ is to be determined. 
	  \end{enumerate}
	As in \cite{EGZ19} we are going to apply the reflection coupling on the area {\rm (i)} and the synchronous coupling on  the area {\rm (ii)}.  However,  note that in \cite{EGZ19} the argument for the contraction on the area {\rm (ii)} relies on a Lyapunov function, which can no longer play its role in the mean-field context. Therefore, we are going to construct another function (the function $G$ in \eqref{eq:funpsi}) which decays exponentially on the area {\rm (ii)}.
\end{rem}

Recall the definitions 
$$ r_t := |\d X_t|, \quad u_t := |P_t|, \quad z_t:= \d X_t \cd P_t. $$
Let $\xi>0$. For technical reason we shall also apply the  synchronous coupling on the area $u_t <\xi$, and
eventually we will let $\xi\downarrow 0$.
In order to couple the two processes $(X, V), (X', V')$, we consider two Lipschitz continuous functions $\rc, \syc:\R^{2n}\longrightarrow [0,1]$ such that $\rc_t^2 + \syc^2_t\equiv 1$, 
\begin{align*}
\rc(\delta X_t, P_t) =
\begin{cases}
0, &\mbox{when}~~ u_t = 0 ~~\mbox{or}~~  r_t +\eta u_t \ge 2M+\xi, \\
1, &\mbox{when}~~ u_t \ge \xi ~~\mbox{and}~~r_t +\eta u_t  \le 2M.
\end{cases}
\end{align*}
The values of the constants $\eta$, $M\in(0,\infty)$ will be determined later. 
Define 
\begin{align*}
e_t^x:=\begin{cases}
\frac{\d X_t}{|\d X_t|}, & \mbox{if } \d X_t\neq 0, \\
~0, &  \mbox{if } \d X_t = 0,
\end{cases}  
\quad \mbox{ and }\quad
e_t^p:=\begin{cases}
\frac{P_t}{|P_t|}, & \mbox{if } P_t\neq 0, \\
~0, &  \mbox{if } P_t = 0.
\end{cases}. 
\end{align*}
With two independent Brownian motions $W^{\rc}$ and $W^{\syc}$ we consider the following coupling
\begin{align*}
\begin{cases}
dW_t = \rc(\delta X_t,P_t)dW_t^{\rc} + \syc(\delta X_t,P_t)dW_t^{\syc}, \\
dW'_t = \rc(\delta X_t,P_t)\big(I_n-2e^p_t(e^p_t)^\top\big)dW_t^{\rc}+\syc(\delta X_t,P_t)dW_t^{\syc},
\end{cases}
\end{align*}
in particular we have $d\delta W_t = 2\rc(\delta X_t,P_t)e^p_t(e^p_t)^\top dW_t^{\rc}.$
By L\'evy characterization, the process $B_t:=(e^p_t)^\top W_t^{\rc}$ is a one-dimensional Brownian motion. 
For the sake of simplicity, denote $\rc_t:=\rc(\delta X_t,P_t)$.
We notice that the Lipschitz continuity of the functions $\rc, \syc$ ensures the existence and uniqueness of the coupling process. 

To conclude, with the reflection-synchronous coupling, the processes $\delta X$ and $P$ satisfy  the following stochastic differential equations
\begin{equation}  \label{eqn:MFLcoupling1}
\begin{aligned}
d\delta X_t &= (P_t - \gamma\delta X_t)dt, \\
dP_t &= -\big(\delta b_t + \lambda\delta X_t\big)dt + 2\sigma\rc_t e_t^pd{ B_t}.
\end{aligned}
\end{equation} 

\subsubsection{The auxiliary function}\label{sec:G}

As reported in Remark \ref{rem:EGZ}, the main novelty of our contraction result is to construct a function exponentially decaying along the process \eqref{eqn:MFLcoupling1}  when $r +\eta u$ is sufficiently large. In this subsection we are going to construct the auxiliary function according to the different settings. 

First, it follows from \eqref{eqn:MFLcoupling1} and It\^o's formula that
\begin{align*}
dr^2_t &= 2\delta X_t\cdot d\delta X_t = 2\delta X_t\cdot 
        = 2\left(\delta X_t\cdot P_t - \gamma|\delta X_t|^2 \right)dt 
        = 2\left(z_t -\gamma r_t^2 \right)dt, \\ 
du^2_t &= 2P_t\cdot dP_t +  I_n: d\langle P\rangle_t 
        = -2\left(\delta b_t\cdot P_t + \lambda z_t\right)dt + 4\sigma \rc_t u_t dB_t + 4\sigma^2\rc_t^2 dt, \\
dz_t & = \delta X_t\cdot dP_t + P_t\cdot d\delta X_t 
       = \left(-\delta b_t\cdot\delta X_t - \lambda r_t^2 + u_t^2 - \gamma z_t\right)dt + 2\sigma \rc_t\frac{z_t}{u_t}dB_t, \\
dr_t & = \frac{1}{|\delta X_t|}\delta X_t\cdot d\delta X_t = \big(e^x_t\cdot P_t - \gamma r_t\big)dt, \\
du_t & = \frac{1}{2}\frac{1}{\sqrt{u_t^2}}du^2_t + \frac{1}{2}\left(-\frac{1}{4}\frac{1}{\sqrt{u_t^2}^3}\right)d\langle u^2\rangle_t
       = -\big(\delta b_t\cdot e^p_t + \lambda e^p_t\cdot\delta X_t \big)dt + 2\sigma\rc_t dB_t.
\end{align*}
We write the dynamics of $\big(z,r^2,u^2\big)$ in the following way
\begin{align} \label{eq:dzru=Adt}
d\begin{pmatrix}
z_t \\
r^2_t \\
u^2_t 
\end{pmatrix}
& = A\begin{pmatrix}
z_t \\
r^2_t \\
u^2_t  
\end{pmatrix}dt 
+ \begin{pmatrix}
-\delta b_t\cdot\delta X_t \\
0 \\
-2\delta b_t\cdot P_t + 4\rc_t^2\sigma^2
\end{pmatrix}dt
+ \begin{pmatrix}
2\frac{z_t}{u_t}  \\
0 \\
4 u_t
\end{pmatrix}\sigma\rc_t dB_t
\end{align}
with the matrix
\begin{align*}
A:= \begin{pmatrix}
-\gamma  &   -\lambda    &  1   \\
2      &  -2\gamma                          &  0   \\
-2\lambda & 0 &  0
\end{pmatrix}.
\end{align*}

\begin{rem}
	As we will show later, the value of $\d b_t$ is small whereas $r_t + \eta u_t$ is big enough. Therefore, the coupling system is nearly linear and  its contraction rate mainly depends on the matrix $A$. 
\end{rem}

The eigenvalues of $A$ solve the equation:
\begin{align*}
0 &= (\zeta+ \g) (\zeta+2\gamma)\zeta + 2\l(\zeta+2\gamma) + 2\l\zeta 
= (\zeta+\gamma)(\zeta^2 + 2\gamma\zeta + 4\lambda).
\end{align*}	
We divide the discussion into two cases, based on the different values of $\lambda$ and $\gamma$.
\begin{itemize}
	\item[{\rm (a)}] If $\lambda <\frac{\gamma^2}{4}$, the matrix has three different negative eigenvalues 
	$$ \zeta = -\gamma, \quad \zeta = -\gamma + \sqrt{\gamma^2-4\lambda}, \quad \zeta = -\gamma - \sqrt{\gamma^2-4\lambda}, $$
	in particular,  it can be diagonalized. More precisely, we have  $QA=\Lambda Q$ with
	the  transformation matrix 
	\begin{align*}
	Q:=
	\begin{pmatrix}
	-\gamma  &  \lambda  & 1 \\
	-\gamma +\sqrt{\gamma ^2-4 \lambda } & \frac{1}{2} \left( \gamma^2 - 2\lambda -\gamma\sqrt{\gamma^2-4 \lambda }\right) & 1 \\
	-\gamma -\sqrt{\gamma ^2-4 \lambda } & \frac{1}{2} \left( \gamma^2 - 2\lambda +\gamma\sqrt{\gamma^2-4 \lambda }\right) & 1 \\
	\end{pmatrix}
	\end{align*}
	and the diagonal  matrix $\Lambda = {\rm diag}\big(-\gamma, -\gamma +\sqrt{\gamma ^2-4 \lambda }, -\gamma-\sqrt{\gamma ^2-4 \lambda} \big)$.
	Multiply $Q$ on both sides of \eqref{eq:dzru=Adt} and obtain
	\begin{align} \label{eq:Qdynamics}
		d Q \begin{pmatrix}
		z_t \\
		r^2_t \\
		u^2_t 
		\end{pmatrix}
		= \Lambda Q \begin{pmatrix}
		z_t \\
		r^2_t \\
		u^2_t  
		\end{pmatrix}dt 
		&+ Q \begin{pmatrix}
		-\delta b_t\cdot\delta X_t \\
		0 \\
		-2\delta b_t\cdot P_t + 4\rc_t^2\sigma^2
		\end{pmatrix}dt \nonumber \\
		&+ Q \begin{pmatrix}
		2\frac{z_t}{u_t}  \\
		0 \\
		4 u_t
		\end{pmatrix}\sigma\rc_t dB_t.
	\end{align}
	Further note that 
	\begin{equation}
	\begin{aligned}\label{eq:quadratic}
	&\big(-\gamma +\sqrt{\gamma ^2-4 \lambda }\big) z_t + \frac{1}{2} \left( \gamma^2 - 2\lambda -\gamma\sqrt{\gamma^2-4 \lambda }\right) r^2_t + u^2_t \\
	&\quad = \bigg| \frac{\gamma -\sqrt{\gamma ^2-4 \lambda } }{2} \d X_t - P_t\bigg|^2,\\
	&\big(-\gamma -\sqrt{\gamma ^2-4 \lambda }\big) z_t +  \frac{1}{2} \left( \gamma^2 - 2\lambda +\gamma\sqrt{\gamma^2-4 \lambda }\right)  r^2_t +u^2_t \\
	&\quad = \bigg| \frac{\gamma +\sqrt{\gamma ^2-4 \lambda } }{2} \d X_t - P_t\bigg|^2.
	\end{aligned}
	\end{equation}
	Now define the function $G$:
	\begin{equation} \label{eq:defG}
		G(x , p)  : =  \bigg| \frac{\gamma -\sqrt{\gamma ^2-4 \lambda } }{2} x- p\bigg|^2 +  \bigg| \frac{\gamma +\sqrt{\gamma ^2-4 \lambda } }{2} x - p \bigg|^2.
	\end{equation}

	Denote by $G_t : = G(\d X_t , P_t)$. Together with \eqref{eq:Qdynamics}, \eqref{eq:quadratic}, we obtain
	\begin{align*}
	d G_t \leq &-\big(\gamma - \sqrt{\gamma^2-4\lambda}\big) G_t dt \\
	&+ \begin{pmatrix}  
	0 \\ 
	1 \\ 
	1 \\	
	\end{pmatrix}^\top \!\!\! Q 
	\left\{\begin{pmatrix}
	-\delta b_t\cdot\delta X_t \\
	0 \\
	-2\delta b_t\cdot P_t + 4\rc_t^2\sigma^2
	\end{pmatrix}dt \! 
	+ \begin{pmatrix}
	2\frac{z_t}{u_t}  \\
	0 \\
	4 u_t
	\end{pmatrix}\sigma\rc_t dB_t\right\}.
	\end{align*}

	\item[{\rm (b)}] If $\lambda >\frac{\gamma^2}{4}$, the eigenvalues of $A$ are 
	$$ \zeta = -\g, \quad \zeta = -\g + i \sqrt{4\l-\g^2}, \quad \zeta = -\g - i \sqrt{4\l-\g^2}.  $$ 
	We have   $QA=\Lambda Q$  with the transformation matrix 
	\begin{align*}
	Q: = \begin{pmatrix}
	-\g  & \l  & 1 \\
	4\l & -\l\g & -\g \\
	0  & \l\sqrt{4\l -\g^2}  &  - \sqrt{4\l -\g^2} 
	\end{pmatrix}
	\end{align*}
	and the standard form 
	\begin{align*}
	\Lambda := \begin{pmatrix}
	-\gamma & 0 & 0  \\
	0 & -\gamma & -\sqrt{4\lambda -\gamma^2} \\
	0 & \sqrt{4\lambda -\gamma^2} &  -\gamma 
	\end{pmatrix} .
	\end{align*}
	Multiplying $Q$ on both sides of \eqref{eq:dzru=Adt}, we again obtain \eqref{eq:Qdynamics}.  Now note that
	\begin{align*}
		-\g z_t + \l r^2_t  + u^2_t = \left| \frac{\g}{2} \d X_t - P_t\right|^2 + \left(\l - \frac{\g^2}{4}\right) |\d X_t|^2 =:G(\d X_t, P_t).
	\end{align*}
	Together with \eqref{eq:Qdynamics}, we obtain
	\begin{align*}
	d G_t = -\g G_t dt +  \begin{pmatrix}  
	1 \\ 
	0 \\ 
	0	
	\end{pmatrix}^\top \!\! Q
	\left\{\begin{pmatrix}
	-\delta b_t\cdot\delta X_t \\
	0 \\
	-2\delta b_t\cdot P_t + 4\rc_t^2\sigma^2
	\end{pmatrix}dt
	+ \begin{pmatrix}
	2\frac{z_t}{u_t}  \\
	0 \\
	4 u_t
	\end{pmatrix}\sigma\rc_t dB_t\right\}.
	\end{align*}
\end{itemize}

By defining 
\begin{align*}
\overline{\gamma} :=
\begin{cases}
\gamma - \sqrt{\gamma^2-4\lambda},   & \mbox{if }\gamma^2>4\lambda \\
\gamma,  & \mbox{if }\gamma^2<4\lambda
\end{cases} 
,\quad
\overline{Q} :=
\begin{cases}
\begin{pmatrix}0 & 1 & 1\end{pmatrix}  Q , & \mbox{if }\gamma^2>4\lambda \\
\begin{pmatrix}1 & 0 & 0\end{pmatrix}  Q , & \mbox{if }\gamma^2<4\lambda
\end{cases},
\end{align*}
we have
\begin{align}  \label{ineq:dGtleq-gamma}
d G_t \le - \overline\g G_t dt 
+ \overline Q \left\{\begin{pmatrix}
-\delta b_t\cdot\delta X_t \\
0 \\
-2\delta b_t\cdot P_t + 4\rc_t^2\sigma^2
\end{pmatrix}dt
+ \begin{pmatrix}
2\frac{z_t}{u_t}  \\
0 \\
4 u_t
\end{pmatrix}\sigma\rc_t dB_t\right\}.
\end{align}

\ms
Finally, notice that in each case the function $G$  is a quadratic form and is coercive, that is,

\begin{lem} \label{lem:uniformelliptic}
There exist $C_{\scriptscriptstyle G},\l_{\scriptscriptstyle G} > 0$ such that 
 $$  \l_{\scriptscriptstyle G} \big(r^2_t + u^2_t\big) \leq G_t \leq C_{\scriptscriptstyle G}\big(r^2_t + u^2_t\big). $$
\end{lem}
\begin{proof}
In both cases, the functions $G$ can be written in the form: 
  $$ G_t = \left| \Sigma \begin{pmatrix}\d X_t \\ P_t\end{pmatrix}\right|^2, $$
   where the matrices $\Si$ are of full rank in both cases. 
Denote by $\l_{\scriptscriptstyle G}$ the smallest eigenvalue of the matrix $\Si^\top \Si$. 
Clearly $\l_{\scriptscriptstyle G} >0$. 
Then we have $G_t \ge \l_{\scriptscriptstyle G}(r^2_t + u^2_t)$. 
Taking $C_{\scriptscriptstyle G}:=\|\Sigma\|^2$ implies the second inequality. 
\end{proof}

\begin{rem}
Careful readers have noticed that we did not discuss the case $\l =\frac{\g^2}{4}$. Indeed, in this case one may extract $\e>0$ from $\l$ and define the new $\tilde\l: =\l -\e <\frac{\g^2}{4}$. Provided that $\e$ is small enough, it will not cause trouble to the following analysis.
\end{rem}

\ms
\begin{rem}\label{rem:opt_rate}
In case $b=0$, the contraction result can directly follow from the synchronous coupling, i.e. $\rc_t \equiv 0$. Since $\l_{\scriptscriptstyle G} (r^2_t + u^2_t) \le G_t \le C_{\scriptscriptstyle G}  (r^2_t+u^2_t)$, it follows from \eqref{ineq:dGtleq-gamma} that
\begin{align*}
	\cW_2(m_t, m'_t) \le \sqrt{\frac{C_{\scriptscriptstyle G}}{\l_{\scriptscriptstyle G}}}  e^{-\frac{\overline \g}{2}t} \cW_2(m_0, m'_0).
\end{align*}
On the other hand, it follows from \cite[Theorem 6.4]{Pav14} that the spectral gap of the operator 
\begin{align*}
	-\mathfrak{L}  := -  v\cd\nabla_x  + \l x \cd \nabla_v  - \g (\D_v -v\cd\nabla_v)  
\end{align*}
is also equal to $\frac{\overline\g}{2}$. It justifies that using the quadratic forms $G$ constructed above, we may capture the optimal contraction rate on the area of interest. 
We also refer the interested readers to \cite{MPP2002} for the structure of spectrum of (possibly degenerate) Ornstein-Uhlenbeck operators. 
\end{rem}

\subsubsection{Proof of contraction}\label{subsec:proofcontraction}

\begin{lem}
Let $c\in \dbR$, $\eta,\beta\in(0,\infty)$, and suppose that $h:[0,\infty) \longrightarrow[0,\infty)$ is continuous, non-decreasing, concave, and $C^2$ except for finitely many points.
Define
\begin{equation*}
\psi_t := \psi(X_t-X_t',V_t-V_t') = (1+\beta G_t)h(\ell_t), \quad \mbox{with} \quad  \ell_t:=r_t + \eta u_t.
\end{equation*}
Then, 
\begin{align} \label{eq:ePsileqPsi0KM}
e^{ct}\psi_t &\leq \psi_0 + \int_0^t e^{cs}K_s ds + M_t, \quad t\geq 0,
\end{align}		
where $M$ is a continuous  local martingale, and 
\begin{equation}  \label{eq:driftK}
\begin{aligned}
K_t &= (1+\beta G_t)h'(\ell_t)\big\{\eta\big|\delta b_t\big| + u_t + \big(\eta\lambda - \gamma\big)r_t\big\} 
+ (1+\beta G_t)2h''(\ell_t)\eta^2\sigma^2\rc_t^2   \\
&\qquad  + 4\beta\eta\sigma^2\rc_t^2 h'(\ell_t)\big\|\overline{Q}\big\|(r_t+2u_t) + c\psi_t - \overline{\gamma}\beta G_t h(\ell_t)  \\
&\qquad  +  \beta h(\ell_t)
\left|\overline Q 
\begin{pmatrix}
-\delta b_t\cdot\delta X_t \\
0 \\
-2\delta b_t\cdot P_t + 4\rc_t^2\sigma^2
\end{pmatrix}\right|.  
\end{aligned}
\end{equation}
\end{lem}

\begin{proof}
Since by assumption, $h$ is concave and piecewise $C^2$, we can now apply the It\^o-Tanaka formula to $h(\ell_t)$. 
Denote by $h'$ and $h''$ the left-sided first derivative and the almost everywhere defined second derivative. 
The generalized second derivative of $h$ is a signed measure $\mu_h$ such that $\mu_h(d\ell)\leq h''(\ell)d\ell$. 
We obtain  
\begin{align*}
dh(\ell_t) &= h'(\ell_t)(dr_t + \eta du_t) + \frac{1}{2}h''(\ell_t)d\langle \eta u\rangle_t \\
&= h'(\ell_t)\Big\{e^x_t\cdot P_t - \gamma r_t-\eta\delta b_t\cdot e_t^p - \eta\lambda e^p_t\cdot\delta X_t \Big\}dt
+ 2h''(\ell_t)\eta^2\sigma^2\rc_t^2dt \\
&\qquad + 2h'(\ell_t)\eta\sigma\rc_t dB_t \\
&\leq h'(\ell_t)\Big\{\eta|\delta b_t| + u_t + \big(\eta\lambda  - \gamma\big)r_t\Big\}dt 
+ 2h''(\ell_t)\eta^2\sigma^2\rc_t^2dt
+ 2h'(\ell_t)\eta\sigma\rc_t dB_t.  
\end{align*}
Calculate the quadratic variation
\begin{align*}
d\langle h(\ell),G\rangle_t 
&= 2h'(\ell_t)\eta\sigma\rc_t\overline{Q}
\begin{pmatrix}
2\frac{z_t}{u_t} \\
0 \\
4 u_t
\end{pmatrix}\sigma\rc_t dt
\leq 4\eta\sigma^2\rc_t^2 h'(\ell_t)\big\|\overline{Q}\big\|(r_t+2u_t)dt.
\end{align*}
Finally, again by It\^o's formula, we obtain
\begin{align*}
d\big(e^{ct}\psi_t\big) 
&= e^{ct}\big((1+\beta G_t)dh(\ell_t) + \beta h(\ell_t)dG_t + \beta d\langle h(\ell),G\rangle_t + c\psi_tdt\big) \\
&\leq e^{ct}\big(K_tdt + d\widetilde{M}_t\big),
\end{align*}
with
\begin{equation*}
d\widetilde{M}_t = (1+\beta G_t)2h'(\ell_t)\eta\sigma\rc_t dB_t 
+ \beta h(\ell_t) \overline{Q}
\begin{pmatrix}
2\frac{z_t}{u_t}  \\
0 \\
4 u_t
\end{pmatrix}\sigma\rc_t dB_t 
\end{equation*}
and the process $K$ defined in \eqref{eq:driftK}. 
The assertion follows by taking $M_t=\int_0^te^{cs}\widetilde{M}_s$.
\end{proof}
\ms

In order to make $\psi_t$ a contraction under expectation, it remains to choose the coefficients  $\eta, \b, h$  so that $\dbE[K_t]\leq 0$. 

\paragraph{Choice of coefficients}
Recall $\l_{\scriptscriptstyle G}$ in Lemma \ref{lem:uniformelliptic}. 
We fix a constant  
  \begin{align} \label{eq:eps0}
  	 \varepsilon_0:=\frac{\overline\gamma \lambda_{\scriptscriptstyle G}}{14\|\overline Q\|} \wedge\frac{\gamma}{2}<\gamma.  
  \end{align}
Recall that there exists $M>0$ such that for all $m,m'\in \Pc_2(\R^n)$
\begin{equation} \label{eq:smallLipschitz}
\big|b(m,x) - b(m,x') \big| \leq \e_0|x-x'|\quad\mbox{whenever $|x-x'|\ge M$.}
\end{equation}
Using $\e_0, M$ above, we define     
\begin{align}  \label{def:eta-theta}
 \eta := \frac{ \e_0}{L^x + \lambda + 4\big\|\overline Q\big\|\sigma^2} \wedge \frac{M\sqrt{\varepsilon_0}}{\sigma}
\quad\mbox{and} \quad \theta:=\frac{1}{\eta}+8\big\|\overline Q\big\|\sigma^2, 
\end{align} 
where $L^x$ is the Lipschitz constant of the function $b$ in $x$. 
Now we are ready to introduce the function
 \begin{align}\label{eq:def-h}
 	h(\ell) = \int_0^{2M\wedge\ell}\varphi(s)g(s)ds,
 \end{align}
 with
\begin{align*}
 \varphi(s) = \exp\left(-\frac{\theta}{\eta^2\sigma^2}\frac{s^2}{4}\right), \quad 
g(s)= 1-\frac{1}{2}\frac{\int_0^s\frac{\Phi(r)}{\varphi(r)}dr}{\int_0^{2M}\frac{\Phi(r)}{\varphi(r)}dr}, \quad 
\Phi(r) = \int_0^r\varphi(x)dx.
\end{align*} 
\begin{rem} \label{rem:phigh}
	The function $h$ and its similar variations are repeatedly used in Eberle \cite{eberle11}, Eberle, Guillin and Zimmer \cite{eberle2019quantitative, EGZ19}, Luo and Wang \cite{LW16} to measure the contraction under the reflection coupling.  
	In particular, the functions $\varphi$, $g$ and $h$ have the following properties:
	\begin{itemize} 
		\item $\varphi$ is decreasing, 
		$$ \varphi_{\min}:=\min_{0\leq s\leq 2M}\varphi(s)= \exp\left(-\frac{\theta}{\eta^2\sigma^2}M^2\right). $$ 
		\item $g$ is decreasing, $g(0)=1$ and $g(s)\geq g(2M)=\frac{1}{2}$ for $r\in[0,2M]$. 
		\item $h$ is non-decreasing, concave, $h(0)=0$, $h'(0)=1$, 
		$$ h'(2M)=\varphi(2M)g(2M)=\frac{\varphi_{\min}}{2}>0 $$
		and $h$ is constant on $[2M,\infty)$
		\begin{equation*}
		h(\ell)\leq \ell, \quad \frac{\Phi(\ell)}{2} \leq h(\ell) \leq \Phi(\ell),\quad \ell\leq 2M, 
		\end{equation*}
		and 
		\begin{equation} \label{eq:DiffIneqForH}
		\theta \ell h'(\ell) + 2 \eta^2 \sigma^2 h''(\ell)\le - \overline\kappa_{\scriptscriptstyle M} h(\ell), ~~ \ell \le 2M, ~~~ \mbox{with} ~~~ 
		\overline\kappa_{\scriptscriptstyle M}:=\frac{\eta^2\sigma^2}{\int_0^{2M}\frac{\Phi(r)}{\varphi(r)}dr}.
		\end{equation}
	\end{itemize}
\end{rem}

\noindent For the later use we further define a constant  $\kappa_{\scriptscriptstyle M}>0$ such that 
\begin{equation}  \label{ineq:rhoM}
\kappa_{\scriptscriptstyle M} :=  \overline\kappa_{\scriptscriptstyle M} \wedge \frac{\f_{\min}}{2}\Big(\g -\eta\big(L^x + \lambda + 4\big\|\overline Q\big\|\sigma^2\big) \Big).
\end{equation}
Note that by the definition of $\eta$ in \eqref{def:eta-theta} we have $\eta\big(L^x + \lambda + 4\big\|\overline Q\big\|\sigma^2\big)<\gamma$. 
Next  introduce the constants 
\begin{align*}
C_1 := 4\big\|\overline Q\big\|L^xM^2\Big(1+\tfrac{2}{\eta}\Big) + 4\big\|\overline Q\big\|\sigma^2, 
\end{align*}
and choose the coefficient $\b\in(0,1]$ such that
\begin{align}\label{eq:choosebeta}
 \b< \frac{\kappa_{\scriptscriptstyle M}}{C_1} \wedge 1, 
\end{align}
e.g., define 
 \begin{align*}
 	\beta:= \frac{\kappa_{\scriptscriptstyle M}}{2C_1}\wedge 1. 
 \end{align*}
Finally we may find a constant $C_0$ such that $r +u \le C_0 \psi$ and thus 
\[\cW_1 \le C_0 \cW_\psi.\]
For the later use, define
 \begin{align*}
 	 C_2 ~:= 2\big\|\overline Q\big\|M\Big(1 + \tfrac{2}{\eta}\Big)C_0,  \quad C_{\scriptscriptstyle M} ~:= 4M^2 C_{\scriptscriptstyle G}\Big(1+\tfrac{1}{\eta^2}\Big),
 \end{align*}
 as well as 
 \begin{align} \label{eq:ciota}
 & c:=\min\bigg\{ \kappa_{\scriptscriptstyle M} -C_1\b - \Big((1+ \beta C_{\scriptscriptstyle M})\eta  C_0 + \beta h(2M)C_2\Big)\iota, \nonumber \\
 &\hspace{46mm} \Big(\overline{\gamma} -\frac{7\big\|\overline Q\big\|}{\lambda_{\scriptscriptstyle G}}\varepsilon_0\Big)\frac{2\beta\lambda_{\scriptscriptstyle G}M^2}{1 + 2\beta\lambda_{\scriptscriptstyle G}M^2}\bigg\},
 \end{align}
 with a constant $\iota>0$ such that 
  \begin{align} \label{eq:smalliota}
    \kappa_{\scriptscriptstyle M} -C_1\b - \Big((1+ \beta C_{\scriptscriptstyle M})\eta C_0 + \beta h(2M)C_2\Big)\iota >0.
  \end{align}
  

\begin{rem}\label{rem:convergence rate}
The constant $c$ defined in \eqref{eq:ciota} represents the contraction rate in Theorem \ref{thm:Contraction}. To enhance the understanding of this quantity, here we provide a lower bound of  $c$ for some specific case.
First we may define the new variables
\begin{equation*}
X' = \frac{\gamma^{3/2}}{ \sigma} X, ~~V' = \frac{\gamma^{1/2}}{\sigma} V, ~~t' = \gamma^{-1} t.
\end{equation*}
Then \(X'_{t'}, V'_{t'}\) satisfy
\begin{equation}
\label{eq:mkfl-normalized}
\begin{aligned}
dX'_{t'} &= V'_{t'} dt, \\
dV'_{t'} &= \big(  - D_m F'\big({\mathcal L}(X'_{t'}), X'_{t'}\big)- V'_{t'} \big)dt +  dW'_{t},
\end{aligned}
\end{equation}
where $W'_{t} = \gamma^{1/2} W_{t'}$ is a standard Brownian motion, and 
\(F' := \left(\frac{1}{\gamma\sigma^2}\right)^{1/2} F\). Therefore, without loss of generality we may assume $\sigma = \gamma = 1$. 
In this case, by a direct computation we obtain the following lower bound of $c$ whenever  $\lambda < \frac{1}{4}\gamma = \frac{1}{4}$:
\[c \ge  \frac{(1-\sqrt{1-4\lambda})^2 (3-2\lambda - \sqrt{4\lambda^2 + 4\lambda +5})^2}{14336 L^x(L^x+\lambda + 8)} \kappa_{\scriptscriptstyle M}.\]
This lower bound indicates that the rate we obtain through the coupling method is indeed small, despite the fact that we prove the contraction result in Theorem \ref{thm:Contraction}. Moreover, if we reduce the value of $\sigma$, the corresponding Lipschitz constant $L_x$ of $F'$ becomes bigger and the lower bound of $c$ above becomes smaller.  
\end{rem}

\begin{lem}
	With the choice of the coefficients $\eta, \b, h, c$ above, there exists $C\geq 0$ such that $\dbE[K_t]\leq C\xi$. 
\end{lem} 
\begin{proof}
	We divide $(r, u)\in \dbR_+ \times \dbR_+$ into two regions: 
	
	\vspace{2mm}
	
	\noindent {\rm (i).} \quad $\ell_t=r_t+\eta u_t\leq 2M$: 
	It follows by Lemma \ref{lem:uniformelliptic} and due to $r_t+\eta u_t\leq 2M$ that 
	$$ G_t \leq C_{\scriptscriptstyle G}\big(r_t^2+u_t^2\big) \leq 4M^2 C_{\scriptscriptstyle G}\Big(1+\tfrac{1}{\eta^2}\Big) = C_{\scriptscriptstyle M}. $$
	It is due to the Lipschitz assumption \eqref{eq:AssumpLipMeasure} and the fact $\cW_1\le C_0\cW_\psi$ that
	\begin{align*} 
	\left|\overline Q \begin{pmatrix}  
	-\delta b_t \cdot\delta X_t \\	
	0 \\
	-2\delta b_t \cdot P_t + 4\sigma^2\rc_t^2
	\end{pmatrix}\right| 
	&\leq \big\|\overline Q\big\|\big(|\delta b_t |r_t + 2|\delta b_t|u_t + 4\sigma^2\rc_t^2\big) \\
	&\leq \big\|\overline Q\big\|\Big(\big(C_0\iota\cW_\psi(m_t,m'_t)+L^xr_t\big)(r_t + 2u_t) + 4\sigma^2\Big) \\
	&\leq 2\big\|\overline Q\big\|M\Big(1 + \tfrac{2}{\eta}\Big)C_0\iota\cW_\psi(m_t,m'_t) \\
	&\qquad + 4\big\|\overline Q\big\|L^xM^2\Big(1+\tfrac{2}{\eta}\Big) + 4\big\|\overline Q\big\|\sigma^2  \\
	& = C_2\iota\cW_\psi(m_t,m'_t) + C_1.
	\end{align*}
	
	\noindent Together with \eqref{eq:driftK} we obtain
	\begin{align*}
	K_t &\leq (1+\beta G_t)h'(\ell_t)\left\{\eta C_0\iota\cW_\psi(m_t,m'_t) + \big(\eta(L^x+\lambda)-\gamma\big)r_t + u_t\right\} \\
	&\qquad + (1+\beta G_t)2h''(\ell_t)\eta^2\sigma^2\rc_t^2 + 4\beta\big\|\overline Q\big\|\eta\sigma^2\rc_t^2 h'(\ell_t)(r_t+2u_t) \\
	&\qquad + c\psi_t + \beta h(\ell_t)\big( C_2\iota\cW_\psi(m_t,m'_t) + C_1\big) \\
	&\leq \big((1+\beta C_{\scriptscriptstyle M})\eta C_0 + \beta h(2M)C_2\big)\iota\cW_\psi(m_t,m'_t) + C_1\beta h(\ell_t) + c\psi_t \\
	&\qquad + (1+\beta G_t)h'(\ell_t)\left\{\eta\left(L^x+\lambda + \frac{4\beta \big\|\overline Q\big\|\sigma^2\rc_t^2}{1+\beta G_t}\right)-\gamma \right\}r_t + I_t \\
	&\mbox{with}\quad I_t : =  (1+\beta G_t)h'(\ell_t)\left(1 + \frac{8\beta \big\|\overline Q\big\|\eta\sigma^2\rc_t^2}{1+\beta G_t}\right) u_t  + (1+\beta G_t)2h''(\ell_t)\eta^2\sigma^2\rc_t^2 
	\end{align*}
	Recall $\theta$ defined in \eqref{def:eta-theta}. Since $\beta < 1$, we have 
	\begin{align*} 
	\frac1\eta + \frac{8\beta \big\|\overline Q\big\|\sigma^2\rc_t^2}{1+\beta G_t} \le \frac{1}{\eta}+8\big\|\overline Q\big\|\sigma^2 =\th.
	\end{align*}
	Further recall that $h$ satisfies the inequality \eqref{eq:DiffIneqForH} and the constant $\kappa_{\scriptscriptstyle M}$ defined in \eqref{ineq:rhoM}. 
	Since $h''(\ell)\leq 0$, $h'(\ell)\leq 1$, $h(\ell)\leq\ell$ and $\rc_t=1$ whenever $u_t\geq\xi$, we obtain 
	\begin{align*}
	I_t	&\leq (1+\beta G_t)\theta \eta u_t h'(\ell_t) + (1+\beta G_t)2h''(\ell_t)\eta^2\sigma^2\rc_t^2 \\
	&\leq (1+\beta G_t)\Big(\theta \ell_th'(\ell_t) + 2\eta^2\sigma^2h''(\ell_t)\Big)1_{\{u_t\geq\xi\}} + (1+\beta C_{\scriptscriptstyle M})\theta\eta\xi 1_{\{u_t\leq\xi\}} \\
	&\leq -(1+\beta G_t)\overline\kappa_{\scriptscriptstyle M} h(\ell_t)1_{\{u_t\geq\xi\}} + (1+\beta C_{\scriptscriptstyle M})\theta\eta\xi \\
	&\leq -(1+\beta G_t)\kappa_{\scriptscriptstyle M} h(\ell_t) + (1+\beta G_t)\kappa_{\scriptscriptstyle M} h(\ell_t)1_{\{u_t\leq\xi\}} + (1+\beta C_{\scriptscriptstyle M})\theta\eta\xi \\
	&\leq -(1+\beta G_t)\kappa_{\scriptscriptstyle M} h(\ell_t) + (1+\beta G_t)\kappa_{\scriptscriptstyle M} r_t + (1+\beta C_{\scriptscriptstyle M})(\kappa_{\scriptscriptstyle M} + \theta)\eta\xi. 
	\end{align*} 
	Hence, 
	\begin{align*}
	K_t &\leq \big((1+\beta C_{\scriptscriptstyle M})\eta C_0 + \beta h(2M)C_2\big)\iota\cW_\psi(m_t,m'_t) + C_1\beta h(\ell_t) + c\psi_t \\
	&\qquad -(1+\beta G_t)\kappa_{\scriptscriptstyle M} h(\ell_t) + (1+\beta G_t)\kappa_{\scriptscriptstyle M} r_t + (1+\beta C_{\scriptscriptstyle M})(\kappa_{\scriptscriptstyle M} + \theta)\eta\xi \\
	&\qquad + (1+\beta G_t)h'(\ell_t)\left\{\eta\left(L^x+\lambda + \frac{4\beta \big\|\overline Q\big\|\sigma^2\rc_t^2}{1+\beta G_t}\right)-\gamma \right\}r_t \\
	&\leq \big((1+\beta C_{\scriptscriptstyle M})\eta C_0 + \beta h(2M)C_2\big)\iota\cW_\psi(m_t,m'_t) + C_1\beta h(\ell_t) + c\psi_t - \kappa_{\scriptscriptstyle M}\psi_t \\
	&\qquad + (1+\beta G_t)h'(\ell_t)\left\{\frac{\kappa_{\scriptscriptstyle M}}{h'(\ell_t)}+\eta\left(L^x+\lambda + 4\big\|\overline Q\big\|\sigma^2\right)-\gamma \right\}r_t   \\
	&\qquad + (1+\beta C_{\scriptscriptstyle M})(\kappa_{\scriptscriptstyle M} + \theta)\eta\xi. 
	\end{align*}
	Due to the choice of $\eta$ in \eqref{def:eta-theta} and $\kappa_{\scriptscriptstyle M}$ in  \eqref{ineq:rhoM}, the factor of $r_t$ above is non-positive, i.e.
	\begin{align*}
	\frac{\kappa_{\scriptscriptstyle M}}{h'(\ell_t)}+\eta\left(L^x+\lambda + 4\big\|\overline Q\big\|\sigma^2\right)-\gamma 
	\leq \frac{2\kappa_{\scriptscriptstyle M}}{\varphi_{\min}}+\eta\left(L^x+\lambda + 4\big\|\overline Q\big\|\sigma^2\right)-\gamma \leq 0.
	\end{align*}
	Therefore, we obtain
	\begin{align*}
	K_t \leq \big((1+\beta C_{\scriptscriptstyle M})\eta C_0 
	 &+ \beta h(2M)C_2\big)\iota\cW_\psi(m_t,m'_t) + \big(C_1\beta + c - \kappa_{\scriptscriptstyle M}\big)\psi_t   \\ 
	 &+ (1+\beta C_{\scriptscriptstyle M})(\kappa_{\scriptscriptstyle M} + \theta)\eta\xi.
	\end{align*} 
	Since $\cW_\psi(m_t,m_t')\leq\dbE[\psi_t]$ and taking expectation on both sides we obtain that
	\begin{align*}  \label{eq:region1}
	\dbE[K_t] & \leq \Big(\big((1+\beta C_{\scriptscriptstyle M})\eta C_0 + \beta h(2M)C_2\big)\iota + \left(C_1\beta + c - \kappa_{\scriptscriptstyle M}\right)\Big) \dbE[\psi_t] \\
	&\qquad + (1+\beta C_{\scriptscriptstyle M})(\kappa_{\scriptscriptstyle M} + \theta)\eta\xi\\
	& \le (1+\beta C_{\scriptscriptstyle M})(\kappa_{\scriptscriptstyle M} + \theta)\eta\xi =: C\xi,
	\end{align*}		
	where  the last inequality is due to the definition of $c$ in \eqref{eq:ciota}.
	
	\vspace{3mm}

	
	\noindent {\rm (ii).}\quad $\ell_t = r_t+\eta u_t\geq 2M$: In this region, $h(\ell_t)$ is constant, $h'(\ell_t)=h''(\ell_t)= 0$.
	Therefore, 
	\begin{equation} \label{eq:simplifyK}
	\begin{aligned}
	K_t = c\psi_t - \overline{\gamma}\beta G_th(2M) 
	+ \beta h(2M)\left| \overline Q \begin{pmatrix}  
	-\delta b_t\cdot\delta X_t \\	
	0 \\
	-2\delta b_t\cdot P_t + C_0 \sigma^2\rc_t^2 
	\end{pmatrix}\right|.  
	\end{aligned}
	\end{equation}
	Further we can divide this region into two parts:
	$$ \{(r,u):~  \eta u +  r \ge 2 M \} ~\subseteq~ \{r \ge M\} \cup \{u\ge \eta^{-1}(r \vee M)\}.$$
	Recall that by the choice of $\eta$ in \eqref{def:eta-theta} we have $\sigma^2\leq \varepsilon_0M^2\big(1\vee\tfrac{1}{\eta}\big)^2$ and $\varepsilon_0 \geq L^x\eta$. 
	Together with \eqref{eq:smallLipschitz} we obtain
	\begin{align*}
	|\d b_t| \le \e_0 r_t, \quad \rc_t^2 \si^2 \le \e_0 r_t^2, \quad \mbox{on}~~\{r \ge M\} 
	\end{align*}
	as well as
	\begin{align*}
	|\d b_t| \le L^x\eta u_t\le \e_0 u_t, \quad \rc_t^2 \si^2 \le \e_0 u_t^2, \quad \mbox{on}~~\{u \ge \eta^{-1} (r \vee M)\}.
	\end{align*}
	Combining the two estimates above, we get
	$$ |\d b_t| \leq \e_0(r_t\vee u_t). $$
	and therefore
	\begin{align}\label{eq:estimateQbar}
	\left|\overline Q
	\begin{pmatrix}
	-\delta b_t\cdot\delta X \\
	0 \\
	-2\delta b_t\cdot P_t + 4\rc_t^2\sigma^2
	\end{pmatrix} \right|
	&\leq 7 \big\|\overline Q\big\|\varepsilon_0 (r^2_t + u^2_t) 
	\leq \frac{7\big\|\overline Q\big\|\varepsilon_0}{\lambda_{\scriptscriptstyle G}}G_t,
	\end{align}
	where for the last inequality we use the coercivity in Lemma \ref{lem:uniformelliptic}.
	Also due to $r_t + \eta u_t\geq 2M$ and $\eta\leq 1$ we have 
	$$ \frac{\beta G_t}{1+\beta G_t} \geq \frac{2\beta\lambda_{\scriptscriptstyle G}M^2}{1 + 2\beta\lambda_{\scriptscriptstyle G}M^2}. $$
	Together with \eqref{eq:simplifyK} and \eqref{eq:estimateQbar} we obtain
	\begin{align*}
	K_t &\leq c\psi_t - \overline{\gamma}\beta G_t h(2M) + \beta h(2M)\frac{7\big\|\overline Q\big\|\varepsilon_0}{\lambda_{\scriptscriptstyle G}}G_t  \\
	&= c\psi_t - \overline{\gamma}\frac{\beta G_t}{1+\beta G_t}\psi_t + \frac{7\big\|\overline Q\big\|\varepsilon_0}{\lambda_{\scriptscriptstyle G}}\frac{\beta G_t}{1+\beta G_t}\psi_t   \\
	&= \psi_t\left(c - \bigg(\overline{\gamma} - \frac{7\big\|\overline Q\big\|}{\lambda_{\scriptscriptstyle G}}\varepsilon_0\bigg)\frac{\beta G_t}{1+\beta G_t} \right) \\
	&\leq \psi_t\left(c -\bigg(\overline{\gamma} - \frac{7\big\|\overline Q\big\|}{\lambda_{\scriptscriptstyle G}}\varepsilon_0\bigg)\frac{2\beta\lambda_{\scriptscriptstyle G}M^2}{1 + 2\beta\lambda_{\scriptscriptstyle G}M^2}\right)\le0,
	\end{align*}
	where the second last inequality is due to the choice of $\e_0$ in \eqref{eq:eps0} and the last one is due to  $c$ defined in \eqref{eq:ciota}.
\end{proof}

\vspace{5mm}

\begin{proof}[Proof of Theorem \ref{thm:Contraction}]
Let $\Gamma$ be a coupling of two probability measures $m_0$ and $m_0'$ on $\R^{2n}$ such that $$ \cW_\psi\big(m_0,m_0'\big) = \int\psi d\Gamma. $$  
We consider the coupling process $\big((X,V),(X',V')\big)$ introduced above with initial law 
  $$ \big((X_0,V_0),(X_0',V_0')\big)\sim\Gamma. $$ 
By taking expectation on both sides of \eqref{eq:ePsileqPsi0KM}, evaluated at localizing stopping times $\tau_n\to t$ and applying Fatou's lemma as $n\to\infty$, we obtain
\begin{align*}
\dbE[e^{ct}\psi_t] 
&\leq \dbE[\psi_0] + \int_0^t e^{cs}\dbE[K_s] ds 
\le  \dbE[\psi_0] + Cc^{-1}(e^{ct}-1)\xi.
\end{align*}
for any $\xi>0$ and $t\geq 0$. 
Note that $\dbE[\psi_0]=\int\psi d\Gamma = \cW_\psi\big(m_0,m_0'\big) $. Therefore
\begin{align*}
\cW_\psi(m_t,m_t') 
&\leq \dbE[\psi_t] \le  e^{-ct}\cW_\psi\big(m_0,m_0'\big)  +  Cc^{-1}\big(1-e^{-ct}\big)\xi\longrightarrow e^{-ct}\cW_\psi\big(m_0,m_0'\big),
\end{align*}
as $\xi\longrightarrow 0$. 
Finally note that by the choice of $\b$ in \eqref{eq:choosebeta}, we have $c>0$ according to \eqref{eq:ciota} provided that $\iota$ is small enough.
\end{proof}

\bibliographystyle{abbrv}
\bibliography{references}

\begin{thebibliography}{10}

\bibitem{ACB2017}
M.~Arjovsky, S.~Chintala, and L.~Bottou.
\newblock {W}asserstein generative adversarial networks.
\newblock In D.~Precup and Y.~W. Teh, editors, {\em Proceedings of the 34th
  International Conference on Machine Learning}, volume~70 of {\em Proceedings
  of Machine Learning Research}, pages 214--223. PMLR, 06--11 Aug 2017.

\bibitem{AM19}
S.~Armstrong and J.-C. Mourrat.
\newblock {Variational methods for the kinetic Fokker-Planck equation}.
\newblock {\em Preprint arXiv:1902.04037}, 2019.

\bibitem{BCG08}
D.~Bakry, P.~Cattiaux, and A.~Guillin.
\newblock {Rate of convergence for ergodic continuous Markov processes:
  Lyapunov versus Poincar{\'e}.}
\newblock {\em J. Funct. Anal.}, 254(3):727--759, 2008.

\bibitem{bally1991connection}
V.~Bally.
\newblock {On the connection between the {M}alliavin covariance matrix and
  {H}{\"o}rmander's condition}.
\newblock {\em Journal of functional analysis}, 96(2):219--255, 1991.

\bibitem{BCK95}
G.~Ben~Arous, M.~Cranston, and S.~Kendall.
\newblock Coupling constructions for hypoelliptic diffusions: two examples.
\newblock {\em Stochastic Analysis, Proceedings of Symposia in Pure
  Mathematics}, 57, 1995.

\bibitem{BGBM20}
F.~Bolley, A.~Guillin, P.~Le~Bris, and P.~Monmarch\'e.
\newblock Wasserstein contraction for kinetic mean field particles system.
\newblock {\em Ongoing}.

\bibitem{BGM10}
F.~Bolley, A.~Guillin, and F.~Malrieu.
\newblock {Trend to equilibrium and particle approximation for a weakly
  selfconsistent Vlasov-Fokker-Planck equation}.
\newblock {\em M2AN Math. Model. Numer. Anal.}, 44(5):867--884, 2010.

\bibitem{BCS97}
L.~L. Bonilla, J.~Carrillo, and J.~Soler.
\newblock {Asymptotic behavior of an initial-boundary value problem for the
  Vlasov-Poisson-Fokker-Planck system}.
\newblock {\em SIAM J. Appl. Math.}, 57(5):1343--1372, October 1997.

\bibitem{BEZ20}
N.~Bou-Rabee, A.~Eberle, and R.~Zimmer.
\newblock {Coupling and convergence for Hamiltonian Monte Carlo}.
\newblock {\em Ann. Appl. Probab.}, 30(3):1209--1250, 2020.

\bibitem{BS2020}
N.~Bou-Rabee and K.~Schuh.
\newblock {Convergence of unadjusted Hamiltonian Monte Carlo for mean-field
  models}.
\newblock {\em Preprint arXiv:2009.08735}, 2020.

\bibitem{BBK1984}
A.~Br{\"u}nger, C.~{Brooks III}, and M.~Karplus.
\newblock {Stochastic boundary conditions for molecular dynamics simulations of
  ST2 water}.
\newblock {\em Chem. Phys. Lett.}, 105(5):495--500, 1984.

\bibitem{CLW19}
Y.~Cao, J.~Lu, and L.~Wang.
\newblock {On explicit $L^2$-convergence rate estimate for underdamped Langevin
  dynamics}.
\newblock {\em Preprint arXiv:1908.04746}, 2019.

\bibitem{Card2018}
P.~{C}ardaliaguet.
\newblock A short course on mean field games.
\newblock {\em Preprint}, 2018.

\bibitem{CDLLMaster2019}
P.~Cardaliaguet, F.~Delarue, J.-M. Lasry, and P.-L. Lions.
\newblock {\em {T}he master equation and the convergence problem in mean field
  games: ({AMS}-201)}.
\newblock Princeton University Press, 2019.

\bibitem{Carmona2016Book}
R.~Carmona.
\newblock {\em Lectures on {BSDE}s, stochastic control, and stochastic
  differential games with financial applications}.
\newblock SIAM, 2016.

\bibitem{CarmonaDelarueMFGBook1}
R.~Carmona and F.~Delarue.
\newblock {\em Probabilistic theory of mean field games with applications.
  {I}}, volume~83 of {\em Probability Theory and Stochastic Modelling}.
\newblock Springer, Cham, 2018.
\newblock Mean field FBSDEs, control, and games.

\bibitem{CM02}
P.~Cattiaux and L.~Mesnager.
\newblock {Hypoelliptic non-homogeneous diffusions}.
\newblock {\em Probability Theory and Related Fields}, 123(4):453--483, 2002.

\bibitem{CCAYBJ2020}
X.~Cheng, N.~S. Chatterji, Y.~Abbasi-Yadkori, P.~L. Bartlett, and M.~I. Jordan.
\newblock {Sharp convergence rates for Langevin dynamics in the nonconvex
  setting}.
\newblock {\em Preprint}, 2020.

\bibitem{CCBJ18}
X.~Cheng, N.~S. Chatterji, P.~L. Bartlett, and M.~I. Jordan.
\newblock {Underdamped Langevin MCMC: A non-asymptotic analysis}.
\newblock {\em Proceedings of Machine Learning research}, 75:1--24, 2018.

\bibitem{chizat2018global}
L.~Chizat and F.~Bach.
\newblock {On the global convergence of gradient descent for over-parameterized
  models using optimal transport}.
\newblock In {\em {Advances in neural information processing systems}}, pages
  3040--3050, 2018.

\bibitem{CKR20}
G.~Conforti, A.~Kazeykina, and Z.~Ren.
\newblock {Game on Random Environement, Mean-field Langevin System and Neural
  networks}.
\newblock {\em to appear in Mathematics of Operations Research}, 2020.

\bibitem{Dala17}
A.~S. Dalalyan.
\newblock {Theoretical guarantees for approximate sampling from smooth and
  log-concave densities}.
\newblock {\em Journal of the Royal Statistical Society: Series B},
  79(3):651--676, 2017.

\bibitem{DKMS13}
J.~Dolbeault, C.~Klar, C.~Mouhot, and C.~Schmeiser.
\newblock Exponential rate of convergence to equilibrium for a model describing
  fiber lay-down processes.
\newblock {\em Appl. Math. Res. express}, (2):165--175, 2013.

\bibitem{DMS09}
J.~Dolbeault, C.~Mouhot, and C.~Schmeiser.
\newblock {Hypocoercivity for kinetic equations with linear relaxation terms}.
\newblock {\em Comptes Rendus Mathematique}, 347(9):511--516, 2009.

\bibitem{DMS15}
J.~Dolbeault, C.~Mouhot, and C.~Schmeiser.
\newblock {Hypocoercivity for linear kinetic equations conserving mass}.
\newblock {\em Transactions of the American Mathematical Society},
  367(6):3807--3828, 2015.

\bibitem{DJMRB2020}
C.~Domingo-Enrich, S.~Jelassi, A.~Mensch, G.~Rotskoff, and J.~Bruna.
\newblock A mean-field analysis of two-player zero-sum games.
\newblock {\em Advances in Neural Information Processing Systems},
  33:20215--20226, 2020.

\bibitem{DT18}
M.~H. Duong and J.~Tugaut.
\newblock {The Vlasov-Fokker-Planck equation in non-convex landscapes:
  convergence to equilibrium}.
\newblock {\em Electron. Commun. Probab.}, 23(19):1--10, 2018.

\bibitem{DE97}
P.~Dupuis and R.~S. Ellis.
\newblock {\em {A Weak Convergence Approach to the Theory of Large
  Deviations}}.
\newblock Wiley, 1997.

\bibitem{DM16}
A.~Durmus and E.~Moulines.
\newblock {Sampling from strongly log-concave distributions with the Unadjusted
  Langevin Algorithm}.
\newblock {\em Preprint arXiv:1605.01559}, 2016.

\bibitem{eberle11}
A.~Eberle.
\newblock {Reflection couplings and contraction rates for diffusions}.
\newblock {\em Probability Theory and Related Fields}, 166(3-4):851--886, 2016.

\bibitem{EGZ19}
A.~Eberle, A.~Guillin, and R.~Zimmer.
\newblock {{C}ouplings and quantitative contraction rates for {L}angevin
  dynamics}.
\newblock {\em Ann. Probab.}, 47(4):1982--2010, 2019.

\bibitem{eberle2019quantitative}
A.~Eberle, A.~Guillin, and R.~Zimmer.
\newblock {Quantitative {H}arris-type theorems for diffusions and
  {M}c{K}ean--{V}lasov processes}.
\newblock {\em Transactions of the American Mathematical Society},
  371(10):7135--7173, 2019.

\bibitem{Eins05}
A.~Einstein.
\newblock {{\"U}ber die von der molekularkinetischen Theorie der W{\"a}rme
  geforderte Bewegung von in ruhenden Fl{\"u}ssigkeiten suspendierten
  Teilchen}.
\newblock {\em Annalen der Physik}, 322(8):549--560, 1905.

\bibitem{Follmer}
H.~F{\"o}llmer.
\newblock {Time reversal on {W}iener space}.
\newblock In S.~A. Albeverio, P.~Blanchard, and L.~Streit, editors, {\em
  {Stochastic Processes - Mathematics and Physics}}, pages 119--129. Springer,
  1986.

\bibitem{FJ16}
J.~Fontbona and B.~Jourdain.
\newblock {A trajectorial interpretation of the dissipations of entropy and
  {F}isher information for stochastic differential equations}.
\newblock {\em Ann. Probab.}, 44(1):131--170, 2016.

\bibitem{metropolis_efficient}
A.~Gelman, G.~O. Roberts, and W.~R. Gilks.
\newblock {Efficient Metropolis Jumping Rules}.
\newblock {\em Bayesian Statistics}, 5:599--607, 1996.

\bibitem{GAN14}
I.~J. Goodfellow, J.~Pouget-Abadie, M.~Mirza, B.~Xu, D.~Warde-Farley, S.~Ozair,
  A.~Courville, and Y.~Bengio.
\newblock {G}enerative {A}dversarial {N}ets.
\newblock In {\em Proceedings of the 27th International Conference on Neural
  Information Processing Systems - Volume 2}, NIPS'14, pages 2672--2680,
  Cambridge, MA, USA, 2014. MIT Press.

\bibitem{GS16}
M.~Grothaus and P.~Stilgenbauer.
\newblock {Hilbert space hypocoercivity for the Langevin dynamics revisited}.
\newblock {\em Methods Funct. Anal. Topology}, 22(2):152--168, 2016.

\bibitem{GLM22}
A.~Guillin, P.~Le~Bris, and P.~Monmarch\'e.
\newblock {Convergence rates for the Vlasov-Fokker-Planck equation and uniform
  in time propagation of chaos in non convex cases}.
\newblock {\em Electron. J. Probab.}, 27:1--44, 2022.

\bibitem{GLWZ2021}
A.~Guillin, W.~Liu, L.~Wu, and C.~Zhang.
\newblock {The kinetic Fokker-Planck equation with mean field interaction}.
\newblock {\em Journal de Math{\'e}matiques Pures et Appliqu{\'e}es},
  150:1--23, 2021.

\bibitem{GM2021}
A.~Guillin and P.~Monmarch{\'e}.
\newblock Uniform long-time and propagation of chaos estimates for mean field
  kinetic particles in non-convex landscapes.
\newblock {\em Journal of Statistical Physics}, 185(2):1--20, 2021.

\bibitem{haussmann1986time}
U.~G. Haussmann and E.~Pardoux.
\newblock {Time reversal of diffusions}.
\newblock {\em The Annals of Probability}, 14(4):1188--1205, 1986.

\bibitem{Henry81}
D.~Henry.
\newblock {\em {Geometric Theory of Semilinear Parabolic Equations}}.
\newblock Springer, 1981.

\bibitem{Herau2006}
F.~H\'{e}rau.
\newblock Hypocoercivity and exponential time decay for the linear
  inhomogeneous relaxation {B}oltzmann equation.
\newblock {\em Asymptot. Anal.}, 46(3-4):349--359, 2006.

\bibitem{HKR19}
K.~Hu, A.~Kazeykina, and Z.~Ren.
\newblock {Mean-field Langevin System, Optimal Control and Deep Neural
  Networks}.
\newblock {\em Preprint arXiv:1909.07278}, 2019.

\bibitem{HRSS19}
K.~Hu, Z.~Ren, D.~\v{S}i\v{s}ka, and L.~Szpruch.
\newblock {Mean-Field {L}angevin Dynamics and Energy Landscape of Neural
  Networks}.
\newblock {\em Ann. Inst. H. Poincar\'e. Probab. Statist.}, 57(4):2043--2065,
  November 2021.

\bibitem{IOS19}
A.~Iacobucci, S.~Olla, and G.~Stoltz.
\newblock {Convergence rates for nonequilibrium Langevin dynamics}.
\newblock {\em Ann. Math. Qu\'ebec}, 43(1):73--98, 2019.

\bibitem{JSS19}
J.-F. Jabir, D.~\v{S}i\v{s}ka, and L.~Szpruch.
\newblock {Mean-Field Neural ODEs via Relaxed Optimal Control}.
\newblock {\em Preprint arXiv:1912.05475}, 2019.

\bibitem{jacod1981weak}
J.~Jacod and J.~Memin.
\newblock Weak and strong solutions of stochastic differential equations:
  Existence and stability.
\newblock In D.~Williams, editor, {\em Stochastic Integrals}, pages 169--212,
  Berlin, Heidelberg, 1981. Springer Berlin Heidelberg.

\bibitem{Kozlov89}
S.~M. Kozlov.
\newblock {Effective diffusion for the Fokker-Planck equation}.
\newblock {\em Math. Notes}, 45:360--368, 1989.

\bibitem{Lang08}
P.~Langevin.
\newblock {Sur la th{\'e}orie du mouvement brownien}.
\newblock {\em CR Acad. Sci. Paris}, 146:530--533, 1908.

\bibitem{Molecular15}
B.~Leimkuhler and C.~Matthews.
\newblock {\em {Molecular Dynamics}}.
\newblock Interdisciplinary Applied Mathematics. Springer, Cham, 1 edition,
  2015.

\bibitem{LRS10}
T.~Leli{\`e}vre, M.~Rousset, and G.~Stoltz.
\newblock {\em {Free Energy Computations}}.
\newblock Imperial College Press, 2010.

\bibitem{LionsCours}
P.-L. Lions.
\newblock Cours au {C}oll\`{e}ge de {F}rance. www.college-de-france.fr.

\bibitem{LMLLY20}
Y.~Lu, C.~Ma, Y.~Lu, J.~Lu, and L.~Ying.
\newblock A {M}ean {F}ield {A}nalysis of {D}eep {R}es{N}et and {B}eyond:
  {T}owards {P}rovably {O}ptimization via {O}verparameterization from {D}epth.
\newblock In H.~D. III and A.~Singh, editors, {\em Proceedings of the 37th
  International Conference on Machine Learning}, volume 119 of {\em Proceedings
  of Machine Learning Research}, pages 6426--6436. PMLR, 13--18 Jul 2020.

\bibitem{LW16}
D.~Luo and J.~Wang.
\newblock {Exponential convergence in $L^p$-Wasserstein distance for diffusion
  processes without uniformly dissipative drift}.
\newblock {\em Mathematische Nachrichten}, 289(14-15):1909--1926, 2016.

\bibitem{MSH02}
J.~C. Mattingly, A.~M. Stuart, and D.~J. Higham.
\newblock {Ergodicity for SDEs and approximations: locally Lipschitz vector
  fields and degenerate noise}.
\newblock {\em Stochastic Processes and their Applications}, 101(2):185--232,
  2002.

\bibitem{mei2018mean}
S.~Mei, A.~Montanari, and P.-M. Nguyen.
\newblock {A mean field view of the landscape of two-layer neural networks}.
\newblock {\em Proceedings of the National Academy of Sciences},
  115(33):E7665--E7671, 2018.

\bibitem{MPP2002}
G.~Metafune, D.~Pallara, and E.~Priola.
\newblock {S}petrum of {O}rnstein-{U}hlenbeck operators in ${L}^p$ spaces with
  respect to invariant measures.
\newblock {\em Journal of Functional Analysis}, 196(1):40--60, 2002.

\bibitem{millet1989integration}
A.~Millet, D.~Nualart, and M.~Sanz.
\newblock {Integration by parts and time reversal for diffusion processes}.
\newblock {\em The Annals of Probability}, 17(1):208--238, 1989.

\bibitem{Mon17}
P.~Monmarch\'e.
\newblock Long-time behaviour and propagation of chaos for mean field kinetic
  particles.
\newblock {\em Stochastic Processes and their Applications}, 127(6):1721--1737,
  June 2017.

\bibitem{Neal11}
R.~M. Neal.
\newblock {\em {MCMC using Hamiltonian dynamics}}.
\newblock {Handbook of Markov chain Monte Carlo}. Boca Raton: CRC Press, 2011.

\bibitem{Nel67}
E.~Nelson.
\newblock {\em {Dynamical theories of Brownian motion}}, volume~2.
\newblock Princeton University Press, 1967.

\bibitem{Pav14}
G.~A. Pavliotis.
\newblock {\em {Stochastic processes and applications: diffusion processes, the
  Fokker-Planck and Langevin equations}}, volume~60.
\newblock Springer, 2014.

\bibitem{RT02}
L.~Rey-Bellet and L.~E. Thomas.
\newblock {Exponential convergence to non-equilibrium stationary states in
  classical statistical mechanics}.
\newblock {\em Comm. Math. Phys.}, 225(2):305--329, 2002.

\bibitem{rotskoff2018neural}
G.~Rotskoff and E.~Vanden-Eijnden.
\newblock {Neural networks as interacting particle systems: Asymptotic
  convexity of the loss landscape and universal scaling of the approximation
  error}.
\newblock {\em arXiv:1805.00915}, 2018.

\bibitem{SS78}
T.~Schneider and E.~Stoll.
\newblock Molecular-dynamics study of a three-dimensional one-component model
  for distortive phase transitions.
\newblock {\em Phys. Rev. B}, 17:1302--1322, February 1978.

\bibitem{Schuh22}
K.~Schuh.
\newblock Global contractivity for {L}angevin dynamics with
  distribution-dependent forces and uniform in time propagation of chaos.
\newblock {\em Preprint arXiv:2206.03082}, 2022.

\bibitem{Sznit89}
A.-S. Sznitman.
\newblock {Topics in propagation of chaos}.
\newblock {\em {\'E}cole d'{\'E}t{\'e} de Probabilit{\'e}s de Saint-Flour XIX,
  Lecture Notes in Math.}, 1464:165--251, 1991.

\bibitem{Talay02}
D.~Talay.
\newblock {Stochastic Hamiltonian systems: exponential convergence to the
  invariant measure, and discretization by the implicit Euler scheme}.
\newblock {\em Markov Process. Related Fields}, 8(2):163--198, 2002.

\bibitem{ustunel2013transformation}
A.~S. {\"U}st{\"u}nel and M.~Zakai.
\newblock {\em {Transformation of measure on Wiener space}}.
\newblock Springer Science \& Business Media, 2013.

\bibitem{Vil07}
C.~Villani.
\newblock {Hypocoercive diffusion operators}.
\newblock {\em Boll. Unione Mat. Ital. Sez. B Artic. Ric. Mat. (8)},
  10(2):257--275, 2007.

\bibitem{Vil09}
C.~Villani.
\newblock {Hypocoercivity}.
\newblock {\em Mem. Amer. Math. Soc.}, 202(950):iv--141, 2009.

\bibitem{SS20}
D.~\v{S}i\v{s}ka and L.~Szpruch.
\newblock {Gradient Flows for Regularized Stochastic Control Problems}.
\newblock {\em Preprint arXiv:2006.05956}, 2020.

\bibitem{Wu01}
L.~Wu.
\newblock {Large and moderate deviations and exponential convergence for
  stochastic damping Hamiltonian systems}.
\newblock {\em Stochastic Processes and their Applications}, 91(2):205--238,
  2001.

\end{thebibliography}

\end{document}